\documentclass[12pt,a4paper]{article}
\usepackage{inputenc}
\usepackage[english]{babel}
 \usepackage{mathrsfs}
\usepackage[T1]{fontenc}
\usepackage{amsmath}
\usepackage{amsfonts}
\usepackage{amssymb}
\usepackage{makeidx}
\usepackage{color}
\usepackage{amsthm}
\usepackage{graphicx}
\usepackage[left=2cm,right=2cm,top=2cm,bottom=2cm]{geometry}
  \usepackage[all]{xy}
\usepackage{hyperref}
  
    \usepackage{epic}
  \usepackage{autograph}
   \usepackage{lscape}
\DeclareMathOperator{\bbbO}{\bold{O}}

\newtheorem{thm}{Theorem}[section]

\newtheorem{lem}[thm]{Lemma}
\newcounter{counter_conj-lem}
\newtheorem{conj-lem}[thm]{Conjecture-Lemma}

\newtheorem{prop}[thm]{Proposition}
\newcounter{counter_conj-prop}
\newtheorem{conj-prop}[thm]{Conjecture-Proposition}

\newcounter{counter_conj-thm}
\newtheorem{conj-thm}[thm]{Conjecture-Theorem}

\newtheorem{cor}[thm]{Corollary}
\newtheorem{defi}[thm]{Definition}

\theoremstyle{remark}

\newtheorem{rem}[thm]{Remark}

\DeclareMathOperator{\pinf}{\boldsymbol{\infty}}

\DeclareMathOperator{\C}{\mathbb{C}}

\DeclareMathOperator{\Fs}{(F^{\sharp})}

\DeclareMathOperator{\Lf}{\mathbb{L}_f}

\DeclareMathOperator{\Lz}{\mathbb{L}_0}

\DeclareMathOperator{\Linf}{\mathbb{L}_{\infty}}
\DeclareMathOperator{\dyndeg}{\delta}

\DeclareMathOperator{\N}{\mathbb{N}}

\DeclareMathOperator{\NN}{\mathcal{N}}
\DeclareMathOperator{\ra}{\rightarrow}
\DeclareMathOperator{\si}{\sigma}
\DeclareMathOperator{\Jul}{\mathcal{J}}

\DeclareMathOperator{\cL}{\mathcal{L}}

\DeclareMathOperator{\Lo}{L^1}
\DeclareMathOperator{\Loloc}{L^1_{loc}}

\DeclareMathOperator{\Jac}{Jac}

\DeclareMathOperator{\Hf}{H_f}

\DeclareMathOperator{\Lc}{\mathbb{L}_c}
\DeclareMathOperator{\Li}{\mathbb{L}_i}
\DeclareMathOperator{\Oinf}{\bold{O}_\infty}

\DeclareMathOperator{\bIz}{\mathbb{I}_0}

\DeclareMathOperator{\bIinf}{\mathbb{I}_{\infty}}

\DeclareMathOperator{\bfz}{\beta_{0}}

\DeclareMathOperator{\bfzp}{\beta^{\prime}_{0}}

\DeclareMathOperator{\binf}{\beta_{\infty}}

\DeclareMathOperator{\binfp}{\beta^{\prime}_{\infty}}

\DeclareMathOperator{\Efz}{\text{E}_{0}}

\DeclareMathOperator{\Efzp}{\text{E}^{\prime}_{0}}

\DeclareMathOperator{\Einf}{\text{E}_{\infty}}

\DeclareMathOperator{\Einfp}{\text{E}^{\prime}_{\infty}}

\DeclareMathOperator{\dtop}{d_{top}}

\DeclareMathOperator{\bP}{\mathbb{P}}

\newcommand{\LL}{{\mathcal{L}}}

\newcommand{\ssk}{\smallskip}
\newcommand{\nin}{\noindent}
\DeclareMathOperator{\sm}{\setminus}
\DeclareMathOperator{\TT}{\mathbb{T}}

\DeclareMathOperator{\divi}{div}

\DeclareMathOperator{\R}{\mathbb{R}}

\DeclareMathOperator{\vol}{vol}

\DeclareMathOperator{\Pg}{\mathbb{P}}

\DeclareMathOperator{\bH}{\mathbb{H}}

\DeclareMathOperator{\cW}{\mathcal{W}}

\title{Spectrum of the basilica group and the associated real and complex 2D dynamics.}
\author{E. Bedford, N.-B. Dang, R. Grigorchuk, M. Lyubich}

\begin{document}

\maketitle

\begin{abstract}
In this paper, we study  the iteration  over the real and complex loci of a specific rational map, found by Grigorchuk and Zuk. 
This two dimensional map arises as the renormalization map which yields the density of states measure associated with the  iterated monodromy group of the Basilica Julia set. It  admits two natural complex currents, and we construct  their real analogs  and show that they are laminar. 
Finally, this yields some consequences on the properties of the density of states measures and its support.      
\end{abstract}

\tableofcontents

\overfullrule0pt

\centerline{\bf Introduction}
\bigskip

{ 

The iteration of rational functions has been actively pursued for over a century, and
now we have a deep understanding of rational dynamics. Although we would like to extend
our understanding to higher dimension, it has been impeded because many of the tools of
holomorphic dynamics on the Riemann sphere do not apply in higher dimension. Further,
starting in dimension dimension two, a rational map need not be everywhere holomorphic
since it can have points of indeterminacy. These are points where the rational map cannot
be defined as a continuous map, and a point of indeterminacy is "blown up" to an algebraic
curve of positive dimension. Compared to the one-dimensional case, the dynamics of
rational mappings in dimension two is in its early stages of development.

%
 } 
{ 
Among all the possible choice of rational maps to study, we focus on a particular one
because of its intimate connection to a general problem in spectral theory: the computation
of the spectrum of the Laplacian on a recursively defined sequence of finite graphs. These
graphs are obtained from a self-similar group using a general procedure (see e.g. [Nek]).
When this group is the iterated monodromy group of a post-critically finite polynomial,
these graphs converge to the Julia set of the corresponding polynomial. In our case, the group is called the Basilica group and the
sequence of graphs converge to the Basilica fractal (the Julia set of $z^2-1$). This group, is a particular self-similar group which 
was introduced in \cite{grigorchuk_zuk_torsion} as an example of a weakly branch group but not a branch group. It is also the first true example of a non-elementary amenable group with exponential growth \cite{bartholdi_virag} (see also \cite{kaimanovich_trick} for further developments)
 and its "spectrum" was conjectured to have unusual spectral and algebraic properties. 
}
{The spectrum of this group is obtained as follows. Take the sequence of finite graphs converging to the fractal and consider the characteristic polynomials $P_n(x,y)$ associated to some weighted adjacency matrix on the $n$-th graph (i.e putting a weight $x$ on the edges labeled by a group element $a$ instead of $1$) and the measures defined by their zero sets:
 }
\begin{equation} \label{eq_approx_density}
\omega_{h,n}:= {1\over 2^n} \sum_{P_n(\lambda,h)=0} \delta_\lambda   
\end{equation}
If we view $P_n$ in terms of its current of integration,  then the measures \eqref{eq_approx_density} are given by the intersection products
\begin{equation}
\omega_{h,n} = {1\over 2^n} [\{ P_n(x,y)=0\}]\wedge [y=h]
\end{equation}
%
so the current of integration {on $P_n=0$} gives an unfolding of the measures.
In the case of the Basilica group, it was shown by Grigorchuk and \v Zuk \cite{grigorchuk_zuk_torsion} that the polynomials are given by { a recursive formula}
\begin{equation} \label{eq_recur}
P_{n+1}(x,y) = y^{2^n} P_n(F(x,y))
\end{equation}
starting with $P_0=2-x-2y$, and where $F$ is the rational map
\begin{equation}
F(x,y) = \left( -2 +  {x(x-2)\over y^2} , {x-2\over y^2}        \right).
\end{equation}
{The above relation \eqref{eq_recur} shows that the map $F$ relates the eigenvalues of the weighted adjacency matrix of the $n$-th graph with the $(n+1)$-st, which "contains" it. Thus, this map can be thought of a renormalization map, as it relates the eigenvalues of the discrete Laplacian between different scales \cite{dang_grigorchuk_lyubich}.} 
The principal object  of interest in this approach is the {\it density of states measure} or KNS spectral measure
\begin{equation}\label{eq_density}
\omega_h:=\lim_{n\to\infty} \omega_{h,n} .
\end{equation}
{ It was shown in \cite{teplyaev_al_basilica} that the support of $\omega_1$ is a Cantor set and for each point of the support, the authors give a detailed study of the  generalized eigenfunctions. }
The { polynomials} $P_n$ arise by an iterative procedure in \eqref{eq_recur}, so we are led to focus on dynamical properties of  $F$.  For this we treat $F$ as a rational map defined over the complex numbers; and then we restrict to the real domain.  We recapture many of the results on the density of states measures  obtained  in \cite{teplyaev_al_basilica}, and we show how they can  also appear  in the context of rational dynamics.  
{ We restate some results taken  from theorems below in terms of the density of states measure. 

\bigskip

\noindent $\bullet$ \textbf{Main Theorem}. For any $h \in \R$, the measure $\omega_h$ is an atomic measure, supported on a Cantor set with positive Lebesgue measure. Furthermore, there exists a finite alphabet $A$ and a family (continuous in $h$) of measurable map $\varphi_h \colon A^{\mathbb{N}} \to \R$ such that for any smooth real function with compact support $\psi$: 
\begin{equation}
\langle \omega_{h,n} , \psi\rangle = \langle \omega_h ,\psi \rangle + \dfrac{\delta^n}{2^n} \langle {\varphi_h}_* \mu, \psi \rangle + o\left( \dfrac{\delta^n}{2^n}\right),
\end{equation}
where $\mu$ is a Bernoulli type measure associated with a particular Markov matrix and where $\delta \sim 1.8$ is the largest root of $t^3 -t^2 -2t +1    $.

\bigskip 

}
Earlier a similar program was carried out for three other groups in \cite{dang_grigorchuk_lyubich}.  The analysis in those cases was aided greatly by the fact that the mappings have invariant algebraic fibrations, and precise formulas are given for the density of states measures in these cases.  By Theorem \ref{thm_value_dyndeg} {and the product formula of Dinh-Nguyen-Truong \cite{dinh_nguyen_truong}}, our map $F$ does not have an invariant algebraic fibration, yet we find an invariant lamination inside $\Bbb R^2$, and this is the basis of much of our analysis.  Correspondingly, in the Basilica case, our description of the density of states measure is less precise.

By $F_{\Bbb P^2}$ we denote the extension of $F$ to projective space ${\Bbb P}^2$, where it takes the form
$$F_{\Bbb P^2}[x:y:z] = [x^2-2xz - y^2: xz-2z^2:y^2]$$  
as a map with (algebraic) degree equal to 2.  We note that ${\rm deg}(F^4)=14\ne {\rm deg}(F)^4$, so we are led to define the
 {\it dynamical degree} of $F$  the exponential rate of growth of the degrees of the iterates of $F$:
$${\delta= \rm ddeg}(F)= \lim_{n\to\infty} \left ( {\rm deg}(F^n)    \right)^{1/n}  \eqno(6)  $$

The inequality ${\rm ddeg}(F)<{\rm deg}(F)$ is due to the fact that there are collapsed curves that are eventually mapped inside the indeterminacy locus.  To determine $\delta$, we introduce a new complex manifold $\pi:X\to {\Bbb P}^2$, which is a blow up at four specially chosen points.  This induces a rational map
$$F_X:=\pi^{-1}\circ F_{\Bbb P^2} \circ\pi:X\ \dasharrow  X$$
With our special choice of $X$,  
$(F_X^n)^*=(F_X^*)^n$ on  $H^{1,1}(X;\Bbb R) $, and so the passage to cohomology is functiorial.   
The dynamical degrees of $F_{\Bbb P^2}$ and $F_X$ are the same and equal to the largest eigenvalue of $F_X^*$.  We will show that $\delta\sim 1.8$ is a root of $t^3 -t^2 -2t +1    $.

Parallel to our interest in the density of states measures is our interest in the dynamics of $F$ itself.  Although a rational map is holomorphic at most points, it fails to be continuous at any point of indeterminacy, and dealing with such points is a challenge.  $F$ cannot be made regular via a birational conjugacy, and passing to the improved model $F_X$ does not  remove indeterminacy, but it does add some understanding of the point of indeterminacy.   Let $E_0:=\pi^{-1}(2,0) \cong {\Bbb P^1}$ denote the exceptional (blowup) fiber over the indeterminate point.  The map $F_X^2$ then leaves $E_0$ invariant, and $F_X^2|_{E_0}$ is conjugate to the map $z\mapsto z^2-12 $ on ${\Bbb P}^1$.  We show in Theorem \ref{thm_repelling_measure} that $E_0$ admits a repelling measure.


Let us return to $F:= F_{\Bbb P^2}$.  From \eqref{eq_approx_density} and \eqref{eq_density} we are led to look for the existence of a limit current
\begin{equation}
S := \lim_{n\to\infty} {1\over 2^n} [P_n(x,y) =0]   = \lim_{n\to \infty} {1\over 2^n} (F^*)^n [L_f]
\end{equation}
where we write $L_f:=\{P_0=0\}$.  Such a current must satisfy $F^*(S) = 2S$.  
We may follow the arguments of Sabot \cite{sabot_spectral}  and apply  two general results of Sibony \cite{sibony_1999}.

\medskip 

\noindent $\bullet$ \textbf{ Theorem C.1 (Sibony).}  If  $F:{\Bbb P}^k\dasharrow {\Bbb P}^k$ is a dominant rational map with $1<\delta < {\rm deg}(F)$, then there is a unique positive closed current  $S$ on $\Bbb P^k$ with total mass 1 and $F^*S={\rm deg}(F)S$ such that if $V\subset\Bbb P^k$ is any variety, there exists $c>0$ such that

\begin{equation} \label{eq_sib}
\lim _{n\to\infty} {1\over {\rm deg}(F)^n} (F^*)^n[V] = cS 
\end{equation}


\medskip 

We will refer to $S$ as the {\it Sibony current}. Now we apply this theorem to our specific case.  Since the curve $L_c:=\{x=2\}$ is the only collapsed curve and  $(2,0)$ is the only point of indeterminacy, we find the constant $c=1$ in \eqref{eq_sib}:

\medskip

\noindent $\bullet$ \textbf{ Theorem C.2 (Sibony)}.  Since  $1<\delta<{\rm deg}(F)$, and $L_c=\{x=2\}$ is the only collapsed curve, 
and if $L$ denotes any line $V$ containing the point $(2,0)$, then
$$S = \lim_{n\to\infty}   {1\over 2^n}(F^*)^n[L]      \eqno(9)  $$
$$=\sum_{k=0}^\infty  c_k [F ^{-k} (L_c) ]   \eqno(10) $$

\medskip

Since $(2,0)\in L_f:=\{P_0=0\}$ we may apply this Theorem to the varieties $(F^*)^n(L) = \{P_n=0\}$ and conclude  that the slices give 
$$\omega_h = S\wedge [y=h]    \eqno(11)$$
for all $h$.  Thus we could also refer to the Sibony current as the {\it density of states current}.
Further, by Theorem C.2,  $S$ is a {\it purely atomic} current, by which we mean that it is a sum of weighted currents of integration.

The maps $F^*$ and $F_X^*$ act on different spaces, but they are related.   Let $\hat{\cal L}_c$ denote the set of all the varieties $W\subset {\Bbb P}^2$ which are irreducible components of $F^{-n}(L_c)$ for some $n\ge0$.  
If $V\subset {\Bbb P}^2$ is a variety, we may consider its inclusion (strict transform) $\overline V \subset X$.  Proposition \ref{prop_error_term} shows that we may compare the pullbacks of $L_f$ under $F^*$ and $F_X^*$ as follows:

$$\pi_*(F_X^{*n}([\overline{L_f}]) = F^{* n}([L_f]) \ \ {\rm mod} \ \ \hat{\cal L}_c  \eqno (12)$$
where  the right hand side of (12) is a sum of currents of integration, with all elements of $\hat{\cal L}_c$ set to zero.  Another version of this is:

\medskip 

\noindent $\bullet$ \textbf{ Theorem C.3.  }
$${1\over 2^n} [F^{-n}L_f] = {1\over 2^n} \sum _{k=0}^n c_{k,n}[F^{-n} L_c] + {\delta^n\over 2^n} E_n\eqno(13)$$
where the coefficients $c_{k,n}$ increase  to the  $c_k$ in (10) as $n\to\infty$, and $E_n=\delta^{-n}\pi_* F_X^{*n}[ \overline{ L_f } ] $ has total mass $\sim1$. 

\medskip 

\noindent $\bullet$ \textbf{ Theorem C.4 (Guedj).}   If $g:X\dasharrow X$ is  a dominant, rational map with ${\rm ddeg}(g)>1$, and $(g^*)^n=(g^n)^*$ on $H^{1,1}(X)$, then there is a unique positive, closed current $T$ of total mass 1 on $X$ such that $g^*(T) = \delta \cdot T$.

\medskip 

The current $T$ is known as the {\it Green current}.  A current  $T$ is said to have an {\it atom} if there is a variety $V$ such that $\epsilon [V]\le T$ for some $\epsilon>0$.

\medskip 

\noindent $\bullet$ \textbf{ Theorem C.5.}  The Green current $T$ is atom-free, and $T=\lim_{n\to\infty}\delta^{-n} F_X^{*n}[\overline{L_f}] $.

\medskip 
We consider the right hand side of (13) as an approximation to the Sibony current $S$.  The first term  is a monotone increasing sum of atoms; and the second term involves a quantity that is converging to something more ``diffuse" than $S$, and which is multiplied by a factor $\sim (\delta/2)^{n}$.  

\bigskip
Now we restrict our discussion to the real domain.
The lines $L_0:=\{y=0\}$ and $L_\infty: =\{z=0\}$ make a forward invariant 2-cycle $L_0\leftrightarrow L_\infty$.   Inside this, we have $I_0 = \{[x:0:1]:-2\le x\le 2\}=[-2,2]\subset \Bbb R^2$ and $I_\infty:=\{ [1:y:0]: -2\le y\le 2\}$, and $I_0\leftrightarrow I_\infty$.  
We refer to a line as {\it vertical} if it intersects both $I_0$ and $I_\infty$.  The endpoints of $I_0$ and $I_\infty$  are the centers of blowup that were used to define the blowup $\pi:X\to {\Bbb P}^2$ which was used above.  We now use the four pencils of vertical lines based at the these four endpoints
 to define a vertical cone field ${\cal C}_v$.  We write the complementary (horizontal) cones as ${\cal C}_h$ (see Figure \ref{fig_schematic_cone_field}).

\begin{figure}[h!]
\centering
\includegraphics[scale=0.3]{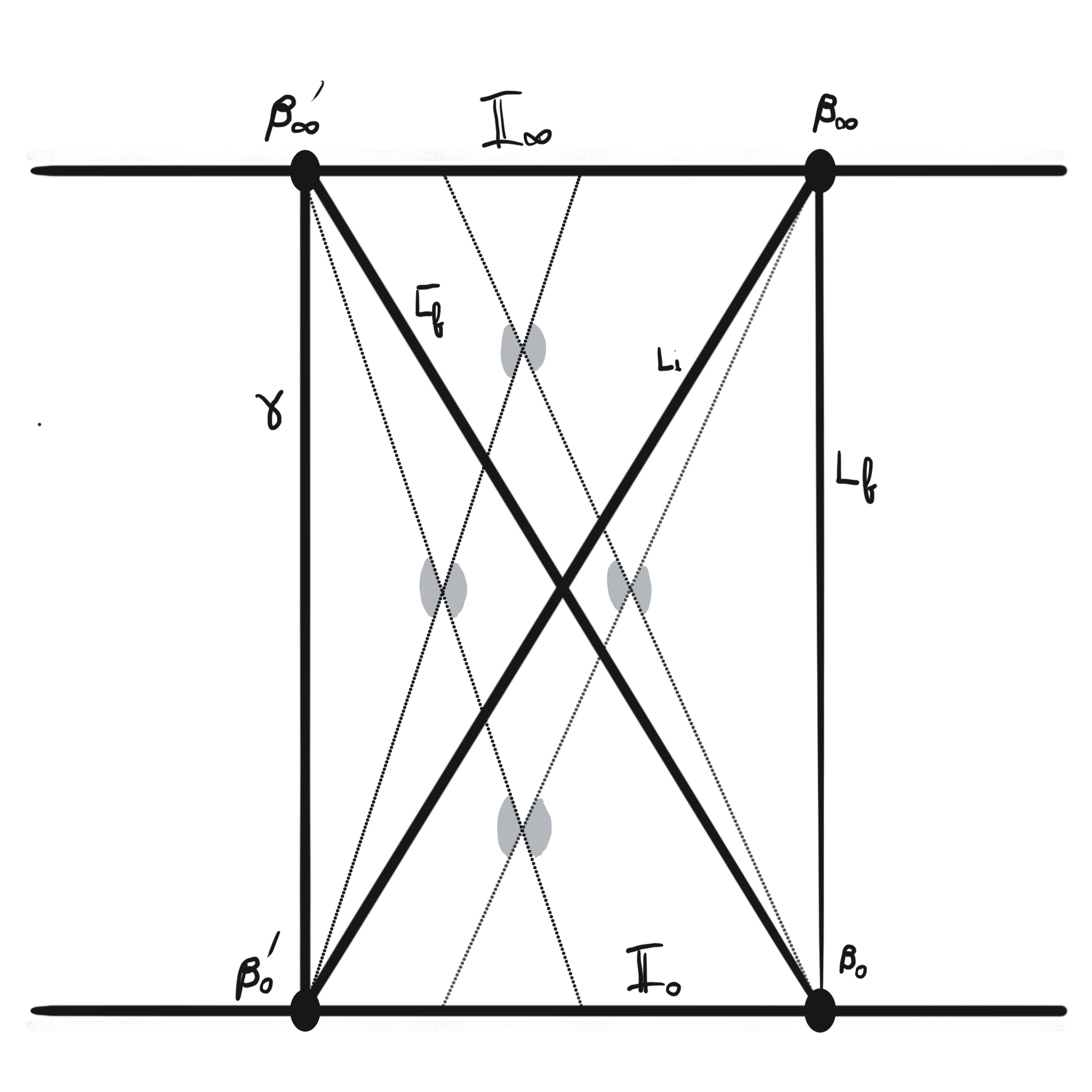} 
\caption{\label{fig_schematic_cone_field} Schematic picture of the horizontal cone field $\mathcal{C}_h$ in shaded gray, the points $\bfz,\bfzp,\binf,\binfp$ are the dots. The boundary curves are the invariant line $\Lf$, its reflection $\bar \Lf$, and the component $\gamma$ of $F^{-1} \bar \Lf$. The lines in the pencil are the dotted thin lines.}
\end{figure} 
 
    The following summarizes Theorem \ref{thm_invariant_cone_field}:
 
\smallskip 
 
\noindent $\bullet$ \textbf{ Theorem R.1.}  The cone field ${\cal C}_h$ is invariant under $dF$, and ${\cal C}_v$ is invariant under $dF^{-1}$.

\smallskip

During our discussion of the real dynamics, we write $f^{-n}(L_c)$ for $f^{-n}(L_c)\cap \Bbb R\Bbb P^2$, etc.  Let $\gamma$ denote a component of $f^{-n}(L_c)\cap \Bbb R^2 =f^{-n}(L_c)-L_\infty$, we may take $\gamma$ to be invariant by $(x,y) \mapsto (x,-y)$. It follows from Theorem R.1  that $\gamma=\{x=\varphi(y):y\in(0,\infty)\}$ is a graph over the $y$-axis.   We let ${\cal L}$ denote the limits of all such graphs.   

Let ${\cal S}_0$ denote the (singular) points of $I_0$, i.e., points that are mapped by $f^n$ to the point of indeterminacy $(2,0)\in I_0$ for some $n\ge0$.  We let ${\cal R}_0$ denote the (regular) points of $I_0$, meaning that their forward orbits are always well-defined.

\smallskip

\noindent $\bullet$ \textbf{ Theorem R.2.}  ${\cal L}- L_0$ is a lamination of $\Bbb R^2-L_0$.  If $x\in{\cal R}_0$, then there is a unique $\gamma\in{\cal L}$ which contains $x$.   If $x\in {\cal S}_0$, then the $\gamma\in{\cal L}$ which contain $x$ form a Cantor bouquet 

\smallskip

We note that Cantor bouquets were observed earlier by Yamagishi \cite{yamagishi} as a manifestation of indeterminate behavior similar to $F$.
Our main tool in  this Theorem will be the Kobayashi pseudo-metric.  Although the infinitesimal Kobayashi metric has a degenerate direction at every point, we may restrict it to the cones ${\cal C}_h$ at points of ${\cal L}$; and the expansion in this direction is sufficient to prove laminarity.

We let $|{\cal L}|$ denote the union of the  graphs in ${\cal L}$.   The point $\infty = [1:0:0]$ is a superattracting fixed point, and we let ${\cal B}(\infty)$ denote the real points in its basin of attraction.  The immediate basin of attraction is the region bounded by $L_f$, $\overline L_f$ and the hyperbola $\gamma$ as in Figure \ref{basilica fig}.    We have the following dichotomy between the basin and the support of the lamination:

\noindent $\bullet$ \textbf{ Theorem R.3.}  $|{\cal L}|=  \Bbb R^2-{\cal B}(\infty)$.   The gaps in ${\cal L}$  are the connected components of ${\cal B}(\infty)$ which are not the immediate basin.

The set $L_0$ consists of branch points for $F^2$, and the restriction of $F^2$ to $L_0$ is the Chebyshev map $x\mapsto x^2-2$.  Its Julia set is $I_0$, and we  define the {\it basin} of $I_0$ to be
$${\cal B}(I_0): = \{(x,y)\in{\Bbb R^2}: f^n(x,y) {\rm \ is\ defined\ for\ all\  }n\ge0, {\rm \ and\ }  {\rm dist}(f^{2n}(x,y), I_0)\to 0  \}$$

\smallskip

\noindent $\bullet$ \textbf{ Theorem R.4.}  The set  ${\cal B}(I_0)$ has positive planar area; it is disjoint from ${\cal B}(\infty)$,  and thus  $|{\cal L}|$ has positive area.

\smallskip

By Theorem R.1, we may orient each graph of $f^{-n}(L_c)$  in the direction of increasing $y$ and let $[f^{-n}(L_c)]$ denote the associated current of integration.  By $|\gamma|$ we mean the point set of $\gamma$, so  $|\gamma|$ is the support of the current $[\gamma]$.  By ${\bf M}_{\cal L}$ we denote the laminar currents generated by ${\cal L}$ in the following sense:  there is a measure $\eta$ on some set $\tau$ and for each $t\in\tau$ there is a $\gamma_t\in{\cal L}$, and the laminar current $$\int_{t\in\tau} [\gamma_t]\, \eta(t)\in {\bf M}_{\cal L}$$ 

We may define a ``real slice'' of the Sibony current as
$$S_{\Bbb R} : = \sum_{k=0}^\infty c_k [f^{-n}(L_c)]\in {\bf M}_{\cal L}$$
In fact, $S_{\Bbb R}$ is a real laminar current, and by the following Theorem, the transversal measures are the spectral measures $\omega_h$, so we could also call it the real spectral current.

\smallskip

\noindent $\bullet$ \textbf{ Theorem R.5.}  If $h\in{\Bbb R}$, then the real slicing agrees with  the complex intersection product:
$$[\Bbb R\times \{h\} ] \wedge_{\Bbb R} S_{\Bbb R} = S\wedge_{\Bbb C} [{\Bbb C}\times \{h\} ]$$

\smallskip

\noindent The `$\wedge_{\Bbb R}$' on the left hand side denotes the oriented intersection of real curves, and the `$\wedge_{\Bbb C}$' on the right is the complex wedge product.

\smallskip

It is evident that $|{\cal L}|$ is the support of $S_{\Bbb R}$.  Recall that the {\it local dimension} of a measure $\nu$ at a point $p$ is defined as a limit: 
$${\rm dim}_{\rm loc}\nu(p):=\lim_{r\to0}\nu(B(p,r))/\log r$$ 
where $B(p,r)$ denotes the ball of radius $r$ about $p$.  Now we summarize the  properties  we have obtained for the density of states measures $\omega_h$:

\smallskip

\noindent $\bullet$ \textbf{ Theorem R.6.}  For $h\in\Bbb R$, the slice ${\cal L}\cap \{y=h\}$ is a Cantor set with positive Lebesgue length, and this Cantor set is the support of $\omega_h$.  If $U$ is a gap in this Cantor set, then one of the endpoints of $U$ is mapped by some $f^n, n\ge0$ to the collapsed line $L_c$, and the other endpoint $p\in\partial U$ is mapped by some $f^m, m\ge0$ to the fixed line $L_f$.  The local dimension of $\omega_h$ at $p$ is $2/3$.

We divide part of ${\cal L}$ into subsets ${\cal L}_a$, with $a\in\Bbb A:= \{I, II, III_-, III_0, III_+, IV\}$,  as in Figure  \ref{fig_schematic_lamination}. 
\begin{figure}[h!]
\centering
\includegraphics[scale=0.3]{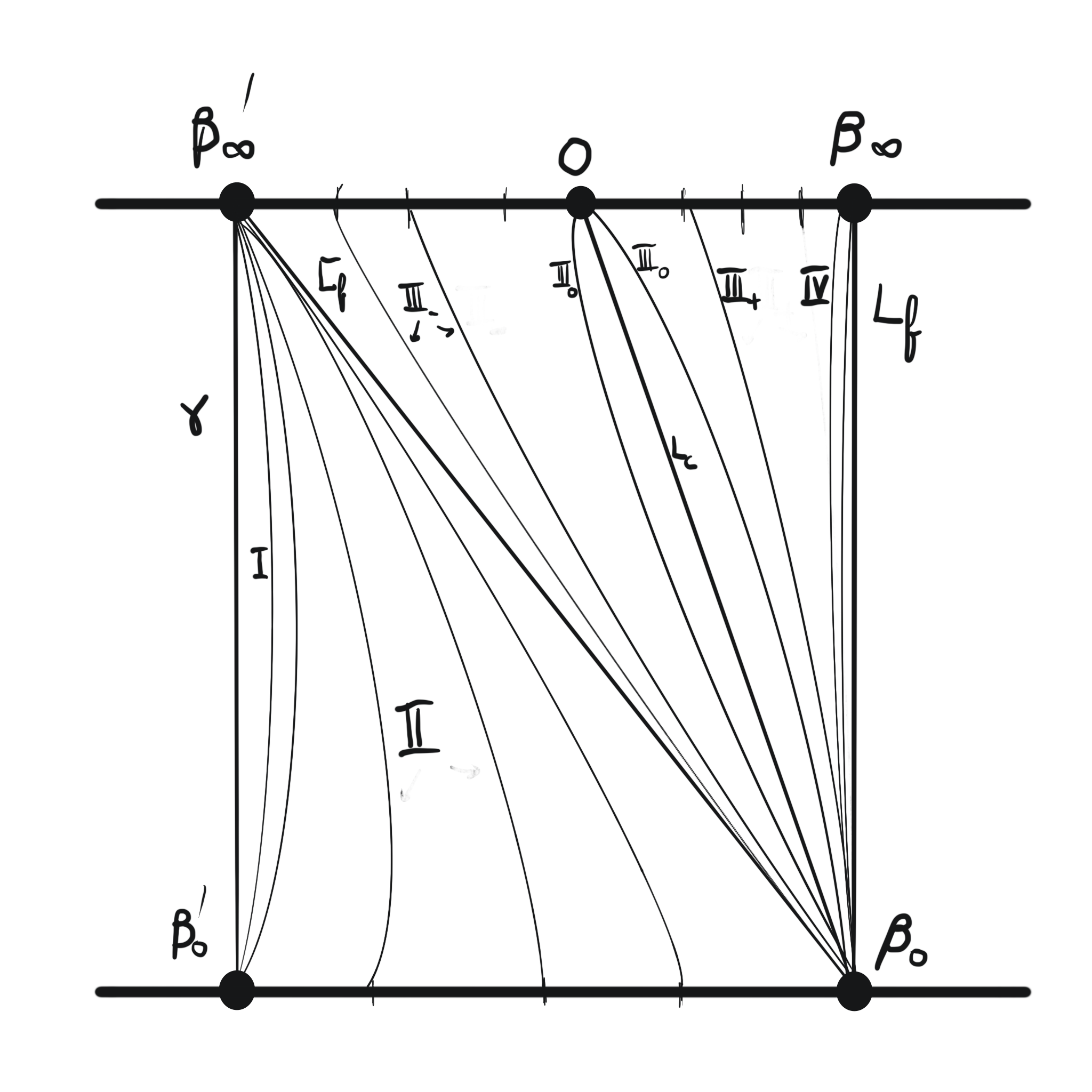} 
\caption{\label{fig_schematic_lamination} Schematic picture of the main leaves of the lamination. The thick lines are the boundary curves $\gamma$, $\bar  \Lf, \Lf$ and the curves of type $I,II,III_-, III_0, III_+ , IV$ are represented.}
\end{figure} 
 This decomposition is  $f^{-1}$-invariant, since for any $\gamma\in{\cal L}_a$, $f^{-1}|\gamma|$ will be the union of one or two curves in ${\cal L}$, and we let $M$ in \eqref{eq_matrix} be the matrix that records this.  Let $\Sigma$ denote the space of bilateral sequences  on the symbol set (alphabet) $\Bbb A$, and let $\Sigma_M$  denote the subshift defined by $M$.  We let $\Sigma^*_M$ denote the sequences $(\underline s)\in\Sigma_M$ such that the symbol  `$III_0$' appears infinitely often.  The largest eigenvalue of $M$ is $\delta$, and we let $\mu_M$ denote the shift-invariant Bernoulli measure on $\Sigma_M$ with entropy $\log \delta$.  We note that $\Sigma_M^*$ has full measure for $\mu_M$.

\medskip

\noindent $\bullet$ \textbf{ Theorem R.7.}  There is a coding map
$\Sigma_M^*\ni (\underline s) = (s_n)_{n\in\Bbb Z} \mapsto \gamma_{\underline s}\in{\cal L}_{s_0}$ with the property that $f(\gamma_{\underline s})  \subset \gamma_{\sigma(\underline s)}$.

\medskip

This coding map allows us to make a correspondence between measures on $\Sigma^*_M$ and currents in ${\bf M}_{\cal L}$.
We define a signed pullback map $f^{\sharp}:{\bf M}_{\cal L}\to {\bf M}_{\cal L}$, which has the property that if $f^{-1}(|\gamma|) = |\gamma'|\cup |\gamma''|$, then $f^{\sharp}[\gamma] = [\gamma'] + [\gamma'']$.  Thus $f^\sharp$ defines a map on the measures on $\Sigma^*_M$, which is a transfer operator for the shift.  Let us define
$$T_{\Bbb R}:=\int_{\underline s\in \Sigma_M^*} [\gamma_{\underline s}]\, \mu_M(\underline s)$$
The following shows that we may think of $T_{\Bbb R}$ as a sort of ``real slice'' of the Green current, { as we observe a real equidistribution phenomenon}.

\medskip

\noindent $\bullet$ \textbf{ Theorem R.8.}  $f^\sharp(T_{\Bbb R}) =\delta\, T_{\Bbb R}$, and if $\gamma\in\Sigma_M^*$, then
$$T_{\Bbb R}= \lim_{n\to\infty}\delta^{-n} f^{\sharp n}[\gamma]$$
If $h\in \Bbb R$, $h\ne0$, then
the complex slice  $T\wedge (\Bbb C\times \{y = h\})$ is the same as the real intersection product $(\Bbb R\times\{h\})  \wedge T$, and via  the coding map this is $\mu_M$.

\medskip

Unfortunately, the coding map is not H\"older continuous {with respect to the left shift}.  But it is still H\"older continuous with respect to a Poincar\'e return map, which is a consequence of the fact that $F$ is expanding for the first return to a set where the Kobayashi pseudo-metric is nondegenerate.  These H\"older conditions are exactly the hypotheses required in the Decay of Correlations results of Sarig and Gouezel \cite{sarig,gouezel}, {which imply the real equidistribution result}.

\begin{figure}[h!]
\centering
\includegraphics[scale=0.2]{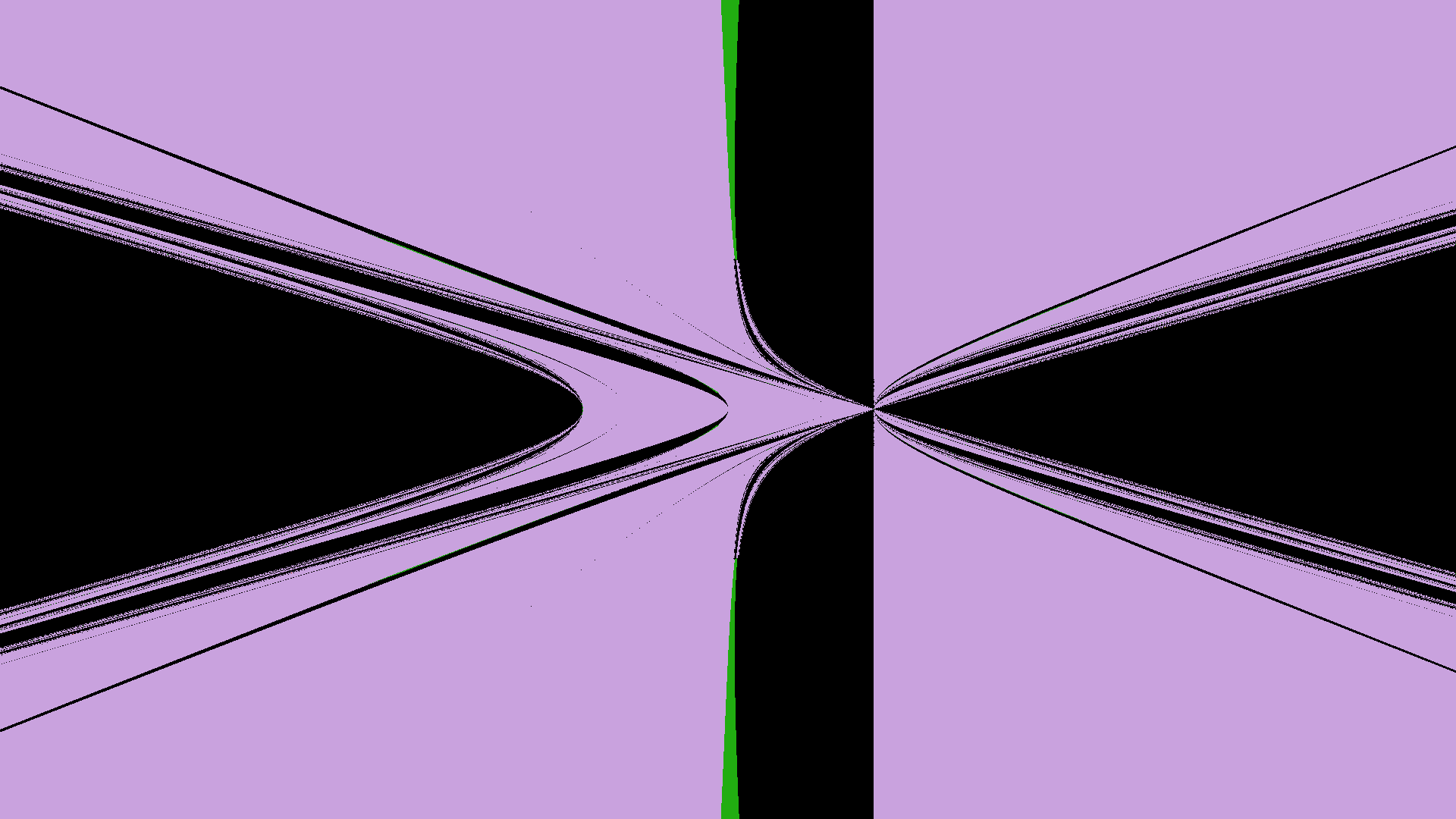}
\caption{ \label{basilica fig} Basins of two measure-theoretic attractors for the basilica renormalization:
In black, the basin of the point $\pinf$ at infinity, in violet are other points. The green area consist of computational errors.}
\end{figure}

\bigskip 

\noindent \textbf{Structure of the paper}: This paper is divided in $9$ sections, with $2$ appendices. In \S~\ref{section_notations}, we fix some notations and we start the study of the complex mapping in \S~\ref{section_general}. We then introduce and describe the Sibony and Green currents in \S~\ref{section_currents} and establish the equidistribution of complex curves. We focus on the real map starting from  \S~\ref{section_map_R} where we construct an invariant cone field, and a prelamination structure $\mathcal{L}$.  
In \S~\ref{section_chebyshev}, we describe the behavior near one of the attractor and show that $\mathcal{L}$ has a Cantor structure. This relies on the proof that $\mathcal{L}$ is a lamination away from the horizontal axis. 
Next, we define the real analog of the Green current in \S~\ref{section_real_green} and show an equidistribution property toward it. Similarly, we define the real Sibony current in \S~\ref{section_real_sibony}. The slicing of all these currents is discussed in \S~\ref{section_slicing} and finally we describe the spectral measures using all the previous result in \S~\ref{section_application}. The first  appendix concerns the relationship between this map and the iterated monodromy group of the Basilica fractal and the second is devoted to some local computations.

\bigskip 

\noindent \textbf{Aknowledgements}: The authors would like to thank Alexander Teplyaev, Roland Roeder, N-V. Dang for their encouragements, various comments, questions raised for the past 6 years concerning this project. 
{One crucial ingredient in the proof of Theorem R.8 is the use of theorem of Gou\"ezel on the decay of correlations, which adapts the Young tower construction together with a refinement of Sarig's application of the renewal equation, and it turned out to be among the suggested reference by  ChatGPT 5.2 to the second author, who asked where precisely in the litterature the treatment of non-uniformly expanding maps have been treated. Apart from this bibliographic reference suggestion, the authors have not used any AI in any shape or form. 
 }
The third author is supported by the Travel Support for Mathematicians grant MP-TSM-00002045 from the Simons Foundation.
This project is funded by ANR-24-CE40-1163, ANR-24-CE40-6184 and ANR-25-CE40-0070.


\section{Notations}
\label{section_notations}

The rest of the paper is devoted to the study of the map $F$ and its behavior under iteration. To that end, we
 will use the following conventions. 
The affine coordinates $(x,y)\in \mathbb{C}^2$ is the standard affine chart for the embedding $(x,y) \in \mathbb{C}^2 \mapsto [x:y:1] \in \mathbb{P}^2$. Moreover, the homogeneous coordinates on $\mathbb{P}^2$ will be of the form $[x: y : z]$.

Recall that the Chebyshev map $z \mapsto 2 z^2 -2$ has a fixed point $z = 2$ called the $\beta$-fixed point and this point has a preimage $z =-2$, the symmetric to the $\beta$-fixed point. 

\smallskip 

\begin{enumerate}
\item[$\bullet$] $\bbbO = [0:0:1]$ is the origin.
\item[$\bullet$]  $\Oinf = [0:1:0]$. 
\item[$\bullet$] $\pinf = [1:0:0]$ is the attractive fixed point at infinity. 
\item[$\bullet$] $\Lz = \{ \mu =0\}$ is the horizontal axis. 
\item[$\bullet$] $\bfz = [2:0:1] = (2,0)$ is the Chebyshev $\beta$-fixed point on the horizontal line $\Lz$. 
\item[$\bullet$] $\bfzp =[-2:0:1] = (-2,0)$ is the symmetric Chebyshev $\beta$-fixed point on the horizontal line $\Lz$.  
\item[$\bullet$] $\bIz = [-2,2]\times \{ 0\}$ is the Chebyshev interval on the horizontal line $\Lz$. 
\item[$\bullet$] $\bIinf = \{ [\alpha : 1 : 0] | \alpha \in [-2,2] \}$.
\item[$\bullet$] $\Linf= \{ w =0 \}$ the line at infinity in $\mathbb{P}^2$.  
\item[$\bullet$] $\binf = [2:1:0]$ is the Chebyshev $\beta$-fixed point on the line at infinity $\Linf$. 
\item[$\bullet$]  $\binfp = [-2:1:0]$ is the symmetric Chebyshev $\beta$-fixed point on the line at infinity $\Linf$.  
\item[$\bullet$] $\bIinf = \{ [u: 1 :0] \in \mathbb{P}^2 | u\in [-2,2] \}$ is the Chebyshev interval on the line at infinity $\Linf$. 

\item[$\bullet$] $\Lc= \{ x - 2 z =0  \}$ is the collapsed line. 
\item[$\bullet$] $I(F) = \{ \bfz \}$ the indeterminacy set of $F$. 

\item[$\bullet$] $ \Li = \{ x - 2 y + 2 z =0  \}$ is the indeterminate line.  
\item[$\bullet$] $\Li^+ = \{ (x, y) \in \R^2 | x -2 y + 2  \leqslant 0  \} \subset \mathbb{R}^2$ is the half space above the line $\Li$.   
\item[$\bullet$] $\sigma \colon (x, y) \in\C^2 \mapsto (x, -y) \in \C^2$ is the reflection with respect to the horizontal axis.
\item[$\bullet$] $\bH := \{-x^2 + 4 y^2 + 2 xz = 0  \} = F^{-1}(\Lc)$. 
\item[$\bullet$] $\Lf := \{ x -2z - 2y =0 \}$ the line of fixed points.
\item[$\bullet$] $\bar\Lf := \{ x -2z + 2y =0 \}$ the reflection of $\Lf$.
\item[$\bullet$] $\Hf := F^{-1}(\bar \Lf)$
\item[$\bullet$] $X$ blow-up of $\mathbb{P}^2$ at $\bfz, \bfzp, \binf, \binfp$. 
\item[$\bullet$] $\pi_X \colon X \to \bP^2$ the blow-down map. 
\item[$\bullet$] $\Efz \subset X$ exceptional divisor over $\bfz$.
\item[$\bullet$]$\Efzp \subset X$ exceptional divisor over $\bfz$.
\item[$\bullet$]$\Einf \subset X$ exceptional divisor over $\binf$.
\item[$\bullet$]$\Einfp \subset X$ exceptional divisor over $\binfp$.
\item[$\bullet$] $\dyndeg  $ the first dynamical degree of $F$. 

\item[$\bullet$] $\cW^s(\bIz \cup \bIinf)$ the stable set for $\bIz \cup \bIinf$. 
\item[$\bullet$] $i_0 \subset \Efz$ is the intersection of $\Efz$ with the strict transform by $\pi_X$ of the line $\Lc$. 
\item[$\bullet$] $i_0' \subset \Efzp $ is the point on $\Efzp$ corresponding to the vertical direction. 
\item[$\bullet$] $i_\infty$ is the preimage of the point $\Oinf$ by $\pi_X$. 
\end{enumerate}

\section{General properties of the map and of the spectral measures}
\label{section_general}

\subsection{The mapping and the density of states measures}

We start by recalling the map $F$. Since this section is devoted to the study of the dynamical properties of $F$, we shall rename the variables $(\lambda, \mu)$ to $(x,y)$ respectively. 

\begin{equation} \label{eq_def_F}
F \colon (x,y) \mapsto \left  ( -2 + \dfrac{x (x-2)}{y^2} , \dfrac{x-2}{y^2}  \right  ). 
\end{equation}

The map induces a map on $\bP^2(\C)$, given in homogeneous coordinates as: 
\begin{equation} \label{eq_def_F_homogeneous}
F \colon [x:y:z] \mapsto [x^2 - 2 x z -2 y^2: xz - 2 z^2 : y^2].
\end{equation}
We will need the expression of the second iterate $F^2 = G$ given by:
\begin{equation} \label{eq_G}
G (x, y) = \left ( x^2 - 2 - \dfrac{2y^2 (-6x + 3 x^2-4y^2)}{(x-2)^2} , \dfrac{y^2 (-2x+x^2 - 4y^2)}{(x-2)^2}  \right ).
\end{equation}

We introduce some sequence of polynomials inductively. 
\begin{equation} \label{eq_inductive}
 P_{n+1} (x , y) = y^{2^n} P_n(F(x, y)),
 \end{equation}
 where $P_0 (x,y) = x- 2 y  - 2$. 
The \textit{density of states} measure is defined as the weak limit (if it exists) of the measures: 
\begin{equation} \label{eq_density_states}
\omega = \lim_{n\rightarrow +\infty} \dfrac{1}{2^n} \sum_{P_n(a,1) = 0} \delta_a. 
\end{equation}
For more about how these polynomial arise, we refer to Appendix \ref{section_appendix_A}.

We summarize the general properties of $F$ on $\bP^2(\C)$ in the following proposition. 
\begin{prop} \label{prop_gen_prop_F_P2} The map $F$ satisfies the following properties: 
\begin{enumerate}
\item[(i)] The map $F$ is dominant, has topological degree and algebraic degree both equal to two. 
\item[(ii)] The map $F$ has one indeterminacy point $\bfz$, which is resolved by a succession of two blow-ups: the blow-up $\pi_1\colon X_1 \to \bP^2$ of $\bfz$, followed by the blow-up $\pi_2 \colon X_2 \to X_1$ of the intersection of the strict transform of $\Lc$ with the exceptional divisor $\pi_1^{-1}(\bfz)$.    
Every point of $\pi_1^{-1}(\bfz)$ except one    is mapped by $F\circ \pi_1$ to $\binf$, while the exceptional divisor of $\pi_2$ is mapped by $F\circ \pi_1 \circ \pi_2$ to the line $\Li$. 
\item[(iii)] Its critical locus consists of two lines, a line $\Lc$ which is collapsed by $F$ to the point $\bfzp$ and a transversely critical line $\Lz$ which is mapped to the line at infinity $\Linf$. 

\end{enumerate}
\end{prop}

\begin{proof}
For assertion (i), the fact that the algebraic degree follows from \eqref{eq_def_F_homogeneous}. If we take any point $p=(x_0, y_0) \in \C^2$, the preimage of $p$ by $F$ consists of points $(x,y)$ satisfying the equations $x-2 = y^2 y_0$ and $-2 + x y_0 =0$. By Bezout's theorem, there are $2$ solutions counted with multiplicity in $\bP^2(\C)$, hence $F$ has topological degree $2$.  

\smallskip 

For assertion (ii), observe from the equations defining $F$ that $\bfz=[2:0:1]$ is the only indeterminacy point. 
Let us take the chart $(e,l)$  on $X_1$ so that the blow-down map is of the form $(e,l) \mapsto [2 + e: l e : 1]$. 
The map $F \circ \pi_1$ is given by:
\begin{equation*}
F\circ \pi_1 \colon (e,l) \mapsto  [- \left(2 e l^2-e-2\right):1:e l^2].
\end{equation*} 
In particular, this shows that $F\circ \pi_1$ is regular on this chart, and taking $e=0$ yields the image $\binf$. 

Denote by $i_\infty$ the intersection of the strict transform by $\pi_1$ of the line $\Lc$ with the exceptional divisor $\pi_1^{-1}(\bfz)$.
We have thus shown that $F \circ \pi_1$ is regular on  $\pi_1^{-1}(\bfz)\setminus \{ i_\infty \}$.

Let us take another  chart $(e,l)$ the exceptional divisor $\pi_2^{-1}(i_\infty)$ so that the map $\pi_1 \circ \pi_2$ is of the form $(e,l) \mapsto [2 + l e^2:e:1]$.   
The map $F\circ \pi_1 \circ \pi_2$ is then given as: 
\begin{equation*}
F \circ \pi_1 \circ \pi_2 \colon (e, l) \mapsto [ e^2 l^2+2 l-2: l:1 ].
\end{equation*}
The other chart $(e,l)$ is such that $\pi_1 \circ \pi_2$ is given by $(e,l) \mapsto [2 + l e^2: l e : 1]$ and $F \circ \pi_1 \circ \pi_2$ is of the form:
\begin{equation*}
F \circ \pi_1 \circ \pi_2 \colon (e,l) \mapsto [  2 - 2 l + e^2 l: 1:  l].
\end{equation*}
In particular, $F \circ \pi_1 \circ \pi_2$ is regular on the exceptional divisor $\pi_2^{-1}(i_\infty)$ and maps this divisor to the line parametrized by $l \mapsto [2 l -2: l:1]$, which is exactly $\Li$.  

\smallskip 

Let us prove assertion (iii). The jacobian of $F$ in homogeneous coordinates factors as:
\begin{equation*}
\Jac(F) = -4 x^2 y + 16 x y z - 16 y z^2 = - 4 y( x-2 z )^2.
\end{equation*}
This shows that the critical locus consists of the two lines $\Lc$ and $\Lz$. 
From \eqref{eq_def_F}, one sees that $F$ maps the line $\Lc$ to the point $(-2 , 0) = \bfzp$ while $\Lz$ is mapped to the line $\Linf$ and using \eqref{eq_def_F_homogeneous}, $F_{|\Lz}$ is of the form $[x: 0:1] \mapsto [x: 1:0 ]$ which is one to one. Since $\Lz$ is in the critical locus, one deduces that $F$ is transversely critical on $\Lz$.  

\end{proof}

\subsection{Model domains and algebraically stable models}

There is a priori not a unique space to study a rational map. However, there are some prefered birational models of $\bP^2(\C)$ for which the iteration is simplified. 

We will consider two models,  the blow-up $X$ of $\bP^2(\C)$ at the four points $\bfz, \bfzp , \binf, \binfp$, and the blow-up $Y$ of $\bP^1(\C) \times \bP^1(\C)$ at two points. 
These two spaces are adequate for our study as they are algebraically stable models for $F$. 

Recall that a map is called algebraically stable on a model $X$ if and only if $(F_X^n)^* = (F_X^*)^*$ in $H^{1,1}(X)$ where $F_X$ is the map induced by $F$. 

We will use the following geometric criterion for algebraic stability, due to Fornaess-Sibony and Diller-Favre.

\begin{thm} (see \cite[Theorem 1.14]{diller_favre})
Let $F_X : X \dashrightarrow X$ be a rational map on a projective surface $X$. 
If there are no algebraic curves $C \subset X$ and no iterates $n$ such that $F_X^n(C)$ is a point of indeterminacy of $F_X$, then $F_X$ is algebraically stable. 
\end{thm}

\begin{rem} Note that algebraic stability is a notion that depends on the surface $X$, a map $F : \Pg^2 \dashrightarrow \Pg^2$ might not be algebraically stable on $\Pg^2$, but after a few blow-up, then the lift $F_X : X \dashrightarrow X$ on $X$ becomes algebraically stable on $X$. 
\end{rem}

The aim of the rest of this section is to show that the lift $F_X$ to $X$ is algebraically stable on $X$. 
  Let us fix the class $L$ of the strict transform in $X$ of a general line on $\mathbb{P}^2$ and let $\Efzp, \Einfp, \Efz, \Einf$ be the exceptional divisors over those points.
Denote by $\pi \colon X \to \mathbb{P}^2$ the blow-down map.  
Near each exceptional divisor $\Efzp, \Efz, \Einf, \Einfp$, we use the following coordinate chart $ \pi_j \colon\C_{s_j,\eta_j} \mapsto \mathbb{P}^2$ defined by: 
\begin{align} \label{eq_coordinate_blow_up}
\pi_1(s_1 , \eta_1) = [s_1 \eta_1 -2:s_1 :1], & \quad \pi_2(s_2 , \eta_2) = [s_2 -2 :1 : s_2 \eta_2],\\
\pi_3 (s_3 , \eta_3) = [s_3 + 2 : \eta_3 s_3 : 1], & \quad \pi_4(s_4 , \eta_4) = [ 2 + s_4 : 1 : s_4 \eta_4]. 
\end{align} 
  
We first analyze the new indeterminacies of the lift $F_X$ of $F$. 

\begin{prop} \label{prop_indet_FX} The lift $F_X$ of $F$ to $X$ has exactly three indeterminacy points, they are  $i_0=\Lc\cap \Efz$,  $i_\infty = \pi_X^{-1}(\Oinf)$ and the point $i_0'$ on $\Efzp$ corresponding to the vertical direction.   
\end{prop}  

\begin{proof} Since $F$ maps regularly the point $\Oinf$ to $\bfzp$, the blow-up $F_X$ must map $\Oinf$ to the exceptional divisor $\Efzp$, in particular $\Oinf$ is an indeterminacy point. 
Let us call $p$ the point on $\Efzp$ corresponding to the vertical direction. 
Note that since $F$ maps $\Lz$ to $\Linf$ and it is locally transversally critical, the exceptional divisor $F_X$ is regular on $\Efzp \setminus \{ p\}$ and its image is a point on $\Einfp$. Since the map $F_X$ is surjective, we deduce that $p$ must blow-up to $\Einfp$, hence it is an indeterminacy point. 

For the point $\Lc\cap \Efz$, the fact that it is an indeterminacy follows from assertion (ii) of Proposition \ref{prop_gen_prop_F_P2}. 
%
\end{proof}
  
We take a basis of $H^{1,1}(X)= \C L \oplus \C \Efzp \oplus \C \Einfp \oplus \C \Efz \oplus \C \Einf $.

We prove the following: 
\begin{prop} \label{prop_lift_X_F} The lift $F_X$ of $F$ to $X$ is algebraically stable and the matrix of the pullback action in $H^{1,1}(X)$ is of the form: 
\begin{equation}
(F_X)^{*} =  \left ( \begin{array}{lllll}
2&1&0&0&0\\
0&0&1&0&0\\
0&0&0&1&0\\
-1&-1&0&0&1\\
0&0&0&1&0
\end{array} \right ).
\end{equation}
Moreover, the pushforward ${F_X}_*$ on cohomology is given by the matrix: 
\begin{equation}
(F_X)_{*}  = \left(
\begin{array}{ccccc}
 2 & 0 & 0 & 1 & 0 \\
 -1 & 0 & 0 & -1 & 0 \\
 0 & 1 & 0 & 0 & 0 \\
 0 & 0 & 1 & 0 & 1 \\
 0 & 0 & 0 & 1 & 0 \\
\end{array}
\right) .
\end{equation}
\end{prop}

\begin{proof}

Let us consider $\tilde{\Lc}$ the strict transform of the collapsed line on $X$. 
Near $\tilde{\Lc}$, $F_X$ is regular and is given by $\pi_1^{-1}\circ F $ and we find: 
\begin{equation*}
\pi^{-1}\circ F \colon [x:y:1] \mapsto \left  (s_1 = \dfrac{x-2}{y^2}  ,\  \eta_1= x \right ).
\end{equation*}
Thus $F_X(\tilde{\Lc})= (s_1= 0, \eta_1= 2)$ and $\tilde{\Lc}$ is a collapsed curve for $F_X$.
Near $\Efzp$: $F_X = \pi_2^{-1}\circ F\circ \pi_1$ is regular and given by
\begin{equation*}
(s_1,\eta_1) \mapsto \left ( s_2 =\frac{s_1 (s_1 \eta_1^2-2 s_1-4 \eta_1)}{s_1 \eta_1-4},\ \eta_2=\frac{s_1}{s_1 \eta_1^2-2 s_1-4 \eta_1} \right )
\end{equation*}
So we see that $\Efzp$, which is determined by the equation $s_1=0$, maps to the point ${(s_2 = 0, \eta_2 = 0)}$. This is expected because the point $\bfzp$ is a regular branch point for $F$. 

\smallskip

Near $\Einfp$, $F_X = \pi_3^{-1}\circ F \circ\pi_2$ is also regular and given by:
\begin{equation*}
(s_2,\eta_2) \mapsto \left (s_3 = -s_2 (2 s_2 \eta_2-s_2-4 \eta_2+4), \eta_3 = \frac{\eta_2 (2 s_2 \eta_2-s_2+2)}{2 s_2 \eta_2-s_2-4 \eta_2+4} \right )
\end{equation*}
so we see that $\Einfp = \{ s_2=0 \}$ is mapped to $\Efz = \{s_3=0 \}$ , taking $(s_2=0, \eta_2=0)$ to $(s_3=0, \eta_3=0)$. 

\smallskip

Near $\Efz,$ $F_X=\pi_4^{-1} \circ F \circ\pi_3$ is given by
\begin{equation} \label{eq_Efz_Einf}
 (s_3 , \eta_3) \mapsto \left (s_4=-s_3 (2 \eta_3^2-1),\ \eta_4=-\frac{\eta_3^2}{2 \eta_3^2-1} \right ). 
 \end{equation}
 Thus $\Efz$ is not exceptional and maps to $\Einf$. Moreover, $(s_3=0, \eta_3=0)$ maps to ${(s_4 =0, \eta_4 = 0)}$. 
 
 \smallskip
 
 Near $\Einf$, $F_X = \pi_3^{-1}\circ F\circ\pi_4$ is given by
 \begin{equation} \label{eq_Einf_Efz}
  (s_4,\eta_4) \mapsto \left (s_3= -s_4 (2 s_4 \eta_4-s_4+4 \eta_4-4) ,\  \eta_3=\frac{\eta_4 (2 s_4 \eta_4-s_4-2)}{2 s_4 \eta_4-s_4+4 \eta_4-4}\right )
  \end{equation} 
 so $\Einf$ is not exceptional and maps to $\Efz$, moreover the point $(s_4 = 0 , \eta_4=0)$ maps to $(s_3= 0, \eta_3=0)$. 
 The action of $F_X$ is then given by: 
 \begin{multline*}
 \tilde{\Lc} \mapsto (s_1=0,\eta_1=0)\in \Efzp \mapsto (s_2=0,\eta_2=0)\in \Einfp \mapsto (s_3=0,\eta_3=0) \in \Efz \\
 \mapsto (s_4=0, \eta_4=0)\in \Einf \mapsto (s_3=0,\eta_3=0) \in \Efz \mapsto \ldots
 \end{multline*}
The orbit of $\tilde{\Lc}$ is thus finite, and so is the orbit of the other collapsed curve $\Efzp$: 
\begin{multline*}
\Efzp \mapsto (s_2=0,\eta_2=0)\in \Einfp \mapsto (s_3=0,\eta_3=0) \in \Efz \\
 \mapsto (s_4=0, \eta_4=0)\in \Einf \mapsto (s_3=0,\eta_3=0) \in \Efz \mapsto \ldots 
\end{multline*}
None of these points of the orbits of $\tilde{\Lc}$ and $\Efzp$ are indeterminate points for $F_X$, so $F_X$ is an algebraically stable mapping.  

\medskip

Let us now compute the action of $F_X$ by pullback on $H^{1,1}(X)$. 
The divisor $\Efzp$ is defined by $\{ s_1=0 \}$, so we see that $F_X^{-1}(\{ s_1=0\})= \{(x-2)/y^2 =0 \}$ which defines $\Lc$ with multiplicity $1$. 
We conclude that $
 F_X^* \Efzp = \{ \Lc \}_X$. Next we observe that $\{\Lc \} = L - \Efz$ and we conclude that: 
\begin{equation*}
F_X^* \Efzp = L - \Efz. 
\end{equation*}
With similar arguments, we obtain that: 
\begin{align*}
&F_X^* \Einfp = \Efzp,\\
&F_X^* \Efz= \Einfp + \Einf,\\
&F_X^* \Einf = \Efz.
\end{align*}
The preimage of a general line of equation $\{  a x + b y + c z=0\}$ is a degree two curve of the form: 
\begin{equation*}
 F^{-1} \{ a x + b y + c z=0\} = \{ a x^2+(b - 2a) x z-2 b z^2+(c -2a) y^2 =0 \}.
 \end{equation*} 
 Note that it contains the point $\bfz$ but none of the other points $\bfzp, \binfp,\binf$. We pull this class in $H^{1,1}(X)$ using the identity $ \pi^* F^* = F_X^* \pi^*$: 
 \begin{equation*}
 \pi^* F^{-1}\{ a x + by + cz =0 \} = 2 L - \Efz  = F_X^* L. 
 \end{equation*}
 So $F_X^*$ is represented by the matrix: 
 \begin{equation*}
(F_X)_{| H^{1,1}(X)}^{*} =  \left ( \begin{array}{lllll}
2&1&0&0&0\\
0&0&1&0&0\\
0&0&0&1&0\\
-1&-1&0&0&1\\
0&0&0&1&0
\end{array} \right ).
\end{equation*}

To recover the pushforward matrix, we use the relation $(F_X)_* = ^t(Q F_X^*\cdot Q^{-1} ) $, where $Q$ is the matrix induced by the intersection form on $H^{1,1}(X)$ with respect to the basis $\C L \oplus \C \Efzp \oplus \C \Einfp \oplus \C \Efz \oplus \C \Einf$:  
\begin{equation*}
Q = \left ( \begin{array}{lllll}
1 & 0 & 0 & 0 & 0 \\
0 & -1 & 0 & 0 & 0 \\
0 & 0 & -1 & 0 & 0 \\
0 & 0 & 0 & -1 & 0 \\
0 & 0 & 0 & 0 & -1 \\
\end{array} \right )
\end{equation*}
We get: 
\begin{equation*}
(F_X)_* = \left(
\begin{array}{ccccc}
 2 & 0 & 0 & 1 & 0 \\
 -1 & 0 & 0 & -1 & 0 \\
 0 & 1 & 0 & 0 & 0 \\
 0 & 0 & 1 & 0 & 1 \\
 0 & 0 & 0 & 1 & 0 \\
\end{array}
\right)
\end{equation*}
\end{proof}

\subsection{Dynamical degree}

Recall that the sequence $(\deg(F^n))$ is submultiplicative \cite{russakovskii_shiffman}.
Thus we have the existence of a limit  defining the  dynamical degree $\dyndeg =\lim_{n\rightarrow +\infty} \deg(F^n)^{1/n}$, which is invariant under birational conjugacies (see \cite{dinh_sibony_une_borne_sup}). 

For general mappings, from the formula for the dynamical degree, computing the  dynamical degree explicitly is not straightforward. We often pass to another blow-up of $\mathbb{P}^2$ on which the calculations simplifies. 
Namely, when $F_X: X \dashrightarrow X$ is algebraically stable, the dynamical degree of $F$ is related to a spectral radius 
$$\dyndeg(F_X) = \rho (F_X^*),$$
where $F_X^*$ is the pullback action of $F_X$ on the Dolbeaut cohomology $H^{1,1}(X)$.
Furthermore, the  rate of convergence in cohomology to the eigenvector for the largest eigenvalue is exponential.

\begin{prop} \label{prop_cohomology} Let $F: X\dashrightarrow X$ be an algebraically stable map on a projective surface $X$. Assume that $\delta(F)^2 > \dtop(F)$, then there exists two nef classes $\theta_+ ,\theta_-\in H^{1,1}(X)$ such that the following properties hold.  
\begin{enumerate}
\item[(i)] $F^{*} \theta_+ = \lambda_1(F) \theta_+$. 
\item[(ii)] $F_* \theta_- = \lambda_1(F) \theta_-$. 
\item[(iii)] $\theta_+^2 =0$. 
\item[(iv)] $\theta_+ \cdot \theta_- >0$.
\item[(v)] For any $\alpha \in H^{1,1}(X)$ such that $\alpha \cdot \theta_- =0$, one has $|| (F^{n})^* \alpha || \leqslant C \dtop(F)^n $ where $C$ is a constant which only depends on $\alpha$. 
\end{enumerate}
\end{prop}

\begin{proof}
The pullback $F^*$ on $H^{1,1}(X)$ preserves the cone of nef classes. By Perron-Frobenius theorem, there is a unique class $\theta_+$ (up to scaling) which is $\lambda_1(F)$-invariant by $F^*$ and belongs to this cone. Similarly, the same applies for $F_*$ and there exists a unique class $\theta_-$ (up to scaling) which is $\lambda_1(F)$-invariant and which is nef. This proves assertion (i) and (ii).
Consider the desingularization $\Gamma$ of the graph of $F$ in $X\times X$ and let us denote by $\pi_1,\pi_2 \colon \Gamma \to X$ the two  induced projections onto the first and second component respectively. 
The class $\theta_+$ has the property that for any exceptional divisor $E$ for $\pi_1$, one has $\theta_+ \cdot E =0$. 
By \cite[Corollary~3.4]{diller_favre}, we obtain that $\lambda_1(F)^2 \theta_+ \cdot \theta_+ =(F^* \theta_+ \cdot F^* \theta_+) = \dtop(F)\theta_+ \cdot \theta_+ $. 
Since $\lambda_1(F)^2 > \dtop(F)$, we obtain that $\theta_+ \cdot \theta_+ = 0$, thus assertion (iii) holds. 
Assertion (iv) follows from \cite[Theorem~5.1 (4)]{diller_favre} and the last assertion from the fact that $\lambda_1(F)$ is the only eigenvalue of $F^*$ on $H^{1,1}(X)$ which is outside the disc of radius $\sqrt{\dtop(F)}$ (see e.g \cite[Remark 5.2]{diller_favre}).
\end{proof}

\begin{thm} \label{thm_value_dyndeg}

The first dynamical degree $\dyndeg \simeq 1.801$ of $F$ is the largest root of the polynomial $X^3- X^2 - 2X +1 $. 
\end{thm}

\begin{proof}
Since $F$ is algebraically stable on $X$, the dynamical degree is equal to the spectral radius of the action of $(F_X)^*$ on $H^{1,1}(X)$. Using the previous proposition, the characteristic polynomial of the matrix induced by $(F_X)^*$ is of the form: 
\begin{equation*}
  -(-1 + x) x (1 - 2 x - x^2 + x^3).
  \end{equation*}
  Set $P = 1 - 2 x - x^2 + x^3$, observe that $P(1) <0$, thus $P$ has a root strictly larger than $1$. In particular, the dynamical degree is the largest root of $P$, as required.   
\end{proof}

\subsection{Behavior at point of indeterminacy}

We will refer to the first subsystem as the \textit{$\Efz$ subsystem}. 
\begin{prop} \label{prop_E0_dynamics} The map $F_X^2$ preserves $\Efz$ and the restriction $(F_X^2)_{|\Efz \setminus I(F_X)}$ is conjugate to the polynomial map $z \mapsto z^2 - 12$.  
\end{prop}

\begin{proof} Recall that $\Efz$ is mapped by $F_X$ to $\Einf$ which is mapped back to $\Efz$. We use the coordinate chart $\pi_3, \pi_4$ given in \eqref{eq_coordinate_blow_up}. 
We compose the maps induced by $F_X$ in these local coordinates given in  \eqref{eq_Efz_Einf} and \eqref{eq_Einf_Efz} restricted to $s_3 = s_4 =0$ so that 
\begin{equation*}
\pi_3^{-1} F_X^2 \pi_3 = \pi_3^{-1} \circ F_X \circ \pi_4 \circ \pi_4^{-1} \circ F_X \circ \pi_3 : (s_3 = 0 , \eta_3) \mapsto \left (0, \dfrac{\eta_3^2 }{2 - 6 \eta_3^2} \right ).
\end{equation*}
One then checks that this map is conjugate to $z \mapsto z^2 -12$ by setting $\eta_2 = 2/ z$. 
\end{proof}

For the specific map $F$, we take $X$ to be the blow-up of $\mathbb{P}^2$ at the four points
$\bfz, \bfzp, \binf, \binfp$.

In the course of the proof of the above proposition, we obtained  the critical locus of $F_X$. This will be used later on, so we state it below. 
\begin{prop} \label{prop_crit_FX} The critical locus of $F_X$ consists of three curves, the exceptional divisor $\Efzp$, the strict transform of the line $\Lc$, of the horizontal line $\Lz$.  
\end{prop}
 
{Finally, let us analyze the typical behavior near $\Efz$. To that end, consider the Julia set $J_0$ of the map $(F_X^2)_{|\Efz}$ and let $\mu_{\Efz}$ be the measure of maximal entropy for this restriction. We claim that the following holds: 

\begin{thm} \label{thm_repelling_measure}
For $\mu_{\Efz}$-almost every point $z \in \Efz$, $F_X^2$ is repelling at $z$. In other words, $\mu_{\Efz}$ is a reppelling type measure.
\end{thm}

\begin{proof}
The maps $F_X^2$ admits only one indeterminacy on $\Efz$, which is the point with coordinates $s_3 = \eta_3 =0$. Its preimages by $F_X^2$ form a countable union of points with zero measure. Let us denote by $\lambda_1, \lambda_2$ the Lyapounov exponents at typical points for the measure $\mu_{\Efz}$. Since $\Efz$ is invariant, we can assume that $\lambda_1$ is the Lyapounov exponent  associated with the direction tangent to $\Efz$ and $\lambda_2$ the transversal one. 
An explicit computation shows that the differential of $F_X^2$  on $\Efz$ in the coordinates $(e,l)= (y , (x-2)/y)$ is of the form: 
\begin{equation*}
D_{(0,\eta_3)} F_X^2 = \left(
\begin{array}{cc}
 \frac{2}{l} & 0 \\
 -\frac{(l-2)^2 (l+2)^2}{2 l} & 4 l \\
\end{array}
\right).
\end{equation*}
In particular, the jacobian of $F_X^2$ at all points $(e=0, l)$ is contant equal to $8$. We get that $\lambda_1 + \lambda_2 = \log 8$. 
Let $G_{\Efz}(z)$ denote the Green function of $J_{\Efz}$ with poles at $z = \infty$. Note that here the linear conjugation imposes $l/2 =z$ to get the conjugation with $p(z) = z^2 -12$.
 The formula of Manning-Przyticki  says that: 
\begin{equation*}
\lambda_1(\mu_{\Efz}) = \log 2 + G_{\Efz}(0),
\end{equation*}
where we are evaluating the Green function at the critical point.
Let us now estimate the Green function at the critical point.
Recall that if $z_n = p^n(z)$, then 
\begin{equation*}
G_{\Efz}(z)= \lim_{n\rightarrow +\infty} \dfrac{1}{2^n} |z_n| = \lim_{n\rightarrow +\infty }G_n(z),
\end{equation*} 
where $G_n(z) = \dfrac{1}{2^n} \log |z_n| $. 
If $z $ is real, then 
\begin{equation*}
G_{n+1} (z) = \dfrac{1}{2^n} \log |z_{n+1}|^{1/2} = \dfrac{1}{2^n} \log \left ( |z_n| \left | 1 - \dfrac{12}{z_n^2} \right  |^{1/2} \right ) < G_n(z).
\end{equation*}
Thus we find that $$G_{\Efz}(0) < G_1(0) = \dfrac{1}{2} \log |p(0)| = \log(\sqrt{12}). $$
This means that $$\lambda_2 = \log 8 - \lambda_1 > \log 8 - \log(\sqrt{12})>0,$$
hence both Lyapounov exponents are positive and $\mu_{\Efz}$ is a repelling type measure.
\end{proof}
}

\subsection{Some local volume estimates}

In this section, we study the local properties of the map $F$ near each point under iteration. At the critical loci, the map $F$ collapses certain regions, one needs to control the frequency of this phenomenon and that the volume of a neighborhood decays under iteration faster that $\dyndeg(F)^{-n}$.

To any point $z \in X$, we associate the \textit{volume exponent} $\sigma(z,F_X)$ as the smallest positive number such that for any $\gamma > \sigma(z,F)$, there exists a constant $K>0$ and a neighborhood $N$ of $z$ such that for any measurable set $Y \subset X$,  
\begin{equation}
\vol(F_X^{-1}(Y)\cap N) \leqslant K \vol(Y)^{1/\gamma}. 
\end{equation}

The goal of this section is to show that $\sigma(z,F_X^n)$ is bounded above by a constant $\alpha< \dyndeg$.
The following two lemmas are used to estimate this quantity.

\begin{lem} \label{lem_jac_exponent}(\cite[Lemma 5.5]{bleher_lyubich_roeder_2}) Let $U,V$ be two Zariski open subsets of $X$ on which $F_X$ restricts to a holomorphic $d_{top}$-to-one map $F\colon U \to V$. If $\Jac(F_X) \sim x^a y^b$ in suitable local coordinates $(x,y)$ centered at $z\in U$, then: 
\begin{equation*}
 \sigma(z, F_X) \leqslant \max(a,b) +1. 
 \end{equation*} 
\end{lem}

\begin{lem} \label{lem_order_volume} Let $U,V$ be two Zariski open subsets of $X$ on which $F_X$ restricts to a holomorphic $d_{top}$-to-one map $F\colon U \to V$. If $\Jac(F_X)$ vanishes at order $k$ at $z$, then $\sigma(z,F_X) \leqslant k+1$.
\end{lem}

Let us denote by $B_n$ the set: 
\begin{equation}
B_n = \{ p \in X | \sigma(p, F_X^n) \geqslant \dyndeg^n  \}.
\end{equation}

Our goal is to show that for sufficiently large $n$, $ \cap_{k\leqslant n} B_k$ is empty. 
The strategy is to stratify the critical locus according to the volume decay and  analyze the recurrence of this locus under iteration. 
One additionnal difficulty is that we have to control the behavior of the volume decay near indeterminacy points as there are some orbits that land on the indeterminacy points which are then mapped to the critical locus of $F_X$. We argue that these orbits are very special so that after some iterates, the volume exponents decays appropriately.   
Recall from Proposition \ref{prop_crit_FX} that $Crit(F_X) = \tilde{\Lc} \cup \tilde{ \Lz} \cup \Efzp$. We start by describing the behavior of $F_X$ on non-indeterminate points of the critical locus.

\begin{prop} \label{prop_vol_decay_FX} The following properties hold:
\begin{enumerate}
\item[(i)] For any point $p \in \Einf \cup \Einfp \cup \Efz \setminus (\{ i_0 \} \cup \Efz \cap \tilde{\Lz})$, one has $\sigma(F_X, p) = 1$. 
\item[(ii)] For any point $p \in \Efzp \setminus \{ i_0' \}$, one has $\sigma(F_X , p) \leqslant 2$. 
\item[(iii)] For any point $p \in \tilde{\Lc} \setminus \{ i_0 , i_\infty \}$, one has $\sigma(F_X,p) \leqslant 2$.
\item[(iv)] For any point $ p \in \tilde{\Lz}$, one has $\sigma(F_X , p) \leqslant 2$.  
\end{enumerate}
\end{prop}

\begin{proof}
For assertion (i), one observes that $F_X$ is locally invertible near $\Einf \cup \Einfp \cup \Efz \setminus(\{ i_0 \} \cup \Efz \cap \tilde{\Lz})$, hence $\sigma (F_X, p) = 1$ for each point $p \in \Einf \cup \Einfp \cup \Efz \setminus (\{ i_0 \} \cup \Efz \cap \tilde{\Lz})$ by Lemma \ref{lem_jac_exponent}. 
\smallskip

Let us prove assertion (ii). Take  a point $p \in \Efzp$ which is not indeterminate for $F_X$. Observe that the exceptional divisor $\Efzp$ gets collapsed to the point of intersection  of $\Einfp$ with the strict transform of $\Linf$.
 Consider some coordinates $(u,v) \mapsto [-2 + u: (l_0 +v)u : 1]$ on $\Efzp$ near $p$ where $l_0 \in \C$, the Jacobian of $F_X$ is of the form: 
 \begin{equation*}
  -\frac{2 u ({l_0}+v)}{2 {l_0}^2 u+4 {l_0} u v+2 u v^2-u+4} = u (l_0 +v) (-1/2 + o(1)).
 \end{equation*}
 By Lemma \ref{lem_jac_exponent}, this shows that $\sigma(F_X, p) \leqslant 2$ for all $p \in \Efzp\setminus \{ i_0' \}$, as required.
 \smallskip
 
 Let us prove the third assertion. Take $p = (2 , y_0)$ a point on $\tilde{\Lc} \setminus \{ i_0 , i_\infty \}$ with $y_0 \in \C^*$. Recall that $\tilde{\Lc}$ is collapsed to the intersection of the strict transform of $\Li$ with $\Efzp$. Using the coordinates $(s_1, \eta_1)$ on $\Efzp$ and the coordinates $(u,v) \mapsto (2 + u , y_0 +v)$ near $p$, the Jacobian of $F_X$ is then of the form: 
 \begin{equation*}
 -\frac{2 u}{(u+2) (v+{y_0})^3} =    u ( -1 + o(1)).
 \end{equation*}
 By Lemma \ref{lem_jac_exponent}, this gives $\sigma(F_X,p) \leqslant 2$, as required. 
 
 \smallskip 
 
 We prove the last assertion. Take $p = (x , 0) $ with $x \neq \pm 2$ so that $p \neq \bfz, \bfzp$ and fix some local coordinates $(u,v) \mapsto ( x + u , v)$ near $p$. Since $F_X$ swaps the line $\Lz$ and $\Linf$ and since $F(p) \neq \binf, \binfp$, we can compute the Jacobian of $F_X$ using the standard coordinates, it factors into: 
 \begin{equation*}
 \frac{2 v}{u+x-2} = v \left ( \dfrac{2}{x-2} + o(1) \right ).
 \end{equation*}
By Lemma \ref{lem_jac_exponent}, we have $\sigma(F_X,p) \leqslant 2$, so the assertion holds for all point of $\tilde{\Lz} \setminus (\Efz \cup \Efzp)$. Moreover, by assertion (ii), it is also true for $\tilde{\Lz} \cap \Efzp$. 
We thus need to find the bound at the point $\Efz \cap \tilde{\Lz}$. 
Using the coordinate $(s_3, \eta_3) , (s_4, \eta_4) $ near $\Efz$ and $\Einf$ respectively, we obtain that the Jacobian of $F_X$ at $\Efz \cap \tilde{\Lz}$ is of the form: 
\begin{equation*}
-\frac{2 \eta_3}{2 \eta_3^2-1} =  \eta_3( 2 + o(1)).
\end{equation*}
By Lemma \ref{lem_jac_exponent}, the volume exponent is also smaller than $2$, as required. 
\end{proof}

Set $G_X = F_X \circ F_X$. 
The above proposition shows that a priori $B_1$ is the entire critical locus of $F_X$. However, the bound on the volume exponent for $F_X$ yields some restrictions on $B_2$.

\begin{prop} \label{prop_vol_decay_GX} The following properties hold:
\begin{enumerate}
\item[(i)] For any point $p \in \Efzp \setminus \{ i_0' \}$, one has $\sigma(G_X , p) \leqslant 2$.
\item[(ii)] For any point $p\in \tilde{\Lc} \setminus \{ i_0 , i_\infty \}$, one has $\sigma(G_X, p) \leqslant 3$.
\item[(iii)] For any point $p \in \tilde{\Lz} \setminus F_X^{-1}\{ i_\infty \}$, one has $\sigma(G_X, p) \leqslant 2$.
\end{enumerate}
In particular, the inclusion  $B_1 \cap B_2 \subset \{ i_0, i_0', i_\infty \} \cup F_X^{-1}\{ i_\infty \}= I(F_X)\cup F_X^{-1}\{ i_\infty \}$ holds. 
\end{prop}

\begin{proof}
Assuming that the three assertions hold, this implies that the complement of $B_2$ contains $$\Efzp \cup  \tilde{\Lz} \cup \tilde{\Lc} \setminus ( \{ i_0', i_0, i_\infty\} \cup F_X^{-1}\{ i_\infty\}) = Crit(F_X) \setminus ( \{ i_0', i_0, i_\infty\} \cup F_X^{-1}\{ i_\infty\}).$$
Since $B_1 $ is contained in the critical locus of $F_X$, the inclusion follows.

 Since $F(\Efzp \setminus \{ i_0' \}) \subset \Einfp$ and as $\sigma(G_X , p) \leqslant \sigma(F_X, p) \sigma(F_X, F_X(p))$ for $p\in X \setminus I(F_X)$, 
we deduce assertion (i)  using the bounds given by (i) and (ii) of Proposition \ref{prop_vol_decay_FX}.
Similarly, assertion (iii) follows from the fact that $F_X (\tilde{\Lz} \setminus F_X^{-1}\{i_\infty \}) \subset \tilde{\Linf} \setminus \{ i_\infty \}$, from the local invertibility of $F_X$ on the latter set as well as item (iv) of Proposition \ref{prop_vol_decay_FX}. 
\smallskip

Let us prove assertion (ii).  Let us choose a point $p = (2 , y_0)$ on $\Lc$ where $y_0 \neq 0$. We choose a local chart $(u,v) \mapsto (2 + u , v + y_0)$, $G_X$ maps to a point on $\Einfp$ and we choose the local chart $(s_2,\eta_2)$ near $\Einfp$ and the second iterate is of the form $(u,v) \mapsto (u',v')$ where: 
\begin{align*}
s_2= & \frac{u \left(-u^3-4 u^2+4 u v^2+8 u v y_0+4 u y_0^2-2 u+8 v^2+16 v y_0+8 y_0^2\right)}{(v+y_0)^2 \left(-u^2-2 u+4 v^2+8 v y_0+4 y_0^2\right)},\\
 \eta_2= &\frac{u}{u^3+4 u^2-4 u v^2-8 u v y_0-4 u y_0^2+2 u-8 v^2-16 v y_0-8 y_0^2}.
\end{align*} 
And the Jacobian is given by: 
\begin{multline*}
-\frac{4 u^2}{(v+y_0)^3 \left(u^3+4 u^2-4 u v^2-8 u v y_0-4 u y_0^2+2 u-8 v^2-16 v y_0-8 y_0^2\right)} =\\
 -4u^2 \left ( - \dfrac{1}{8 y_0^5} + o(1)  \right ).
\end{multline*}
This shows that $\sigma(p,G) \leqslant 3$.
\end{proof}

Let us denote by $\tilde{\bbbO} = F_X^{-1}\{i_\infty\} = \pi^{-1}(\bbbO)$.
By Proposition \ref{prop_vol_decay_GX}, we are then reduced to study the properties of the volume under iteration near the four points $i_0, i_0', i_\infty , \tilde{\bbbO}$.  
We find an upper bound for $\sigma(p, G_X)$ at three points.

\begin{prop} \label{prop_volume_three_indet} The following assertions hold.
\begin{enumerate}
\item[(i)] One has $\sigma(\tilde{\bbbO}, G_X) \leqslant 3$. 
\item[(ii)] One has $\sigma(i_\infty, G_X) \leqslant 4$.
\item[(iii)] One has $\sigma(i_0', G_X) \leqslant 2$. 
\end{enumerate}

\end{prop}

\begin{proof}
Let us prover the first assertion. 
Consider $\pi' \colon X' \to X$ the blow-up of $X$ at $\tilde{\bbbO}$ with exceptional divisor $E$. 
One has: 
\begin{equation*}
\sigma(\tilde{\bbbO} , G_X) \leqslant \sup_{p\in E} \sigma(p, G_X\circ \pi'). 
\end{equation*}
We choose some local coordinates $(u,v) \mapsto [u: (v+ l_0) u:1]$ near a point $p\in E$ where $l_0\in\C$. The map $G \circ \pi'^{-1}$ is given by  $(u , v) \mapsto (u',v')$ where
\begin{align*}
u'=& \frac{u^2 \left(2 l_0^2 u+4 l_0 u v+2 u v^2-u+2\right) \left(4 l_0^2 u+8 l_0 u v+4 u v^2-u+2\right)}{(u-2)^2}\\
v' =& -\frac{u (l_0+v)^2}{2 l_0^2 u+4 l_0 u v+2 u v^2-u+2}
\end{align*}

 and the Jacobian of the form: 
 \begin{equation*}
 \frac{4 u^2 (l_0+v) \left(4 l_0^2 u+8 l_0 u v+4 u v^2-u+2\right)}{(u-2) \left(2 l_0^2 u+4 l_0 u v+2 u v^2-u+2\right)} = 4 u^2 (l_0 +v) \left ( - \dfrac{1}{2} + o(1) \right ).
 \end{equation*}
  Using Lemma \ref{lem_jac_exponent}, $\sigma(p, G\circ \pi'^{-1}) = 3$ for all $l_0\in \C$. 
Let us calculate the volume exponent for the point $p$ on $E$ in the second chart, corresponding to the vertical direction. 
Choose some local coordinates $(u,v) \mapsto [u v : u:1]$, on $\Efzp$ we choose the local coordinates $(u',v') \mapsto [-2 +  u' v' :u' : 1] $ and the Jacobian of $G \circ \pi'^{-1}$ is given by: 
\begin{equation*}
  \frac{4 u \left(u v^2-4 u-2 v\right)}{u v-2}.
  \end{equation*}  
  In particular, the Jacobian vanishes at order $2$, by Lemma \ref{lem_order_volume}, $\sigma(p,G \circ \pi'^{-1})=3$.  
We have that $$\sup_{p\in E} \sigma(p, G\circ \pi'^{-1}) = 3,$$ hence 
$\sigma(\tilde{\bbbO}, G_X)=3$.  

\medskip

We prove assertion (ii) in a similar fashion. We consider the blow-up $\pi'\colon X' \to X$ of $X$ at the point $i_\infty$ with exceptional divisor $E$. We have: 
\begin{equation*}
 \sigma(i_\infty, G_X) \leqslant \sup_{p\in E} \sigma(p, G_X\circ \pi'^{-1}).  
 \end{equation*}
  On some local coordinates, the Jacobian of $G\circ \pi'^{-1}$ is of the form: 
  \begin{multline*}
  \frac{4 u^3 (l_0+v) (2 l_0+2 v-1)^2}{4 l_0^3 u^2+12 l_0^2 u^2 v-2 l_0^2 u^2+12 l_0 u^2 v^2-4 l_0 u^2 v-2 l_0 u^2+4 u^2 v^3-2 u^2 v^2-2 u^2 v+u^2-4} \\
  = 4 u^3 (l_0+v) (2 l_0+2 v-1) \left ( - \dfrac{(2l_0-1)^2}{4} + o (1) \right ),
  \end{multline*}
  hence Lemma \ref{lem_jac_exponent} shows that $\sigma(i_\infty, G\circ \pi'^{-1}) \leqslant 4$. 
\medskip

Let us prove the third assertion. Observe that the point $i_0'$ has coordinates $(s_1 = 0 , \eta_1=0)$ in the local coordinates $(s_1, \eta_1)$ from \eqref{eq_coordinate_blow_up}. Consider the blow-up $\pi'\colon X' \to X$ of $X$ at the point $i_0'$ and denote by $E$ the exceptional divisor. 
The two coordinate charts on $E$ are given by $(u_1, v_1) \mapsto   [-2 + u_1^2 v_1 : u_1:1]$ and $(u_2, v_2 ) \mapsto [-2 + u_2^2 v_2 : u_2 v_2:1]$. Recall that $F_X$ blows up $i_0'$ to the divisor $\Einfp$.
 We take the chart on $\Einfp$ determined by $(s_2', \eta_2') \mapsto [-2 + s_2' \eta_2' : 1 : s_2' ]$ and the map $F_X \circ \pi'$ is then of the form: 
\begin{equation*}
(u_1 , v_1) \mapsto \left( s_2' =\frac{u_1^2}{u_1^2 v_1-4},{\eta_2'} =u_1^2 v_1^2-4 v_1-2\right )
\end{equation*}
  The Jacobian is equal to: 
\begin{equation*}
-2 u_1, 
\end{equation*}
hence by Lemma \ref{lem_jac_exponent}, $\sigma(p, F_X \circ \pi') \leqslant 2$ for any point $p= (u_1= 0 , v_1)$ with $v_1 \in \C$. 
We are then estimate  the exponent at the point $(u_2 = 0, v_2 = 0)$. Taking the chart $(s_2, \eta_2)$ on $\Einfp$ from \eqref{eq_coordinate_blow_up}, the map $F_X \circ \pi'$ is of the form: 
\begin{equation*}
(u_2, v_2) \mapsto \left(s_2 = \frac{u_2^2 v_2 \left(u_2^2 v_2-2 v_2-4\right)}{u_2^2 v_2-4},\eta_2 = \frac{v_2}{u_2^2 v_2-2 v_2-4}\right).
\end{equation*}
The Jacobian  then factors into: 
\begin{equation*}
\frac{2 u_2 v_2}{u_2^2 v_2-2 v_2-4} = u_2 v_2 \left  ( - \dfrac{1}{2} + o(1) \right ). 
\end{equation*}
By Lemma \ref{lem_jac_exponent}, we get that at $p= (u_2 = 0, v_2=0)$,  $\sigma(p, F_X \circ \pi') \leqslant 2$. 
Overall, this shows that $\sigma(p, F_X \circ \pi') \leqslant 2$ for any point $p \in E$. 
By assertion (i) of Proposition \ref{prop_vol_decay_FX}, $\sigma(q, F_X) = 1$ on $\Einfp$, so we obtain $$\sigma(i_0', G_X) \leqslant \sup_{p \in E} \sigma(p, F_X\circ \pi') \sigma(F_X(\pi'(p)), F_X) \leqslant  2 ,$$ as required.
\end{proof}

 We now analyze the situation near the indeterminacy point $i_0$. 
To that end, we consider $\pi' \colon X' \to X$ the blow-up of the indeterminacy point $i_0$ whose exceptional divisor is $E$.  
In the proposition below, we will distinguish different orbits for the points $p\in E$ under $F_X \circ \pi'$, one generic behavior and three special ones. Denote by $p_c, p_\infty,p_0$ the points on $E$ which are mapped by $F_X \circ \pi'$ to $\tilde{\Lc} \cap \tilde{\Li}, \tilde{\Li} \cap \Einf, \tilde{\Li} \cap \Efzp$ respectively.  

\begin{prop} \label{prop_volume_exponent_i0} The following holds: 
\begin{enumerate}
\item[(i)] If $p\in E $ is distinct from $p_c, p_\infty, p_0$, then $\sigma (p,G_X\circ \pi') \leqslant 2$. 
\item[(ii)] For $ p \in \{ p_c, p_\infty \}$, one has $\sigma(p_c, G_X\circ \pi')\leqslant 2$. 
\item[(iii)] One has $\sigma(p_0, G\circ \pi') \leqslant 3 $. 
\end{enumerate}
In particular, this shows that the point $i_0$ does not belong to $B_2$. 
\end{prop}
 
\begin{proof}
Let us prove assertion (ii). Fix a chart near $p_c$ given by $(u,v) \mapsto [2 + u^2 (2+v): u : 1 ]$. Since $F_X\circ \pi'(p_c)$ belongs to the collapsed line $\tilde{\Lc}$, its image will be in $\Efzp$ and we choose the chart $(u',v') \mapsto [-2+ u': v' u' : 1]$ on $\Efzp$. 
The map $G_X\circ \pi'$ is then of the form: 
\begin{equation*}
(u,v) \mapsto \left ( u'=(v+2) \left(u^2 v+2 u^2+2\right),v'=\frac{1}{u^2 v+2 u^2+2} \right ).
\end{equation*}
Note that this map is locally holomorphic near $u=v=0$ and the Jacobian is given by: 
\begin{equation*}
\frac{2 u (v+2)}{u^2 v+2 u^2+2} = 2 u (1 + o(1)).
\end{equation*}
By Lemma \ref{lem_jac_exponent}, we obtain  $\sigma(p_c, G\circ \pi') \leqslant 2$.
\medskip

Similarly for $p= p_\infty$. We choose a chart near $p_\infty$ given by $(u,v) \mapsto [2+u^2 v : uv :1]$ and since $G_X\circ \pi' (p_\infty) \in \Efz$, we fix a chart at the target given by $(u',v') \mapsto [2+ u' v': u':1]$. 
The map $G_X \circ \pi'$ is then of the form: 
\begin{equation*}
(u,v) \mapsto \left (u'=  v \left(u^2 v-4 v+2\right),v'=\frac{u^4 v-6 u^2 v+4 u^2+8 v-12}{u^2 v-4 v+2}\right ). 
\end{equation*}
This map is locally holomorphic and its jacobian yields: 
\begin{equation*}
-4 u \left(u^2 v-4 v+2\right) = -4 u (2 + o(1)).
\end{equation*}
By Lemma \ref{lem_jac_exponent}, this shows that $\sigma(p_\infty, G_X\circ \pi') \leqslant 2$, as required. 
\medskip

Let us prove assertion (iii). Let us take a chart near $p_0$ given by $(u,v) \mapsto [2 + u^2 v : u: 1]$. Take the chart $(s_2,\eta_2)$ near $G\circ \pi'(p_0) \in \Einfp$.
 The map $G_X\circ \pi'$ is of the form: 
\begin{equation*}
(u,v) \mapsto \left (s_2 =\frac{v \left(u^4 v^3+4 u^2 v^2-4 u^2 v+2 v-8\right)}{u^2 v^2+2 v-4},\eta_2=\frac{v}{u^4 v^3+4 u^2 v^2-4 u^2 v+2 v-8} \right ).
\end{equation*}
It is also holomorphic near zero and the Jacobian is of the form: 
\begin{equation*}
\frac{4 u v^2}{u^4 v^3+4 u^2 v^2-4 u^2 v+2 v-8} = 4 u v^2 (-1/8 + o(1)).
\end{equation*}
In particular, Lemma \ref{lem_jac_exponent} shows that $\sigma(p_0, G\circ \pi') \leqslant 3$.

\medskip

Let us prove assertion (i). We choose a chart near $p \in E\setminus \{ p_0, p_c, p_\infty \}$ given by $(u,v) \mapsto [2 + u^2 (l+v): u: 1]$ where $l \in \C \setminus \{  0, 2\}$. Since $G_X\circ \pi'(p)$ is not on any divisor $\Efz, \Efzp, \Einf, \Einfp$, we choose the chart on the image determined by the affine chart in $\C^2$. The map $G_X\circ \pi'$ is then given by: 
\begin{align*}
 x' =& \frac{1}{(l+v)^2} ( 6 l^2 u^4 v^2+4 l^3 u^4 v+12 l^2 u^2 v+l^4 u^4+4 l^3 u^2-6 l^2 u^2+2 l^2+4 l u^4 v^3+12 l u^2 v^2-12 l u^2 v+ \\
& 4 l v-12 l+u^4 v^4+4 u^2 v^3-6 u^2 v^2+2 v^2-12 v+8), \\
 y' = &\frac{l^2 u^2+2 l u^2 v+2 l+u^2 v^2+2 v-4}{(l+v)^2}.
 \end{align*} 
The map is locally holomorphic and the Jacobian factors into: 
\begin{equation*}
 -\frac{4 u \left(l^2 u^2+2 l u^2 v+2 l+u^2 v^2+2 v-4\right)^2}{(l+v)^3} = - 4u \left (  \dfrac{( 2l- 4)^2}{l^3} + o(1) \right ).
 \end{equation*} 
 By Lemma \ref{lem_jac_exponent}, this shows that $\sigma(p, G\circ \pi')\leqslant 2$. 
\end{proof}

\begin{cor} \label{cor_vol_exponent_three} The inclusion $B_1 \cap B_2 \subset \{ i_\infty \} $ holds. Furthermore, the set $B_1 \cap B_2 \cap B_3$ is empty. 
\end{cor}
 
 \begin{proof}
 By Proposition \ref{prop_vol_decay_GX}, $B_1 \cap B_2 \subset \{ i_0, i_0' , i_\infty, \tilde{\bbbO}\}$. 
 However, as $3 < \dyndeg^2$, assertion (i) and (iii) of Proposition \ref{prop_volume_three_indet} as well as Proposition \ref{prop_volume_exponent_i0} imply that the three points $\tilde{\bbbO}, i_0', i_0$ are also not in $B_2$. This gives the inclusion $B_1 \cap B_2 \subset \{ i_\infty \}$, as required. 
 As for the additional statement, by assertion (ii) of Proposition \ref{prop_volume_three_indet} and since $G_X$ blows up $i_\infty$ to $\Einfp$ on which $F_X$ is invertible, we have: 
 \begin{equation*}
 \sigma(i_\infty, F_X^3) \leqslant \sigma(i_\infty, G_X) \sup_{p \in \Einfp} \sigma(p, F_X)  \leqslant 4 < \delta^3.
\end{equation*}  
This shows that $i_\infty \notin B_3$, so $B_1 \cap B_2 \cap B_3$ is empty. 
 \end{proof}

\begin{thm} \label{thm_vol_exponent} There exists a constant $\alpha < \dyndeg$ such that for any  open set $\Omega \subset X$,  
there exists a constant $K>0$ such that for any $n$, for any orbit $\underline{x}=(x_1, \ldots, x_n)$ in $\Omega^n$, 
\begin{equation}
 \sigma(\underline{x} , F^n) \leqslant K \alpha^n. 
 \end{equation} 
\end{thm}

\begin{rem} If one keeps track of the bound on the volume exponent, we see that $\alpha  $ can be taken equal to $ \max(\sqrt{3}, 4^{1/3})= \sqrt{3}$.
\end{rem}

\begin{proof}
Let us take any open set $\Omega$. 

Take any orbit $\underline{x}=(x_0, \ldots, x_{n}) \in \Omega^n$. 
We denote by $\Omega_1 = \Omega \setminus B_1 $, $\Omega_2 = \Omega \cap (B_1 \setminus B_2)$, $\Omega_3 = \Omega \cap B_1 \cap B_2$. 
By Corollary \ref{cor_vol_exponent_three}, $B_1 \cap B_2 \cap B_3$ is empty so we have $\Omega = \Omega_1 \sqcup \Omega_2 \sqcup \Omega_3$. 
By definition of $B_i$, there exists $1 \leqslant \alpha_i < \dyndeg$ such that for any $p \in \Omega_i$, $\sigma(p, F_X^i)\leqslant \alpha_i^i$. 
Take $\alpha = \max_{i=1,2,3}(\alpha_i)$ and denote by $T$ the left shift so that $T \underline{x} = (x_1, x_2, \ldots, x_n)$.  Given an orbit $\underline{x}$, we denote by $\sigma(\underline{x}, F_X^n)$ the volume exponent of the first point $x_0$ under the local map  $F_X^{n}$ sending a neighborhood of $x_0$ to $x_{n}$. Note that this notation is relevant as we can encounter an indeterminacy point so one must specify the target point. 
The key fact is that if $x_0 \in \Omega_i$, then $\sigma(\underline{x}, F_X^i) \leqslant \alpha_i^i \leqslant \alpha^i$. In particular, this implies: 
\begin{equation*}
\sigma(\underline{x}, F_X^n) \leqslant \sigma(\underline{x}, F_X^i) \sigma(T^i \underline{x}, F_X^{n-i}) \leqslant \alpha^i \sigma(T^i \underline{x}, F_X^{n-i}), 
\end{equation*}
whenever $x_0 \in \Omega_i$ for $i=1,2,3$. 
By induction, we obtain: 
\begin{equation*}
\sigma(\underline{x}, F_X^n) \leqslant \sigma(T^{n-k} \underline{x}, F_X^k) \alpha^{n-k} \leqslant \sigma(T^{n-k} \underline{x}, F_X^k) \alpha^n,
\end{equation*}
with $k \in \{ 0,1,2 \}$. Consider 
 $K = \max_{k=1,2} \sup_{x \in \Omega} \sigma(x, F_X^k)$, we have proved: 
 \begin{equation*}
 \sigma(\underline{x}, F_X^n) \leqslant K \alpha^{n}, 
 \end{equation*}
 as required.
\end{proof}

\section{Two currents}
\label{section_currents}

\subsection{The Sibony current}
\label{section_sibony}

We recall the construction of a particular current due to Sibony. 

\begin{thm} \cite[Theorem 1.9.1]{sibony_1999} Given any dominant rational map $F \colon  \bP^2 \dashrightarrow \bP^2$ which is not algebraically stable on $\bP^2$. Choose a lift $\tilde{F}$ of $F$ on $\C^3$ and write $\tilde{F}^n = h_n \tilde{F}_n$ where $\tilde{F}_n$ is homogeneous without any common factors. Then the limit in the sense of currents on $\bP^2$: 
\begin{equation}
S = \lim_{n\rightarrow +\infty} \dfrac{[h_n=0]}{\deg(F)^n} 
\end{equation}
exists. 
\end{thm}

In our situation, the map $F$ is of algebraic degree $2$ and is not algebraically stable on $\bP^2$. So the construction of Sibony can be applied and yields a current we denote by $S$, the \textit{Sibony current}.  
We explain how it arises naturally in this context. 
Recall from Lemma \ref{lem_recursion} that the sequence of homogeneous characteristic polynomial $\tilde{P}_n \in \C[x, y , z]$ of interest satisfies the relation $\tilde{P}_{n+1} = \tilde{P}_n \circ \tilde{F}$, where $\tilde{F}$ is the lift of $F$ to $\C^3$.   And the spectral current is an accumulating current of the sequence $[ \tilde{P}_n = 0 ]$ in $\bP^2$.  

\begin{thm} \label{thm_sibony_spectral}Let $\tilde{P}_n$ be the sequence of polynomials defined in Lemma \ref{lem_recursion} and let $S$ be the Sibony current associated to $F$.  Then one has: 
\begin{equation*}
S = \lim_{n\rightarrow +\infty} \dfrac{[\tilde{P}_n = 0]}{2^n} 
\end{equation*}
\end{thm}

\begin{proof}
This is a consequence of the fourth assertion in \cite[Theorem 1.9.1]{sibony_1999}. The current $dd^c\log |\tilde{P}_0|$ defines a closed positive current in $\bP^2$ and $2^{-n}dd^c \log |\tilde{P}_n| =2^{-n} dd^c \log |\tilde{P}_0 \circ \tilde{F}^n|$. The statement implies that the limit of these currents exists and coincide with $S$, as required.  
\end{proof}

We explain in the next technical but crucial proposition  the relationship between the spectral current and the successive preimages by $F$ of certain lines. 

We also denote by $\pi \colon X \to \bP^2$ the blow-down morphism and by $\pi^o$ the strict transform by $\pi$.  
Recall that $\Efzp, \Einfp, \Efz, \Einf$ are the exceptional divisors above $\bfzp, \binfp, \bfz$ and $\binf$ respectively.

We lift this relation on the blow-up. 
\begin{prop} \label{prop_error_term}
The following equality of currents holds on $\bP^2(\C)$:
\begin{equation*}
[\tilde{P}_n  = 0] =\sum_{k=0}^{n-2} c_{n-1-k} \pi_*(F_X^k)^* [\pi^o\Lc] +  \pi_*(F_X^n)^* [\pi^o\Lf], 
\end{equation*}
where the coefficients $c_n$  are given by the formulas:
\begin{equation*}
c_n= \dfrac{1}{3} \left (2^{n-1} + \dfrac{3 + (-1)^{n+1}}{2}\right ).
\end{equation*}
for $n \geq 2$ and $c_1 = 0$. 
\end{prop}

We postpone the proof of the above proposition.

First let us show how one can deduce an asymptotic expansion of the spectral current. 

\begin{cor} \label{cor_sibony_series} The limit  $S=\dfrac{1}{2^n} [\tilde{P}_n=0]$ in the sense of currents exists and is equal to the following series: 
\begin{equation} \label{eq_density_asymp}
\lim_{n\rightarrow +\infty}  \dfrac{1}{2^n} [\tilde{P}_n=0] =  \dfrac{1}{12} \pi_* \left ( \sum_{k\geqslant 0} \dfrac{1}{2^k} (F_X^k)^* [\pi^o \Lc] \right ).
\end{equation}
\end{cor}

\begin{proof}
Using the formula in Proposition \ref{prop_error_term} and the previous formula, we obtain the identity:
\begin{align*}
\dfrac{1}{2^n} [P_n = 0] & = \sum_{k=0}^{n-3} \dfrac{c_{n-1-k}}{2^n} \pi_* (F_X^k)^* [\pi^o \Lc] + \dfrac{c_1}{2^n}   \pi_* (F_X^{n-2})^* [\pi^o \Lc] + \dfrac{1}{2^n} \pi_* (F_X^n)^* [\pi^o \Lf] \\
& = \dfrac{1}{12} \left ( \sum_{k=0}^{n-3} \dfrac{1}{2^k} \pi_* (F_X^k)^* [\pi^o \Lc] \right ) + \dfrac{1}{6 \cdot 2^n} \left ( \sum_{k=0}^{n-3} (3+ (-1)^{n-k})\pi_* (F_X^k)^* [\pi^o \Lc] \right ) \\
&+ \dfrac{1}{2^n} \left ( c_1 \pi_* (F_X^{n-2})^* [\pi^o \Lc] + \pi_* (F_X^n)^* [\pi^o \Lf] \right ) .
\end{align*}
Using the fact that the mass of  $(F_X^k)^* \pi^o\Lc,(F_X^k)^* \pi^o\Lf $ is bounded by $C \delta^k$, we get: 
\begin{equation*}
 \dfrac{1}{2^n} [P_n =0] = \dfrac{1}{12} \pi_* \sum_{k\geqslant 0} \dfrac{1}{2^k} (F_X^k)^* [\pi^o \Lc]  + O\left  ( \dfrac{\delta^n}{2^n}  \right )
\end{equation*}
Taking the limit as $n\rightarrow +\infty$ yields: 
\begin{equation}
\lim_{n\rightarrow +\infty} \dfrac{1}{2^n} [P_n =0] = \dfrac{1}{12} \pi_* \sum_{k\geqslant 0} \dfrac{1}{2^k} (F_X^k)^* \pi^o \Lc, 
\end{equation}
as required
\end{proof}

So the error term is controlled by:
\begin{equation*}
S  - \dfrac{1}{2^n} [P_n = 0] =  \dfrac{\dyndeg^n}{2^n}  \left (  \sum_{k \geqslant n}  \dfrac{1}{2^{k-n}\dyndeg^n} \alpha_k (F_X^k)^* \pi^o \Lc -  \dfrac{1}{\dyndeg^n} (F_X^n)^* \pi^o [P_0= 0] \right ). 
\end{equation*}

If we show that the curves equidistribute to the Green current $T_F$, then the right hand side of the equation is controlled by $ \lambda_1^n/2^n T_F$.
\begin{proof}[Proof of Proposition \ref{prop_error_term}]
Given any meromorphic function $f$ on $X$ or $\bP^2$, we shall write by $\divi(f)$ the difference of current $[f=0] - [f=\infty]$. 
We shall use the following equality of currents on the blow-up $X$ of $\Pg^2$.
\begin{align*}
\pi^*\divi(y) &= \pi^o[y=0] + [\Efzp] + [\Efz] - [L_\infty] - [\Einfp] - [\Einf].\\
F_X^* [L_\infty] &= 2 \pi^o [y = 0] +  [\Efzp].\\
F_X^* [\Efzp] &= [\pi^o \Lc]. \\
F_X^* [\Einfp] &= [\Efzp]. \\
F_X^* [\Efz] &= [\Einfp] + [\Einf].\\
F_X^* [\Einf] &= [\Efz], 
\end{align*}
where $L_\infty$ is the strict transform by $\pi$ of the line of infinity in $\bP^2(\C)$. 
Note also that since the line $P_0 = 0$ passes through the points $\bfz, \binf$, we have:
\begin{equation*}
\divi(P_0 \circ \pi) = \pi^o[P_0 = 0]  + [\Efz] - [L_\infty] - [\Einfp].
\end{equation*}
We get using the recursive formula from Lemma \ref{lem_recursion} together with the previous relations:
\begin{align*}
\divi(P_1 \circ \pi) =& 2 \pi^o [y = 0] + 2 [\Efzp] + 2 [\Efz] - 2 [L_\infty] - 2 [\Einfp] - 2 [\Einf] \\
& + F_X^* \pi^o[P_0 = 0] + [\Einf] + [\Einfp] - 2 \pi^o[y=0] - [\Efzp] - [\Efzp] \\
 = &   F_X^* \pi^o[P_0 = 0] + 2 [\Efz] - 2 [L_\infty]   -  [\Einfp] -[ \Einf]. 
\end{align*}
In particular, this shows that the degree of $P_1$ in the two variables $x, y$ satisfies $\deg(P_1) = 2$. 

Since $P_n$ is a polynomial, it has no pole in $\C^2$ and we can write 
\begin{equation*}
\divi P_n \circ \pi = R_n + c_n [\Efzp] + d_n [\Efz] - e_n [\Einfp] - f_n [\Einf] - \deg(P_n) [L_\infty],
\end{equation*}
where $c_n , d_n , e_n , f_n$ are non-negative integers and where $R_n$ is a closed positive current and $R_1 = F_X^*\pi^o\Lf$.
Using the inductive relation \eqref{eq_inductive} and the previous relations, we obtain the following recursion formula:
\begin{equation*}
\left  \lbrace
\begin{array}{lll}
R_{n+1} = F_X^* R_n + c_n \pi^o\Lc,\\
\deg(P_{n+1}) = 2 \deg(P_n),\\
c_{n+1} = 2^{n+1} -   e_n - \deg(P_n) , \\
d_{n+1} = 2^{n+1} - f_n , \\
e_{n+1} = 2^{n+1} - d_n , \\
f_{n+1} = 2^{n+1} - d_n.
\end{array} \right .
\end{equation*}
Together with the fact that $e_1 = f_1 = 1, c_1 = 0 , d_1 = 2 $, the above relation implies that $e_n = f_n $ for all $n\geqslant 1$. And we obtain:
\begin{equation*}
\left  \lbrace
\begin{array}{lll}
c_{n+1} = 2^n - e_n , \\
d_{n+2} = 2^{n+1} + d_n , \\
e_{n+2}= f_{n+2} = 2^{n+1} + e_n ,
\end{array} \right .
\end{equation*} 
for all $n \geqslant 1$. 
We get:
\begin{equation*}
\left \lbrace  \begin{array}{ll}
d_{2n} = 2 \dfrac{2^{2n} - 4}{3} + 3, \\
d_{2n -1} = 2 \dfrac{2^{2n-1}  - 2}{3} + 2,\\
e_{2n} = 2 \dfrac{2^{2n} -4}{3} + 2 = \dfrac{2}{3} (2^{2n} -1), \\
e_{2n-1} = 2 \dfrac{2^{2n - 1} - 2}{3} + 1= \dfrac{1}{3} (2^{2n} - 1),\\
c_{2n} = \dfrac{1}{3} (2^{2n-1}+1 ),\\
c_{2n +1} = \dfrac{1}{3} (2^{2n} + 2) .
\end{array} \right .
\end{equation*}
for all $n \geqslant 1$.
We thus conclude that $c_n$ satisfies is given by the formula: 
\begin{equation*}
c_n= \dfrac{1}{3} \left (2^{n-1} + \dfrac{3 + (-1)^{n+1}}{2}\right )
\end{equation*}
for all $n\geqslant 2$ and this gives the required formula. 
Furthermore, the recursive formula on $R_n$ yields the equality of currents:
\begin{equation*}
R_n = \sum_{k=0}^{n-3} c_{n-1-k} (F_X^k)^* \pi^o\Lc + (F_X^n)^* \pi^o\Lf. 
\end{equation*}
Finally, we deduce the formula for the current $[\tilde{P}_n=0]$ by pushing forward by $\pi$.
\end{proof}

\begin{rem} Our proof of the Proposition \ref{prop_error_term} also yields the limit of the current $2^{-n}[\tilde{P}_n \circ \pi =0]$ on the blow-up $X$. Precisely, continuing the computation, we also have that $d_n$ is a sequence of the form $2^{n+1}/3 + \epsilon_n$ where $\epsilon_n$ is a periodic sequence so that both $c_n/2^n$ and $d_n/2^n$ converge toward $2/3$ and this yields: 
\begin{equation*}
\lim_{n\rightarrow +\infty} \dfrac{1}{2^n} [\tilde{P}_n \circ \pi = 0] = \lim_{n\rightarrow +\infty} \dfrac{[R_n]}{2^n} + \dfrac{c_n[\Efzp] + d_n [\Efz]}{2^n} =  \dfrac{1}{12} \sum_{k \geqslant 0} \dfrac{1}{2^k} (F_X^k)^* [\pi^o(\Lc)] + \dfrac{2}{3} ([\Efz] + \Efzp).
\end{equation*}
\end{rem}

Let us study the current $S_X$ on $X$ given by: 
\begin{equation}
S_X = \dfrac{1}{12} \sum_{k\geqslant 0} \dfrac{1}{2^k} (F_X^k)^* \pi^o \Lc.  
\end{equation}

\begin{thm} The current $S_X$ is the unique closed positive current such that 
\begin{equation} \label{eq_functional_current}
(2 - F_X^*) S_X = \dfrac{1}{6} \pi^o \Lc.
\end{equation}
\end{thm}

\begin{proof}
We first argue that the class of  $[S_X] \in H^{1,1}(X)$ is unique. 
Indeed, the pullback action $F_X^*$ on $H^{1,1}(X)$ has spectral radius $\delta$, so the matrix $2Id - F_X^*$ is invertible, so the class $[S_X]$ is uniquely determined. 

Recall that a DSH function on $X$ is outside a pluripolar set a difference of quasi-plurisubharmonic functions (i.e quasi-psh which means a plurisubharmonic function minus a smooth function).  
Take another closed positive current $S'$ satisfying the same functional equation. Note that it must be in the same cohomology class as $S$, and by \cite[Proposition 2.2]{dinh_sibony_distribution}, there exists a DSH-function $\varphi$ such that $dd^c \varphi = [S_X]- [S']$.
Moreover, we have $\varphi = \varphi_+ - \varphi_-$ with $\varphi_{\pm}$ quasi-plurisubharmonic and a constant $C>0$ such that $dd^c \varphi_{\pm} \geqslant -C \omega$.   
The functional equation is of the form: 
\begin{equation*}
0 = (2 -F_X^*) (S_X - S') = 2 dd^c \varphi - dd^c \varphi \circ F_X=dd^c\left ( \dfrac{1}{2}  \varphi \circ F_X - \varphi \right ). 
\end{equation*} 
 This implies that $1/2 \varphi \circ F_X - \varphi$ is constant. By removing an appropriate constant we can then infer that $1/2\varphi \circ F_X = \varphi$. 
 Let us show that $\varphi$ is equal to zero. Since $\varphi_{\pm}$ are quasi-psh functions, we can assume that they are non-positive, and by Theorem \ref{thm_vol_exponent} and the inclusion $\{\varphi < -t\} \subset \{ \varphi_+ < -t \}$, one has for any $t>0$:
\begin{equation*}
\vol(\{ \varphi \leqslant -t \}) \leqslant \vol(\{\varphi_+ \leqslant -t\}) =  \vol(\{ \varphi_+ \circ F_X^n \leqslant -2^n t \}) \leqslant K \vol( \{\varphi_+ \leqslant -2^n t\})^{1/\sigma^n}.
\end{equation*}
Since $\varphi_+$ is $C\omega$-psh and  using Lemma \ref{lem_kiselman}, we have that for all $t>0$, $\vol(\{ \varphi_+ \leqslant -t\}) \leqslant A e^{-Bt}$ with $A,B >0$ and  we obtain:
\begin{equation*}
\vol(\{ \varphi \leqslant - t\}) \leqslant K A^{1/\sigma^n} \exp\left( - Bt 2^n/\sigma^n\right). 
\end{equation*}
As $2> \sigma$, taking the limit as $n \rightarrow +\infty$ shows that for all $t>0$, $\vol(\{\varphi \leqslant -t \})=0$. Similarly, one obtains that $\vol(\{-\varphi \leqslant -t \})=0)$ and we conclude that $\varphi=0$ almost everywhere, hence $[S]=[S']$.

Let us now establish the functional equation for $S$. 
 Consider the resolvant $R(z)= (z Id - F_X^*)^{-1}$ acting on closed positive $(1,1)$-currents when $z > \delta$. 
Observe that we have: 
\begin{equation*}
z R(z) = \sum_{k\geqslant 0} \dfrac{1}{z^k} (F_X^k)^*, 
\end{equation*}
So when $z= 2$, we get in particular 
\begin{equation*}
(2 - F_X^*) S_Xs = \dfrac{1}{6} (2 - F_X^*)  R(2) \pi^o \Lc = \dfrac{1}{6} \pi^o \Lc, 
\end{equation*}
as required.
\end{proof}

\begin{cor}
	The sliced measures $  [y=t] \wedge S$ for $t\in \C$ vary continuously in $t$. 
\end{cor}

\begin{rem} A priori, the intersection between the pullbacks of $\Lc$ and a horizontal line $\{ y=t\}$ for $t\in \R$ is a measure supported on $\C \times \{t \}$. However, we will see that the measure is actually supported in $\R \times \{ t\}$.
\end{rem}

\begin{proof}
For any $\epsilon > 0$, we truncate the series
\begin{equation}
	S = S_N + \dfrac{1}{12} \pi_* \left ( \sum_{k\geqslant N} \dfrac{1}{2^k} (F_X^k)^* \pi^o \Lc \right ),
\end{equation}
	where $N$ is chosen large enough so that the mass of $S - S_N$ satisfies $|| S - S_N|| \leqslant \epsilon $.
	Since $S_N$ is a laminar current since it is the integration along finitely many algebraic curves, the sliced measures $t \mapsto [y=t] \wedge S_N$ vary continuously. 
	This shows that the sliced measures $[y=t] \wedge S_N$  converge uniformly toward the sliced measure $[y=t] \wedge S$ as $N\rightarrow +\infty$, hence the measures $S \wedge [y=t]$ vary continuously. 
	\end{proof}

\subsection{Green current}

In the rest of the paper, we will denote by $\dyndeg$ the dynamical degree of the map $F$.

\begin{thm} (see \cite[Theorem 2.1]{diller_dujardin_guedj_I}) Let $F : X \dashrightarrow X$ be an algebraically stable  rational map on a compact K\"ahler surface and let $\theta \in H^{1,1}(X)$ such that:
\begin{enumerate}
\item[(i)] the dynamical degree $\dyndeg$ of $F$ satisfies $\dyndeg > 1$; 
\item[(ii)] $F^* \theta = \dyndeg \theta$ in $H^{1,1}(X)$;
\item[(iii)] $\dyndeg^2 > \dtop(F)$ where $\dtop(F)$ is the topological degree.  
\end{enumerate}
Then there exists a unique current $T_F$ representing $\theta$. Moreover, for any smooth form $\alpha $ representing $\theta$, one has the convergence in the sense of currents
\begin{equation*}
\lim_{n\rightarrow +\infty} \dfrac{1}{\dyndeg^n} (F^n)^* \alpha =T_F. 
\end{equation*}
\end{thm} 

The current $T_F$ is called the \textit{Green current} of $F_X$. 
The previous theorem shows that this current is canonical as pullbacks of smooth forms converge to it. 

We investigate the regularity properties of the Green current.

Recall that by Siu's decompositiion theorem, the Green current $T_F$ can be decomposed into $T_F = A + R$ where $A$ is a countable sum of currents of integration on $X$ and where $R$ is a positive current with the property that for each $c>0$, the locus of points  $x$ where the Lelong number at $x$ is larger or equal to $c$ is of dimension $<1$. When $A=0$, we say that $T_F$ has no \textit{atomic part}. 

\begin{prop} The Green current $T_F$ has no atomic part. 
\end{prop}

\begin{proof}
Applying Siu's decomposition to $F^* T_F = \dyndeg T_F$, we get that $F^* A = \dyndeg A$ where $A$ is the atomic part of $T_F$. 
Let us decompose $A$ into 
$A =  \sum_{i\geq 0} c_i [V_i],
$ where $c_i > 0$ and where $V_i$ are irreducible curves in $X$. 
Up to a reordering of the indices, we can assume that the sequence $c_i$ is decreasing and $c_0 = 1$. 
The number $c_i$ can be determined geometrically as follows, given any regular point $x \in V_i \setminus \cup_{j\neq i} V_j$, the Lelong number $\nu(T_F,x)$ of $T_F$ at $x$ is equal to $c_i$ since $\nu([V_i],x)$ is equal to the multiplicity of  $V_i$ at $x$.

Choose any point $x' \in F_X^{-1}(V_i \setminus \cup_{j\neq i}V_j)$ which is a regular point for $F_X$. Such point exists if $F^{-1}(V_i)$ does  not contain the branch locus of $F_X$. Since the branch locus of $F_X$ is $\Lz$ which is forward invariant by $F_X^2$, we deduce that such point $x'$ exists.  
By construction, $x'$ belongs to a curve $V_{j}$ for some $j$. 
 Since $F_X$ is a local isomorphism near $x'$ from $F_X^{-1}(V_i)$ to $V_i$, in particular, the Lelong number of $F_X^{-1}(V_i)$ at $x'$ is equal to $c_{j}$.  
 
We now claim that for any two distinct irreducible components $V_i, V_j$ , the preimage $F_X^{-1}(V_i)$ and  $ F_X^{-1}(V_j)$ have no common irreducible components apart from  collapsing ones. Indeed, any non collapsing curve would map to both $V_i$ and $V_j$, which are irreducible, so $V_i$ would be equal to $V_j$. 

\smallskip 

In order to show that $A = 0$, we argue depending on the nature of certain components of $A$. 
Assume first that $A$ charges a collapsing curve. Since there are only two collapsing curves $\Efzp$ and the strict transform $\pi^o[\Lc]$ of $\Lc$, and since $\Efzp$ pulls back this curve, we deduce that $A$ contains $C =\pi^o(\Lc)$.  

Observe that  none of the pullback of $C$ by $F_X^n$ contain any collapsing curve. Indeed, both $\Efzp$ and $C$ are collapsed to the same preperiodic orbit which disjoint from $C$. 
  If  $c$ is the Lelong number of $A$ at a generic point of $C$ and write by $b_n$ the Lelong number of $A$ at a generic point of $F_X^{-1}(C)$, we have: 
\begin{equation*}
A = c [C] + \sum_{p\geqslant 1} b_p (F_X^p)^* [C] + R,
\end{equation*}
where $R$ is a positive current which consists of the part of $A$ supported outside the pullbacks of $C$. 
We have: 
\begin{equation*}
F_X^* A = \delta A = \delta  c[C] + \sum_{p\geqslant 1} \delta b_p (F_X^p)^* [C] + \delta R = c F_X^* [C] + \sum_{p\geqslant 1} b_p (F_X^{p+1})^* [C] + F_X^* R.  
\end{equation*}
This gives that $\delta b_1 = c$ and $b_{p-1}= \delta b_p$ for all $p\geqslant 2$ which gives $b_p = c / \delta^p$. 
We get: 
\begin{equation*}
A =  c \sum_{p\geqslant 0} \dfrac{1}{\delta^p} (F_X^p)^* [C] + R. 
\end{equation*}
We get a contradiction as the mass of  $\sum_{p\geqslant 0} (F_X^p)^* [C]/\delta^p$ is infinite: 
\begin{equation*}
Mass(A) \geqslant Mass \left (c \sum_{p\geqslant 0} \dfrac{1}{\delta^p} (F_X^p)^* [C]\right ) = +\infty.
\end{equation*} 
This shows that $A$ does not charge any collapsing curve. 

\bigskip

Now this means that $A = \sum_{i\geqslant 1} c_i [V_i]$ where each curve $V_i$ is not collapsed. 
Recall that the weights $c_i$ are decreasing in $i$ and $c_1 = 1$. 
Take $V_j$ the image by $F_X$ of $V_1$. 
The Lelong number of $A$ at a generic point of $V_j$ is equal to $c_j \leqslant 1$. Taking the pullback, the Lelong number of $F_X^*A$ at a generic point of $V_1$ is also equal to $c_j$. On the other hand, $F_X^* A = \delta A$ so its Lelong number at a generic point of $V_1$ is equal to $\delta$, this gives $ \delta = c_j \leqslant 1$, which is a contradiction. 
\end{proof}

We now  show that this property holds for currents of integration along any algebraic curve.

\subsection{Equidistribution of curves}
We now prove the equidistribution statement, compare with  \cite{diller_dujardin_guedj_I} for equidistribution of smooth forms towards the Green current and with \cite{bleher_lyubich_roeder_2} for the equidistribution of curves. In fact, we adapted the techniques of both papers in our specific setting.    


\begin{lem}  Let $u_n$ be a sequence of quasi-psh functions which are 
$ \theta$-psh where $\theta$ is a smooth semi-positive form, then $(u_n)$ is uniformly bounded in $\Lo$ and relatively compact in $\Lo$.  
\end{lem}


The next lemma is an estimate on the sublevel sets of quasi-psh functions. 

\begin{lem} \label{lem_kiselman} Let $u$ be a quasi-psh function which is $\theta$-psh  where $\theta$ is a smooth semi-positive form, then there exists $A,B$  where $A$ and $B$ only depend on $||u||_{\Lo}$ and  $\theta$ respectively, such that: 
\begin{equation}
\vol \left  ( \{u < -t \} \right ) \leqslant A e^{-B t}
\end{equation}
\end{lem}

\begin{proof} We observe that the Lelong number of $u$ at a point $p$ is given by 
\begin{equation}
\nu(u,p) = \lim_{r\rightarrow 0} \int_{|z-p|< r} dd^c u \wedge dd^c (\log |z - p|).
\end{equation}
Since $dd^c u + \theta = P$ where $P$ is a positive current, we get: 
\begin{equation}
\nu(u,p)  \leqslant \lim_{r\rightarrow 0} \int_{|z-p|< r} P \wedge dd^c (\log |z - p|) \leqslant C \int_X P\wedge \omega = C\{ \theta\} \cdot \{\omega \},  
\end{equation}
where $C>0$ is a constant which only depends on $\omega$. 
Overall, this shows that the Lelong number of $u$ is bounded by a constant which only depends on $\theta$. 

By \cite[Theorem 4.2]{zeriahi}, we get for any compact set $K$ and any $\alpha < 2/ (C \{ \theta\}\cdot \{ \omega \})$
\begin{equation}
\vol (\{ u < -t\}) \leqslant \exp \left (c_1 \int_K |u| \omega^2 + c_2 \right ) e^{-\alpha t}, 
\end{equation}
where $c_1, c_2$ are absolute constants. 
Taking $A =  \exp \left (c_1 \int_K |u| \omega^2 + c_2 \right )$ and $B = \alpha$ gives the result.
\end{proof}

\begin{thm} \label{thm_equidistribution} For any curve $C \subset X$, 
  $$\lim_{n\rightarrow +\infty }\dfrac{1}{\dyndeg^n}[F^{-n}(C)] = \deg(C) T_F, $$
in the sense of currents where $T_F$ is   the Green current of $F$ on $X$. 
\end{thm}

\begin{proof} 
 Set $[C]$ the current of integration on the strict transform of $C$.
Let us fix a smooth form $\theta_+$ representing $\{ T_F\}$ and let us consider the class $\theta_-$ satisfying the condition of Proposition \ref{prop_cohomology}.
We normalize $\theta_-$ so that $\theta_+ \cdot \theta_- =1$. 
Set $a = [C]\cdot \theta_-$ we write: 
\begin{equation}
\dfrac{1}{\dyndeg^n}(F^n)^*[C] = a \theta_+ + \eta_n + dd^c u_n. 
\end{equation}
where $\eta_n$ is a smooth closed $(1,1)$ form satisfying $\{\eta_n \}\cdot \theta_-=0$ and $u_n$ is a quasi-psh function. 
 We show that  $\eta_n$ can be chosen so that $\eta_n $ is close to zero in the sense of currents as follows. 
 Let us choose some smooth closed $(1,1)$ forms $\alpha_1, \alpha_2, \alpha_3,\alpha_4$ such that 
 $(\{ \theta\}, \{\alpha_1 \}, \{\alpha_2 \}, \{\alpha_3 \}, \{\alpha_4 \})$ form a basis of $H^{1,1}(X)$ and such that: 
 \begin{equation*}
  \theta_- \cdot \{  \alpha_i \} =0, 
 \end{equation*}
 for all $i$.
 By definition, the cohomology class $(1/\dyndeg^n) \{(F^n)^* [C]\}$ can be decomposed into: 
 \begin{equation}
\dfrac{1}{\dyndeg^n}(F^n)^*[C] = a \{ \theta_+\} + \sum_{i=1}^4 m_{i,n} \{\alpha_i\} \in H^{1,1}(X), 
 \end{equation}
 where $a , m_{i,n} \in \R$. 
 We thus set $\eta_n = \sum m_{i,n} \alpha_i$. Applying assertion (v) of Proposition \ref{prop_cohomology}, the class $[C] - a \{ \theta_+\}$ is orthogonal to $\theta_-$ so 
 \begin{equation}
 \dfrac{1}{\dyndeg^n} ||(F^n)^* ([C] - a \{ \theta_+ \} )|| \leqslant M \dfrac{2^n}{\dyndeg^n},
 \end{equation}
 where $M>0$ is a constant which depends only on the curve $C$. 
This shows that the coefficients $m_{i,n}$ converge to zero as $n\rightarrow +\infty$. We thus have: 
\begin{equation}
 -\epsilon_n \omega \leqslant  \eta_n \leqslant \epsilon_n \omega, 
 \end{equation} 
 where $0 < \epsilon_n $ decreases to zero.  
 
 We pullback $k$ more times : 
 \begin{equation}
 \dfrac{1}{\dyndeg^{n+k}} (F^{n+k})^* [C] \leqslant a \dfrac{1}{\dyndeg^k} (F^k)^* (\theta_+) + \dfrac{ \epsilon_n}{\dyndeg^k} (F^k)^* \omega + \dfrac{1}{\dyndeg^k} dd^c u_n \circ F^k. 
 \end{equation}
Normalizing $\theta_+$ so that $\int_X \theta^+ \wedge \omega =1 $, we decompose the pullback $(F^k)^* \omega$ as: 
 \begin{equation}
 \dfrac{1}{\dyndeg^k} (F^k)^* \omega = \theta^+  +  \sum m_{i,k}'  \eta_i   + dd^c v_{k}\leqslant \theta_+ + \epsilon_k \omega + dd^c v_k, 
 \end{equation}
 where $v_k$ is a quasi-psh function, which is $\theta_+ + \epsilon_k \omega$-psh. 
 Combining with the previous equation, we get: 
 \begin{equation*}
 \dfrac{1}{\dyndeg^{n+k}} (F^{n+k})^* [C] \leqslant \dfrac{a}{\dyndeg^k} (F^k)^* \theta_+ +  \epsilon_n \theta_+ + \epsilon_n \epsilon_k \omega + dd^c \left ( \dfrac{u_n\circ F^k}{\dyndeg^k} + \epsilon_n v_{k} \right )
 \end{equation*}
 Let us set $w_{n,k} = u_n \circ F^k/\dyndeg^k +  \epsilon_n v_k$, we show that $w_{n,k} $ converges in $\Loloc$ to a function $W$  and that $W = 0$.
 Since the function $u_n$ are all $a \theta_+ + \epsilon_0 \omega$-psh, they are uniformly bounded in $\Loloc$, hence the sequence $(u_n) $ is relatively compact in $\Loloc$.  
 Similarly, $v_k$ is $\theta_+ + \epsilon_0 \omega$-psh, and $(v_k)_k$ is also relatively compact in $\Loloc$. 

By Theorem \ref{thm_vol_exponent}, the volume exponent is well behaved, we have: 
\begin{equation*}
\vol( \{ u_n \circ F^k < - \dyndeg^k t \})= \vol(F^{-k} \{ u_n < - \dyndeg^k t \})  \leqslant  C \vol( u_n < -\dyndeg^k t)^{1/ \sigma^k} 
 \end{equation*} 
 where $\sigma < \dyndeg$.
 By Lemma \ref{lem_kiselman}, the inequality $\vol (u_n < -\dyndeg^k t) \leqslant A e^{-  B \dyndeg^k t}$ yields: 
\begin{equation*}
 \vol( \{ u_n \circ F^k < - \dyndeg^k t \}) \leqslant C' \exp \left (-  tB \dfrac{\dyndeg^k}{\sigma^k} \right ),
 \end{equation*} 
where $C' >0$ only depends on $||u_n||_{\Lo}$ (which is uniformly bounded on $n$).  
Choosing a sequence $t_k = 1/(1+\epsilon)^k$ where $\epsilon > 0$ satisfies the condition  $(1+\epsilon) \sigma < \delta$, the series $\sum_k \vol(u_n \circ F^k < -\delta^k t_k)$ is finite, and by Borel-Cantelli lemma, 
\begin{equation}
\vol \left (  \left \{ \lim_k \dfrac{u_n \circ F^k}{\dyndeg^k} < 0 \right  \} \right ) =0 
\end{equation}
This shows that $u_n \circ F^k / \dyndeg^k$ converges to zero almost everywhere. 

Recall that there exists $A',B'$ such that: 
\begin{equation}
\vol \left ( \{ v_k < -t \} \right ) \leqslant A' e^{-B't}.
\end{equation}
Observe that $A'$ is a constant which does not depend on  $k$. 
Choose a sequence $t_n =   n\epsilon_n / B'$ where $\lambda$ satisfies the condition $1 > \lambda$. 
Using the previous inequality, we have: 
\begin{equation}
\vol \left ( \{ \epsilon_n v_k < -t_n \} \right ) = \vol \left ( \{  v_k < -t_n/\epsilon_n \} \right ) \leqslant A' e^{-n }.  
\end{equation}
Choose a subsequence $n_p, k_p$ such that $\epsilon_{n_p} v_{k_p}$ converges  $V$. 
The series $\sum_p \vol \left ( \{ \epsilon_{n_p} v_{k_p} < -t_{n_p} \} \right ) $ is finite by the previous inequality, Borel-Cantelli's lemma implies that $V =0$ almost everywhere. 
We have shown that $\omega_{n_p, k_p}$ converges to zero in $\Lo$. 
We obtain: 
\begin{equation}
\dfrac{1}{\dyndeg^{n_p+k_p}} (F^{n_p+k_p})^* [C] \leqslant  \dfrac{a}{\dyndeg^{k_p}} (F^{k_p})^* \theta_+ + \epsilon_{n_p} \theta_+ + \epsilon_{n_p} \epsilon_{k_p} \omega + dd^c w_{n_p,k_p}
\end{equation}
Taking the limit as $p\rightarrow +\infty$ and using the convergence of $\theta_+$ toward the Green current, we have: 
\begin{equation}
\lim_{p\rightarrow +\infty} \dfrac{1}{\dyndeg^{n_p+k_p}}(F^{n_p + k_p})^* [C] \leqslant a T_F. 
\end{equation}
As any limiting current must be in the same cohomology class as $a T_F$, we conclude that the limit must be $a T_F$. 
\end{proof}

\section{Induced dynamics on $\R^2$: invariant cone field}
\label{section_map_R}

\subsection{The mapping between regions of $\R^2$}
\label{section_regions}

$\bullet$ We subdivide the upper half plane $\R \times \R^+$ into 6 open regions, denoted $R_1, R_2,R_3,R_4, R_5, R_6$ whose boundaries are delimited by the lines $\Lc, \Lf ,  \bar \Lf$, the vertical line through $\bfzp$ and the branch of the hyperbola $ \gamma= F^{-1} (\bar \Lf)$ passing through $\bfzp$. The regions $R_2, R_3,R_4$ are subdivided into two subregions, $R_2^{\pm}, R_3^{\pm} ,R_4^{\pm}$ situated below and above the indeterminate line $\Li$. 
We get the following picture in Figure \ref{fig_subdivision}. 

\begin{figure}[h!]
\centering
\def\svgwidth{8cm}
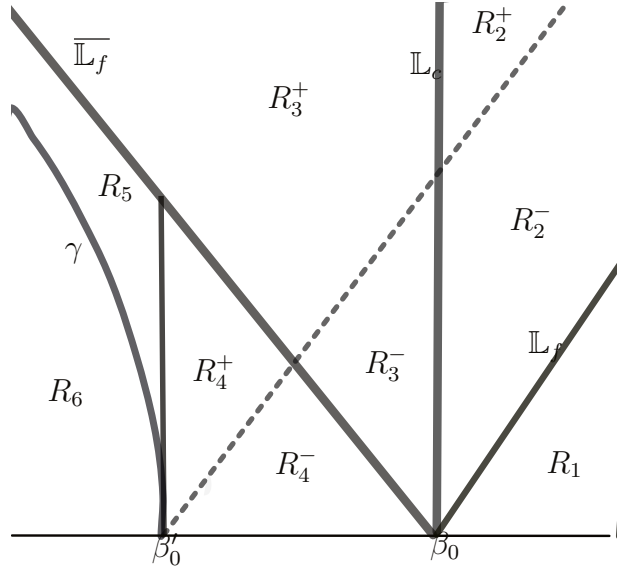
\caption{\label{fig_subdivision} Subdivision of the upper-half plane into 6 regions.}

\end{figure}

We obtain a subdivision of the lower half-plane $\mathbb{R} \times \mathbb{R}^-$ by taking the image of the  decomposition of the upper half-plane by the reflection $\sigma$ by the horizontal axis.

Our goal is to describe how $F$ acts on this subdivision of the real plane. To that end, we will use the properties of the image by $F$ of certain lines of specific pencils. 

\begin{enumerate}
\item The pencil of lines $\mathcal{P}_0$ of lines through $\bfz$. An element of this pencil through $p$ will be denoted $L_p$.
\item The pencil of lines $\mathcal{P}_0'$ of lines through $\bfzp$. An element of this pencil through $p$ will be denoted $L_p'$.
\item The pencil of lines $\mathcal{P}_\infty$  through $\binf$. 
An element of this pencil through $p$ will be denoted $l_p$.
\item The pencil of lines $\mathcal{P}_\infty'$ through $\binfp$.
An element of this pencil through $p$ will be denoted $\bar l_p$.
\item The pencil of lines $\mathcal{P}_c$ through the critical point at infinity. An element of this pencil through $p$ is a vertical line which is denoted $V_p$.
\end{enumerate}


\begin{lem} \label{lem_forward_curves} For any point $p \in \R^2 \setminus \Lz$. The following properties hold.
\begin{enumerate}
\item[(i)] Except when $L_p = \Lc$, one has $F(L_p) = l_{F(p)}$. 
\item[(ii)] Except when $l_p = \Lf$ (resp. $\bar l_p = \bar \Lf$),
 the curve $F(l_p)$  (respectively $F(\bar l_p)$) is a parabola tangent to the indeterminate line $\Li$ at $\bfzp$ which passes through $\bfz$. 
\item[(iii)] Except when $V_p = \Lc$, one has $F(V_p) = L_{F(p)}'$. 
\item[(iv)] The curve $F(L'_p)$ is a parabola tangent to the indeterminate line $\Li$ at $\bfzp$ which passes through $\binfp$. 
Moreover, the convex set enclosed by this parabola contains the half-line $\sigma(\Li) \cap \bH_- \setminus \{ \bfzp \}$.
 \end{enumerate}
\end{lem}

\begin{proof}
Let us prove assertion (i). 
Take $(x,y) \in L_p$, this means that $(x-2)/y$ is a constant that we denote by $C$. Using the fact that $x(x-2)/y^2  = C^2 + 2C/y$, the restriction of $F$ to $L_p$ is then of the form: 
\begin{equation*}
F(x,y) = \left ( -2 + C^2 + \dfrac{ 2 C}{y} , \dfrac{C}{y}  \right).
\end{equation*}
This shows that $F(x,y)=(x',y')$ with $2 y'- x' $ equal to a constant, hence $F$ maps $L_p$ to the line $l_{F(p)}$.

\bigskip

Let us prove assertion (ii). Let us observe that the equality $F\circ \sigma = F$ implies that $l_p$ and $\bar l_p$ have the same image by $F$ when $p$ belongs to the horizontal line $\Lz$. 
We claim that $F(l_p) $ is a degree $2$ curve. 
 Observe that $l_p$ intersects the horizontal axis at a certain point $q \neq \bfz$ and 
 since $F^{-1} (\Linf ) = \Lz $ by assertion (v) of Lemma \ref{lem_compute}, we obtain that $F(l_p)$ intersects the line at infinity at exactly one point $F(q)$.  
Let us compute the multiplicity at $F(q)$ of the intersection $F(l_p) \cap \Linf$. Write $q = (x_0, 0)$ where $x_0 \neq 2$. 
The local map from a neighborhood of $q$ to a neighborhood of $F(q)$ is given by: 
\begin{equation*}
[x: y : 1] \mapsto [ x' : 1 : z'  ].
\end{equation*}
where $$ x' =\dfrac{x^2 - 2 y^2 - 2 x}{x - 2} = x  - 2 \dfrac{y^2}{x - 2}$$
and where $$z'= \dfrac{y^2 }{x - 2}.$$
When we parametrize $l_p$ by setting $x = x_0 + t$ and $y = t/2$, one obtains
\begin{equation*}
\left \lbrace \begin{array}{l}
x' = x_0 + t + O(t^2), \\
z' = \dfrac{t^2 }{x_0 -2} + O(t^3).
\end{array} \right .
\end{equation*}
In particular, this shows that the multiplicity of the intersection $F(l_p) \cap \Linf$ is $2$, hence $F(l_p)$ is a degree $2$ curve by Bezout's theorem, and it is a parabola. 
Note that the line $l_p$ (resp. $\bar l_p$) passes through $\binf$ (resp. $\binfp$) so its image passes through $\bfz$, since these lines both intersect $\Lc$ in $\mathbb{P}^2$, the parabolas $F(l_p), F(\bar l_p)$ also pass through $\bfzp$. We have thus shown that $F(l_p), F(\bar l_p)$ are parabolas going through $\bfz$ and $\bfzp$. 
The fact that $F(l_p)$ and $F(\bar l_p)$ are tangent to the indeterminate line follows from the computation in the proof of Proposition \ref{prop_lift_X_F}. 
 
\medskip 
 
Let us prove assertion (iii), we calculate
$$(y - t(x +2)) \circ F = -\dfrac{(-2 + x) (-1 + t x)}{y^2}, $$
where $t \in \R$. In particular, the image of the vertical line $\{x= t\}$ is mapped to the line passing through $\bfzp$ of slope $t$. 

Let us prove assertion (iv). Take a line $L_p'$. 
 We see that $F(\bfzp) = \binfp$ and since $F$ is transversely critical at $\bfzp$, $F(L_p')$ is a parabola going through $\binfp$. Moreover, since $L_p'$ intersects the line $\Lc$ (in the projective space), the same argument as in the proof of Proposition \ref{prop_lift_X_F} shows that $F(L_p')$ is tangent to $\Li$ at $\bfzp$.
Let $K$ be the convex set enclosed by the parabola $F(L_p')$.
 We show that $\sigma(\Li) \cap \bH_-$ is contained in $K$. Observe that $\sigma(\Li)$ intersects $F(L_p')$ at exactly one point $\bfzp$ on $\R^2$ because it intersects $F(L_p')$ at infinity. 
 By convexity, this implies that $\sigma(\Li) \cap \bH_+$ or $\sigma(\Li) \cap \bH_+$ is contained in $K$. Note that $F(L_p')$ diverges toward infinity in the lower half-plane, this is because $y \circ F = (x- 2)/y^2$ and since $L_p'$ passes through $\bfzp$, nearby points have an image whose   $y$-coordinates is close to $-\infty$. By convexity, we deduce that $\sigma(\Li) \cap \bH_-$ is contained in $K$, as required.
\end{proof}

\begin{prop} \label{prop_mapping_regions}
\begin{enumerate}
\item[(i)] The subset $R_6\cup R_1 \cup \sigma(R_1\cup R_6) \setminus \{\bfz\} $ is forward invariant by $F$. 
\item[(ii)] One has $F(R_2 \setminus \{ \bfz \})  \subset R_2^- \cup R_3^- \cup R_4^-  $. 
\item [(iii)] One has $F(R_3  \setminus \{ \bfz \}) \subset \sigma(R_4 \cup R_5 \cup R_6)$. 
\item[(iv)] One has $F(R_4 \cup R_5 \setminus \{ \bfz \}) \subset \sigma (R_2\cup R_3)$.
\item[(v)] If $B = [-2,2] \times \R$, $F(B\cap R_3 \setminus \{ \bfz \}) \subset \sigma(R_4^+ \cup R_5 \cup R_6 )$ and $F^{-1}(\sigma(R_4^-)) \cap R_3 = R_3 \setminus B$.
\item[(vi)] One has $F(R_4^+  \setminus \{ \bfz \} ) \subset \sigma(R_2^+ \cup R_3^+)$. 
\item[(vii)] One has $F(R_5  \setminus \{ \bfz \}) \subset \sigma(R_2^- \cup R_3^-)$.
\end{enumerate}
\end{prop}

\begin{rem} \label{rem_special_vert_bfzp} The proof of the proposition also shows that the image of the vertical line through $\bfzp$ is contained in the line $\sigma(\Li)$.  
\end{rem}

\begin{proof}

Let us prove assertion (i). First we observe that the line $\Lf$ is invariant by $F$. Moreover, since 
\begin{equation}
y \circ F = \dfrac{x-2}{y^2},
\end{equation}
  the region $\{ x \geq 2 \}$ maps to the upper-half plane by $F$ and in particular, the half line $R_1 \cap \Lf$ is preserved by $F$.  Using assertion (iii) of Lemma \ref{lem_forward_curves}, any vertical segment joining $p \in \Lf\cap R_1$ and a point $q$ on the horizontal axis maps to a half line directed by the line joining $F(p)$ and $\bfzp$. Since this half line starts on the line $\Lf$ at $F(p)$ and cannot contain $\bfzp$, we deduce that this half-line is contained in $R_1$. Using the fact that $F\circ \sigma = F$, we obtain that $R_1\cup \sigma(R_1)  $ is forward invariant. 
  To conclude, we also use the fact that $R_5\cup \sigma
 (R_5)$ maps by $F$ to $R_1\cup \sigma(R_1)$.   

\smallskip

Let us prove assertion (ii). Let us choose any point $p \in R_2\cap \Lf$ and take $L$ the vertical half line starting from $p$ going up toward infinity. Since the half-line $\Lf \cap R_2$ is forward invariant, $F(p)$ belongs also belongs to this half-line and assertion (v) shows that $F(L)$ is exactly the segment joining $F(p)$ and $\bfzp$. We have thus proved that $F(R_2) \subset R_2^- \cup R_3^- \cup R_4^-$.  

\smallskip

Let us prove assertion (iii) and (iv). Fist we observe that  $\mu\circ F$ is non-positive on $R_3\cup R_4 \cup R_5$ so $F(R_3 \cup R_4 \cup R_5) \subset \R \times \R^-$. 
Observe that the equality $F\circ \sigma = F$ and the invariance of the line $\Lf$ shows that $\bar \Lf\cap R_3$ is mapped by $F$ to  $\Lf\cap \R\times \R^-$. For any vertical half-line starting from a point $p\in \bar \Lf$ and contained in $R_3$, its image is by assertion (iii) of Lemma \ref{lem_forward_curves} is a segment joining $F(p) \in \Lf$ and the point $\bfzp$. This shows that $F(R_3)\subset \sigma(R_4\cup R_5 \cup R_6)$. 
Similarly, the image of any  segment joining a point $p $ on the hyperbola in the boundary $\partial R_5$   and $\bfz$ maps by $F$ to a half-line in the lower half plane, passing through $F(p)\in \bar \Lf$ and of slope $1/2$.  This shows that $F(R_4 \cup R_5) \subset \sigma( R_2\cup R_3  )$. 

\smallskip 

Let us prove assertion (v). Note that the vertical line $V$ through $\bfzp$ is mapped to $\sigma(\Li)$, this means that the half line from $V \cap \bar \Lf$ going toward infinity in the upper-half plane is mapped by $F$ to the segment from $\sigma(\Li \cap \Lf)$ to $\bfz$. Take any point $p$ on this half line, the segment $[p , \bfz]$ is contained in $R_3$ and maps by $F$ in the lower half plane and is directed by a line of slope $1/2$. This means it starts at $F(p)$ on the segment from  $\sigma(\Li \cap \Lf)$ to $\bfz$ and goes toward infinity with slope $1/2$ in the lower half plane. This shows that $F(B\cap R_3) \subset \sigma(R_4^+ \cup R_5 \cup R_6)$. For the second equality, one observes that the image by $F$ of any half-line in $\mathcal{P}_0$ starting at $p$ on the vertical line $V$ and going toward infinity in the upper-half plane is a segment contained in $\sigma(R_4^-)$. One obtains $F(R_3 \setminus B) \subset \sigma(R_4^-)$ and the equality follows.

\smallskip

Let us prove assertion (vi). Note that $F(-2,2) = (0,-1)= x_0$ which is the intersection between $\sigma(\Li)$ and $\Lf$. Moreover, $F(0,1) = (-2, -2) = x_1$ is the reflection of the corner of $R_4^+$. Choose any point $p\in R_4^+$ and consider the vertical segment $S$ starting from $\bar \Lf$, passing through $p$ and ending on the line $\Li$. 
By assertion (iii) of Lemma  \ref{lem_forward_curves}, it is mapped by $F$ to a half-line in the lower half-plane starting at $t\in [x_0,x_1]$ and is contained in  $L_{t}' \cap \sigma(R_2 \cup R_3)$. Since $t\in \sigma(R_2^+ \cup R_3^+)$, we deduce that $F(R_4^+ ) \subset \sigma(R_2^+ \cup R_3^+)$, as required.   

\smallskip 

Let us prove assertion (vii). Since the corner $p=(-2,2)$ is mapped to $x_0$ the intersection between $\sigma(\Li) $ and $\Lf$, the half line from $p$ of slope $-1/2$ going toward infinity in the upper-half plane is mapped by $F$ to the segment joining $x_0$ to $\bfz$. If $q \in R_5$, we take the vertical segment going through $q$ and contained in $R_5$, it is mapped to a segment starting from $t \in [x_0, \bfz]$, ending at the point $\bar \Lf \cap L_t'$. This shows that $F(R_5) \subset \sigma(R_2^- \cup R_3^-)$, as required.     
\end{proof}

 \begin{cor} \label{cor_image_parabola} For any line $L \in \mathcal{P}_0' \cup \mathcal{P}_\infty \cup \mathcal{P}_\infty'$ which is not vertical, the image $F(L)$ is a parabola containing $\bfzp$, tangent to $\Li$ at $\bfzp$ and diverging to infinity in the lower half plane.
 \end{cor}
 
 \begin{proof}
By assertion (ii), (iv) of Lemma \ref{lem_forward_curves}, the image $F(L)$ is a parabola which is tangent to $\Li$ at $\bfzp$. 
Moreover, $L$ intersects the horizontal line $\Lz$ at a point on $\bIz \setminus \{ \bfz \}$. In particular, a neighborhood of the intersection point is contained in $R_4$, hence by assertion (iv) of Proposition \ref{prop_mapping_regions}, its  image by $F$ will have in the lower half plane near infinity. This shows that the parabola diverges to infinity in the lower half plane, as required.    
 \end{proof}

\subsection{Construction of a cone field}

Let us set $E = \R^2 \setminus (R_1 \cup R_6 \cup \sigma(R_1 \cup R_6))$ and let $S = \Lf \cup \bar \Lf \cup \Li \cup \sigma(\Li) $ be the union  of four lines.

\begin{defi} \label{def_cone_field} For each point $p \in E \setminus S$, the cone field at $p$ contains the horizontal direction and its boundary is determined by the only two pencils among $\{ \mathcal{P}_0, \mathcal{P}_0', \mathcal{P}_\infty, \mathcal{P}_\infty'  \}$ such that the member $L$ passing through $p$ intersects $\Lz \cup \Linf$ at a point on $\bIz \cup \bIinf \setminus \{ \bfz, \bfzp, \binf, \binfp \}$. 
\end{defi}

We refer to Figure \ref{fig_cone_field} for an illustration of the cone field at various points.

\begin{figure}[h!] \label{fig_cone_field}
\centering
\includegraphics[scale=0.6]{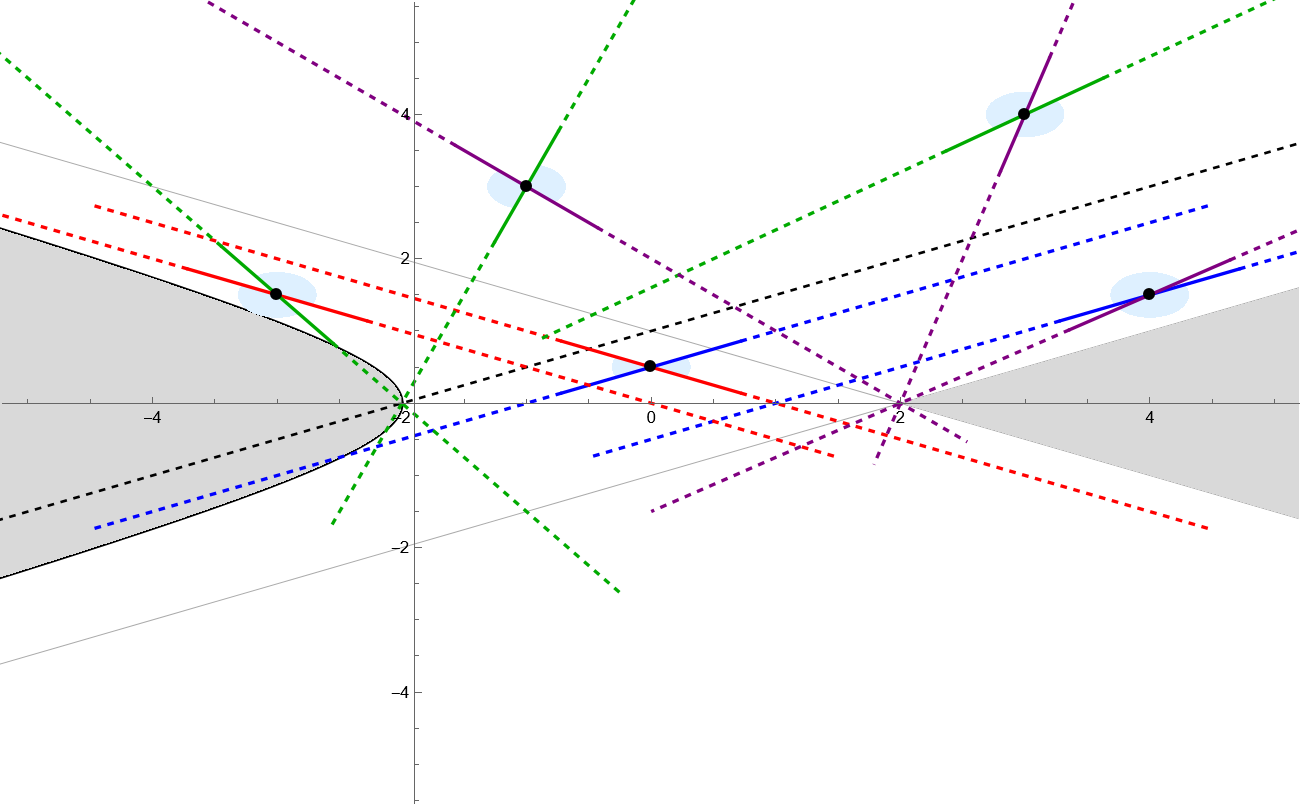}
\caption{Illustration of the cone field at the points $(-3,1.5), (-1,3) ,(0,0.5), (3,4),  (4,1.5) $. Line in  the pencils $\mathcal{P}_0, \mathcal{P}_0', \mathcal{P}_\infty, \mathcal{P}_{\infty}' $ are drawn in purple, green, blue and red respectively. }
\end{figure}

\begin{rem} Observe that the cone field is symmetric with respect to the reflection $\sigma$. We extend the cone field by continuity to $E$. 
\end{rem}

We also put an orientation of the cone field, a tangent vector $v$ at $p$ belongs to $\mathcal{C}^+_p$ (resp. $\mathcal{C}_p^-$) if it belongs to $\mathcal{C}_p$ in the corner of $\mathcal{C}$ containing the vector $(1,0)$ (resp. the vector $(-1,0)$). 

A line through $p$ of one of the pencil $\mathcal{P} \in \{\mathcal{P}_0, \mathcal{P}_0' ,\mathcal{P}_\infty, \mathcal{P}_\infty' \}$ cuts the plane into two half-spaces, which will be denoted $\mathcal{P}^{\epsilon}(p)$ with $\epsilon \in \{ +, -\}$ respectively, with the compatibility condition $\mathcal{C}_p^\epsilon \subset \mathcal{P}^\epsilon(p) $.

For the next proof, we will use the following lemma that will ensure that the intersection between a parabola and a line is transverse. 

\begin{lem} \label{lem_parabola_line_transverse} Let $P$ be a parabola and let $Conv(P)$ the convex set enclosed by $P$. If $L$ is any line which passes through a point in the interior of $Conv(P)$, then $L$ and $P$ intersect transversly. 
\end{lem}

\begin{proof}
Seeing that the line cuts the convex set $Conv(P)$ into two connected components, the intersection between $L$ and $P$ is transverse. 
\end{proof}

\begin{prop} \label{prop_transversal} Let $p \in E\setminus S$ be any point  for which $F(p)$ is distinct from $\bfz,\bfzp$ and take any non vertical line $L$ through $p$ which is a member of $\mathcal{P}_0', \mathcal{P}_\infty, \mathcal{P}_\infty'$. Then the image $F(L)$ is a curve which is transverse to the lines defining the cone field at $F(p)$.
\end{prop}

\begin{proof}
By Lemma \ref{lem_forward_curves}, $F(L)$ is a non-degenerate parabola which is tangent to the indeterminate line at $\bfzp$. 

$\bullet$ Assume that $L \in \mathcal{P}_\infty \cup \mathcal{P}_\infty'$, then $F(L)$ passes through $\bfz, \bfzp$.  
Take a line $C$ of the cone field at $F(p)$.

 Assume that $C$ also passes through $\bfz$ or $\bfzp$. As it also contains the point $F(p)$ which is distinct from $ \bfz, \bfzp$,  Bezout's theorem ensures that $F(L)$ and $C$ are transverse at $F(p)$  and the result holds.  

We thus assume that $C$ is a member of $\mathcal{P}_\infty \cup \mathcal{P}_\infty'$. 
Consider $q$ the intersection between $C$ and $\bIz \setminus \{\bfzp, \bfz\}$. Since $F(L)$ intersects $\Lz$ at $\bfz, \bfzp$, $q$ does not belong to the parabola $F(L)$ by Bezout's theorem. The segment $\bIz$ has its boundaries on the parabola $F(L)$, in particular, the point $q$ is in the interior of the convex set  $Conv(F(L))$.  By Lemma \ref{lem_parabola_line_transverse}, the line $C$ must intersect $F(L)$ transversely at $F(p)$ and the result holds.    
\smallskip

$\bullet$ Assume that $L \in \mathcal{P}_0'$ is not vertical, then the  parabola $F(L)$ passes through $\bfzp$ and another point in $ \bIz$. Take a line $C$ of the cone field at $F(p)$.

 Assume that $C$ also passes through $\binfp$ or $\bfzp$. As $C$ contains a point $F(p)$ distinct from $\binfp,\bfzp$, the line $C$ and $F(L)$ are transverse at $F(p)$ by Bezout's theorem and the result holds.
  
  We are thus reduced to the case where $C \in \mathcal{P}_\infty \cup \mathcal{P}_0$. 
 Denote by $Conv(F(L))$ the convex set enclosed by the parabola $F(L)$. 
By assertion (iv) of Lemma \ref{lem_forward_curves}, $Conv(F(L))$ contains the half-line $l =\sigma(\Li) \cap \bH_- \setminus \{\bfzp \}$ in its interior. 
We claim that any member $C $ of $\mathcal{P}_0 \cup \mathcal{P}_\infty$ through $F(p)$ will intersect $l$. 
Indeed, if $C \in \mathcal{P}_\infty$, then it is a line through $q' \in \bIz\setminus \{ \bfz, \bfzp \}$ and $F(p)$, of slope $1/2$. As $\sigma(\Li)$ is a line of slope $-1/2$,  it intersects $C$ in the lower half plane, hence on $l$. Otherwise, if $C \in \mathcal{P}_0$,  it  does not intersect $\sigma(\Li)$ on the line at infinity and belongs in the interior of the region $R_2 \cup R_3 \cup \sigma(R_2\cup R_3)$. Since the latter  only intersects $\sigma(\Li)$ on the lower half plane. This implies that $C$ and $\sigma(\Li)$ intersect on the lower half plane, hence on $l$.  
 By Lemma \ref{lem_parabola_line_transverse}, as any member $C $ of $\mathcal{P}_0 \cup \mathcal{P}_\infty$ intersects a half line in interior of $Conv(F(L))$, $F(L) $ and $C$ must intersect transversely at $F(p)$.       
\end{proof}

\subsection{Forward invariance of the cone field}

The main result of the section is the invariance of the cone field, stated below. 
\begin{thm} \label{thm_invariant_cone_field}
For any point $p \in E \setminus \Lc$ such that $F(p) \in E$, 
the differential of $F$ maps $\mathcal{C}_p$ into $\mathcal{C}_{F(p)}$. 
Moreover, when $p, F(p) \in E \setminus \Lc$ and $F^2(p) \in E $ , then  $D F^2$ maps $\mathcal{C}_p$ strictly into $\mathcal{C}_{F(p)}$. 
\end{thm}

\begin{proof} Using the fact that $F \circ \sigma = F$ and using the fact that the cone field at $\sigma(p)$ is the image by $\sigma$ of the cone field at $p$, it is enough to prove that the cone field at $p$ is forward invariant for each point $p \in \bH_+$. The structure of the proof is as follows, we need to study  the image of the lines of each of the four pencils $\mathcal{P}_0, \mathcal{P}_0', \mathcal{P}_\infty, \mathcal{P}_\infty'$. This corresponds to the cases labeled A,B,C,D. Since the lines of the cone field at a given point $p$ are chosen depending on the location of that point, this provides some further restriction on the point $p$, and using Proposition \ref{prop_mapping_regions} on its image by $F$. We will first distinguish a subcase for which $F(p)$ is on the upper half plane which occurs in case A, C and D. Ultimately, we exhaust all the possible locations of $F(p)$ and show that the image of a segment or half-line of the four pencils are mapped in the cone field at $F(p)$. This consists in showing that the image belongs to the intersection of two half spaces cut by the lines of the cone field at $F(p)$. We have included some figures in each cases and subcases (except the last one) to illustrate the arguments.

\medskip

We first take a segment $L$ between $p$ and a point on $\bIz$ which determines a boundary direction for the cone field $\mathcal{C}_p^+$ at $p$. 

$\bullet$ Case A (see Figure \ref{caseA} for an illustration): let us first treat the special case where $L \in \mathcal{P}_0^+(p)$. 
By assertion (i) of Lemma \ref{lem_forward_curves}, the image $F(L)$ is a half-line starting at $F(p)$ diverging to infinity, which belongs to $\mathcal{P}_\infty$. 

If $p \in R_2$, then  by  assertion (ii) of Proposition  \ref{prop_mapping_regions}, $F(L) \subset \bH_+$, hence $F(L) \in \mathcal{P}_\infty^+(F(p))$.

Otherwise $p \in E \cap R_3 \subset E \setminus R_2$, then $F(L) \in \bH_-$ using assertions (iii), (iv) of Proposition \ref{prop_mapping_regions}. 
This implies that $F(L) \subset \mathcal{P}_{\infty}^-(F(p))$. 

\begin{figure}
\centering
\includegraphics[scale=0.7]{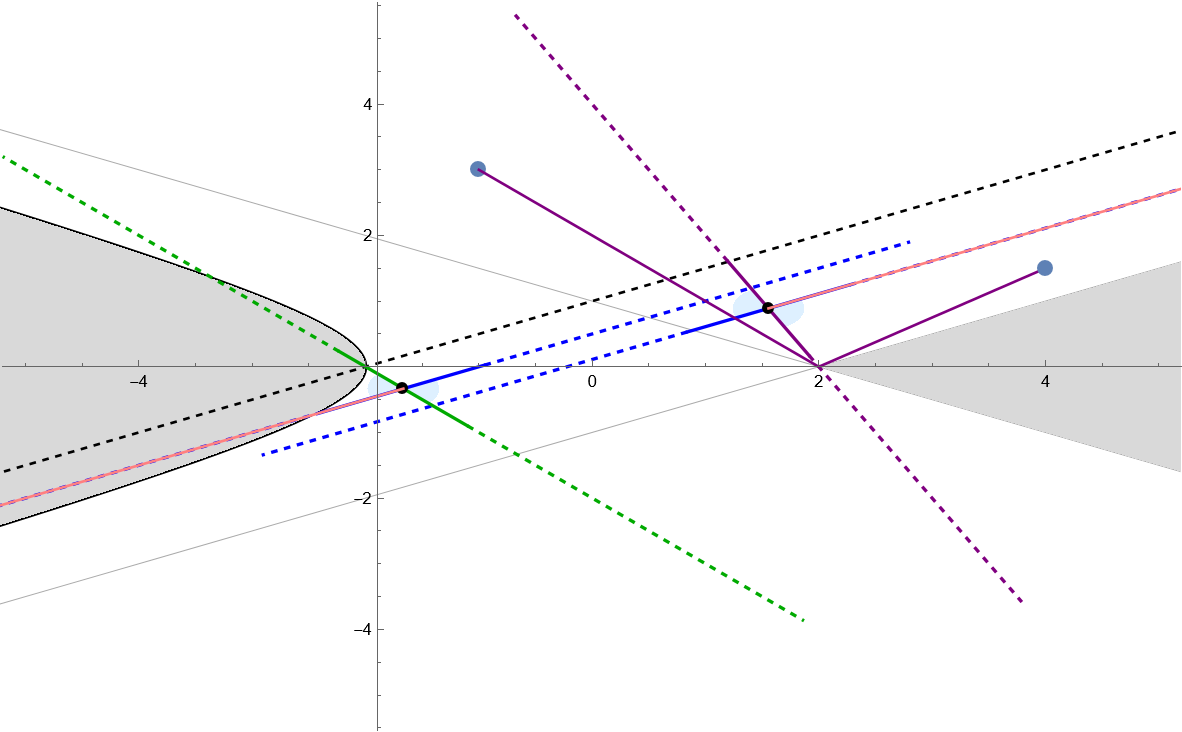}
\caption{\label{caseA}Case A illustration in the case where $p=p_{A,1} = (4, 1.5) \in R_2$, $F(p_{A,1})\simeq (1.55,0.88)$ and  $p=p_{A,2} = (-1,3)$. The segment $L$ in each situation is drawn in solid purple and its image is in solid pink.}
\end{figure}

\smallskip 

$\bullet$ Case B (see Figure \ref{caseB} for an illustration): let us assume that $L \notin \mathcal{P}_0$, which means that $L \in \mathcal{P}_\infty^{'+}(p)$ passes through a point on $\bIz \setminus \{ \bfz\}$ and that $p \in R_4\cup R_5$.
 By assertion (ii) and (iv) of  Lemma \ref{lem_forward_curves}, the image $F(L)$ is an arc  of a parabola $\Gamma$ that starts at $F(p)$ and diverges to infinity in the lower half plane. 
 Moreover, by assertion (iv) of Proposition \ref{prop_mapping_regions}, we have $F(p) \in \sigma(R_2\cup R_3) \subset \bH_-$. We claim that $F(L) \subset \mathcal{C}_{F(p)}^-$. 
Since $F(p) \in \sigma (R_2\cup R_3)$, one of the boundary line for the cone field at $F(p)$ belongs to the pencil $\mathcal{P}_0$. 
Take $C \in \mathcal{P}_0$ a line through $F(p)$, by Proposition \ref{prop_transversal}, $C$ and $F(L)$ are transversal at $F(p)$.  
 As a result, the arc $F(L)$ is either on the left hand side or on the right hand side of $C$ \footnote{The right hand side of $C$ corresponds to the half plane separated by $C$ which contains the horizontal direction $(1,0)$}. By Lemma \ref{lem_forward_curves}, the parabola $\Gamma$ starts at infinity in the lower half plane and passes successively through the points $\bfzp, \bfz, F(p)$ and diverges toward infinity in the lower half plane. By convexity, $F(L)$ is then on the left hand side of $C$.
  
  Assume that the other line $C'$ defining the cone field at $F(p)$ belongs to $\mathcal{P}_\infty'$, which means that $F(p) \in \sigma(R_2^- \cup R_3^-)$. By Lemma \ref{lem_parabola_line_transverse}, $C'$ and $\Gamma$ are transverse at $F(p)$, in particular the two curves intersect at another point $q \neq F(p)$ and the segment $[F(p) q]$ is contained in $Conv(\Gamma)$. By convexity, $F(L)$ is below\footnote{below means it belongs to the half space separated by $C'$ containing the direction $(0,-1)$} the line $C'$. 
  We have thus shown that $F(L)$ is on the left hand side of $C$ and below the line $C'$, hence $F(L) \in \mathcal{C}_{F(p)}^-$. 
  
  Assume now that the other line $C'$ defining the cone field at $F(p)$ belongs to $\mathcal{P}_0'$ so that $F(p) \in \sigma(R_2^+ \cup R_3^+)$.  By Lemma \ref{lem_parabola_line_transverse}, $\Gamma$ and $C'$ intersect transversely at $F(p)$. By convexity, the line $C'$ cuts the parabola into $3$ connected components, a bounded component,  $F(L)$ and another unbounded component. This shows that $F(L)$ is on the left hand side of $C'$. 
  We have thus shown that $F(L)$ is on the left hand side of $C$ and $C'$, hence $F(L) \in \mathcal{C}_{F(p)}^-$.  
In either cases, we have shown that $F(L) \in \mathcal{C}_{F(p)}^-$.   

\begin{figure}
\centering
\includegraphics[scale=0.7]{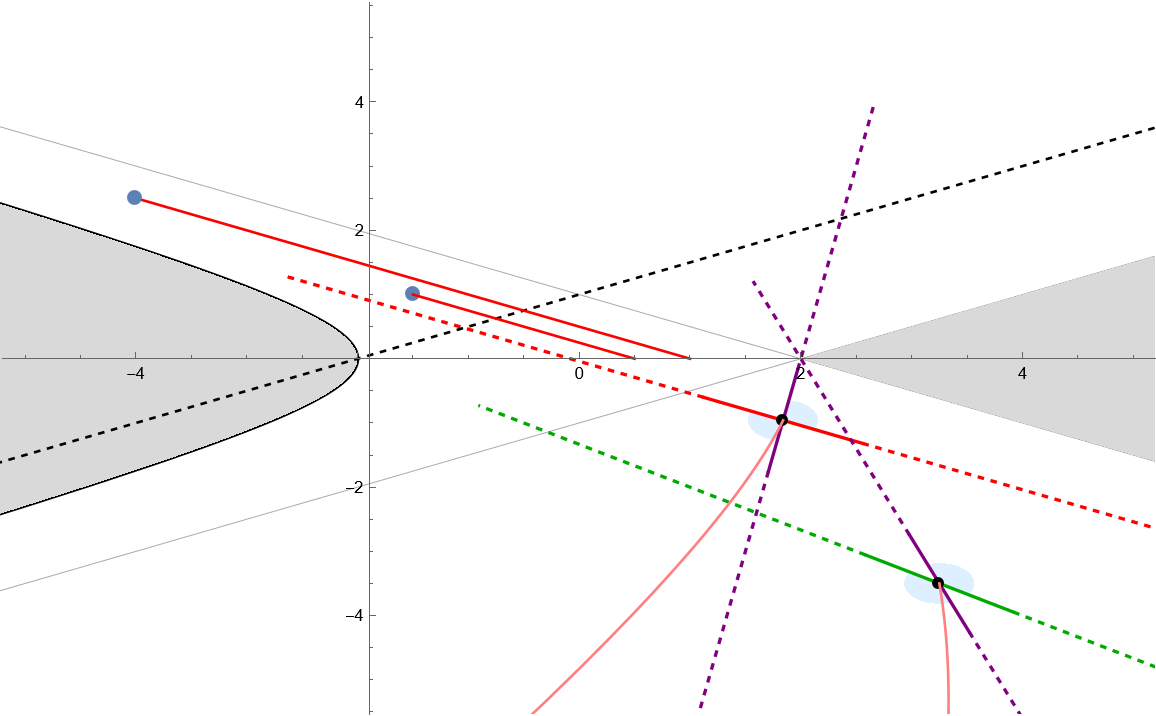}
\caption{\label{caseB}Case B illustration in the case when $p=p_{B,1} = (-4,2.5)$, $F(p_{B,1})=(1.84, -0.96)$ and when $p=p_{B,2} = (-1.5,1)$, $F(p_{B,2})=(3.25, -3.5) $. The segment $L$ in each case is in solid red while its image is in solid pink.}
\end{figure}

  \bigskip
 
We now take a half line $L \subset \bH_+$ starting at $p$ which determines a boundary direction for the cone field $\mathcal{C}_p^+$.

$\bullet$ Case C: let us assume that $L \in \mathcal{P}_\infty$ so that $p \in R_4^-  \cup R_3^- \cup R_2^- $. 
By assertion (ii) of Lemma \ref{lem_forward_curves}, the image $F(L)$ is an arc of a parabola $\Gamma$, which goes from $F(p)$ and ends at $\bfz$. 
 Note that the parabola $\Gamma$ also diverges toward infinity in the lower half plane. 
 
 $\bullet \bullet$ Subcase C.a: assume that $F(p) \in \bH_+$ so that $p \in R_2^-$. Take the  line  $C \in \mathcal{P}_\infty$ defining the cone field at $F(p)$. By Lemma \ref{lem_parabola_line_transverse}, the curves $\Gamma$ and $C$ intersect transversely at $F(p)$, moreover by convexity, the arc $F(L)$ is on the right hand side of $C$. Let us take $C'$ the other line defining the cone field at $F(p)$. Either $C' \in \mathcal{P}_0 $ and $F(p) \in R_2^- \cup R_3^-$, or $C' \in \mathcal{P}_\infty'$ and $F(p) \in R_4^-$. 
 
 Assume that $C' \in \mathcal{P}_0$. By convexity, the segment $[F(p) \bfz]$ belongs to $Conv(\Gamma)$ and lies in the upper-half plane, hence the arc $F(L)$ is above $C'$. We have thus obtained that $F(L)$ is on the right hand side of $C$ and above $C'$, hence $F(L) \in \mathcal{C}_{F(p)}^+$. 
 
 Assume that $C' \in \mathcal{P}_\infty'$. By Lemma \ref{lem_parabola_line_transverse}, the curves $\Gamma$ and $C'$ intersect transversely so there is another point of intersection $q \neq F(p)$. Observe that the parabola $\Gamma$ starts from infinity in the lower half plane, passes successively through $\bfzp, F(p), \bfz$ and diverges toward infinity in the lower half plane.  
  By convexity the segment $[F(p)q]$ is below the arc $F(L)$, hence  $F(L)$ is above $C'$. This shows that $F(L)$ is on the right hand side of $C$ and above $C'$, hence $F(L) \subset \in\mathcal{C}_{F(p)}^+$. 
In both situations of subcase C.a, we have shown that $F(L) \in \mathcal{C}_{F(p)}^+$, as required. 

\begin{figure}
\centering 
\includegraphics[scale=0.7]{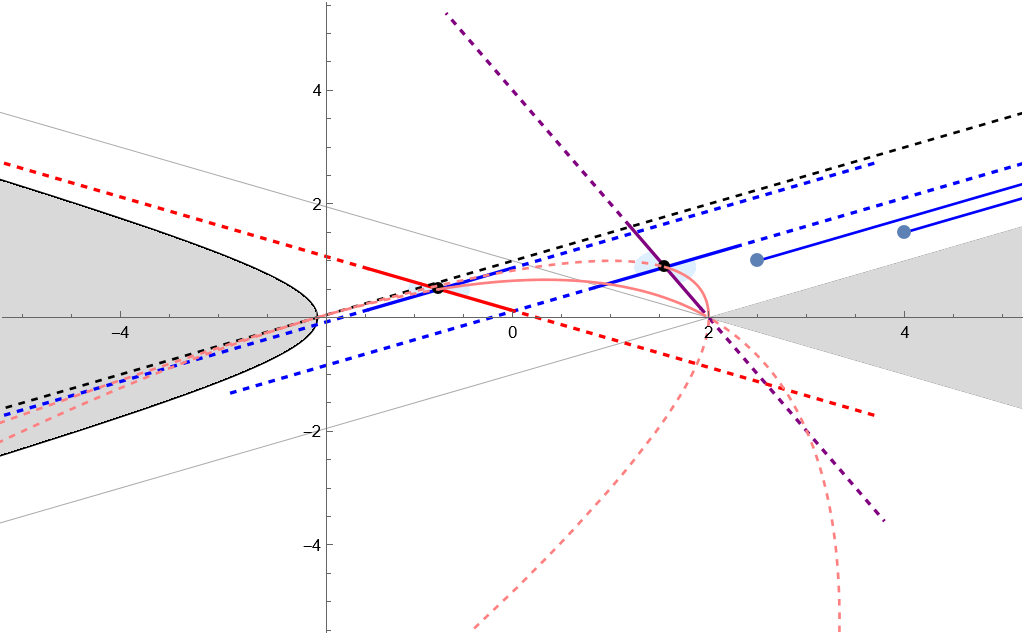}
\caption{\label{caseCa} Subcase C.a illustration when $p = p_{C_a,1} = (4,1.5)$ and when $p=p_{C_a,2} = (2.5,1)$. In both situations, the corresponding half-line $L$ is in solid blue and its image in solid pink, while the rest of the parabola $\Gamma$ is in dashed pink. }
\end{figure}

In the following subcases, $F(p)$ is on the lower half plane so that $p \in R_3^- \cup R_4^-$. In particular, the arc $F(L)$ passes successively to $F(p), \bfzp, \bfz$.  

$\bullet \bullet $ Subcase C.b (see figure \ref{caseCb}): 
assume that $p \in R_3^-$ so that $F(p) \in \sigma(R_4 \cup R_5) $ by Proposition \ref{prop_mapping_regions}. 
Take $C \in \mathcal{P}_\infty$ a line defining the cone field at $F(p)$. By Lemma \ref{lem_parabola_line_transverse}, $C$ and $\Gamma$ intersect transversely at $F(p)$. 
Observe that the subarc of $F(L)$ starting from $\bfzp$ to $\bfz$ intersects $C$ at a point $q \neq F(p)$. This shows that the segment $[F(p) q]$ belongs to $Conv(\Gamma)$, hence the arc $F(L)$ is locally at $F(p)$ on the left hand side of $C$. 
Let us take $C'$ the other line defining the cone field at $F(p)$. 
Since $F(p) \in \sigma(R_4 \cup R_5)$, either $C' \in \mathcal{P}_\infty'$ or $C' \in \mathcal{P}_0'$. By Lemma \ref{lem_parabola_line_transverse}, $C'$ and $\Gamma$ intersect transversely at $F(p)$. We claim that the line $C'$ cannot be in $\mathcal{P}_{\infty}'$. Indeed, the point $p$ belongs to $R_3^-$ which is bounded by a segment contained in $\Li$ going from $\bar \Lf$ to $\Lc$, a segment contained in $\Lc$ ending at $\bfz$ and the segment form $\bfz$ to $(0,1)$, contained in $\bar \Lf$. Since $F(\Li)$ is a parabola passing through $\bfzp$ by assertion (ii) of Lemma \ref{lem_forward_curves}, the image of $R_3^-$ is a region bounded by an arc contained in $F(\Li)$ joining $\bfzp$ and $F(0,1) = (-2,-1)$, by the half-line contained in $\Lf$ going from $(-2,-1)$ toward infinity in the lower half-plane. In particular, this shows that $F(R_3^-) \subset \sigma(R_5^-)$, hence the curve $C'$ belongs to $\mathcal{P}_0'$. 
%
Note that the point $\bfz$ is on the right hand side of $C'$. In particular, this shows that $F(L) $ is locally on the left hand side of $C'$. 
In both situations, $F(L)$ is locally near $F(p)$ on the left hand side of $C$ and $C'$, hence $F(L) \in \mathcal{C}_{F(p)}^-$.

\begin{figure}
\centering
\includegraphics[scale=0.7]{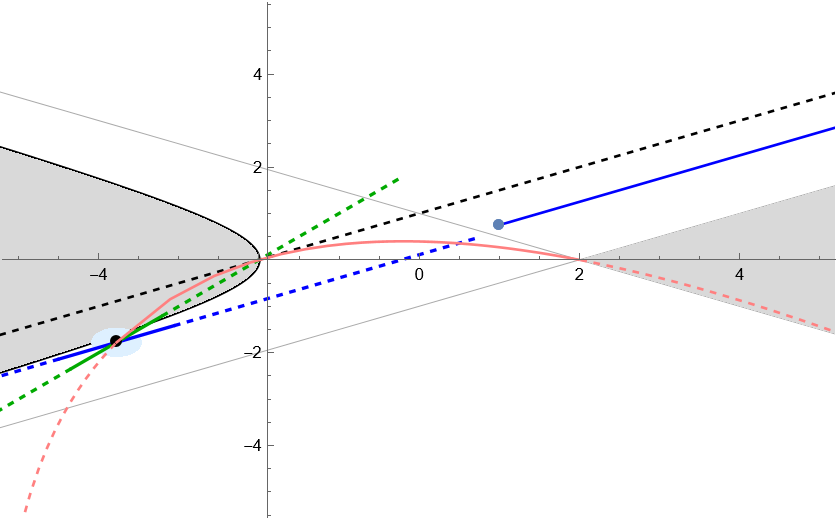}
\caption{\label{caseCb}Subcase C.b illustration with $p=(1,0.75)$ and $F(p) \simeq (-3.7,-1.7)$. The half-line $L$ is in solid blue and its image in solid pink. The rest of the parabola $\Gamma$ is in dashed pink. }
\end{figure}

$\bullet \bullet $ Subcase C.c: 
assume that $p \in R_4^-$ so that $F(p) \in \sigma(R_2 \cup R_3) $ by Proposition \ref{prop_mapping_regions}. 
Take $C \in \mathcal{P}_0$ be a line defining the cone field at $F(p)$. The arc $F(L)$ passes through $F(p)$ and $\bfz$, hence it is either on the left or right hand side of $C$. Since the point $\bfzp$ is on the left hand side of $C$ and since the arc $F(L)$ passes successively to the point $F(p), \bfzp, \bfz$, it lies also on the left hand side of $C$. Take $C'$ the other line defining the cone field at $F(p)$. Note that $C' \in \mathcal{P}_\infty'$ or $C' \in \mathcal{P}_0'$ and that $C'$ and $\Gamma$ are transversal by Proposition \ref{prop_transversal}.

Assume that $C' \in \mathcal{P}_\infty'$ (see figure \ref{caseCc2}) so that $F(p) \in \sigma(R_2^- \cup R_3^-)$. Note that the point $\bfzp$ is on the left hand side of $C'$ and that $F(L)$ intersects $C'$ at another point $q \neq F(p)$. In particular, the subarc of $F(L)$ from $F(p)$ to $q$ passing through $\bfzp$ is on the left hand side of $C'$, hence $F(L)$ is locally near $F(p)$ on the left hand side of $C'$. 
\begin{figure}[h!] 
\centering
\includegraphics[scale=0.7]{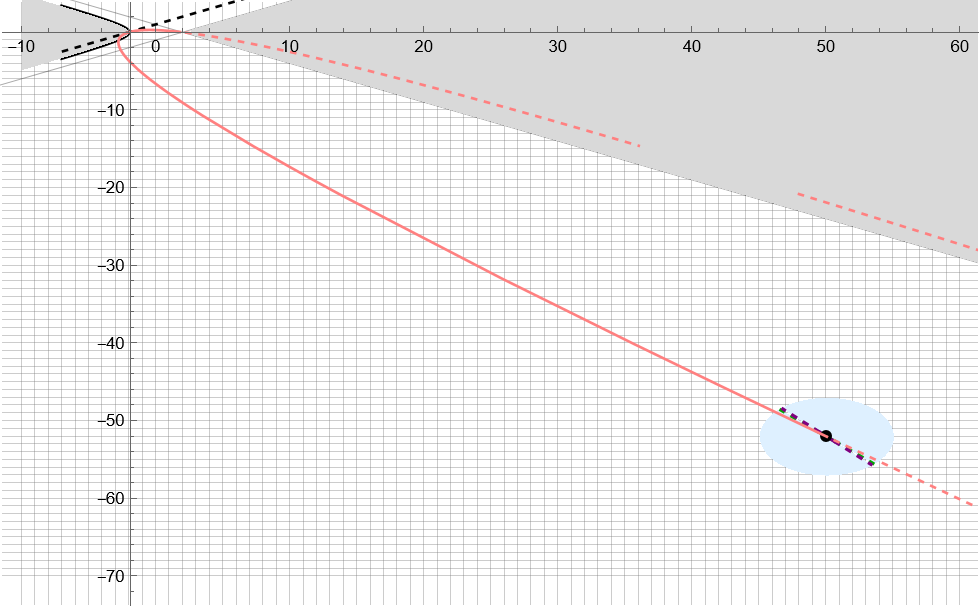}
\caption{\label{caseCc2} Subcase C.c when $p=(-1,0.24)$ and $F(p)\simeq (50.0.8, -52.08)$. The image of the half-line $L$ is in solid pink while the rest of the parabola $\Gamma$ is in dashed pink.}
\end{figure}

Assume that $C' \in \mathcal{P}_0'$ (see figure \ref{caseCc1}) so that $F(p) \in \sigma(R_2^+ \cup R_3^+)$. Note that the point $\bfz$ is on the right hand side of $C'$ while the arc $F(L)$ passes successively through the points $F(p), \bfzp, \bfz$. Hence, $F(L)$ is locally near $F(p)$ on the left hand side of $C'$. 
\begin{figure}
\centering
\includegraphics[scale=0.7]{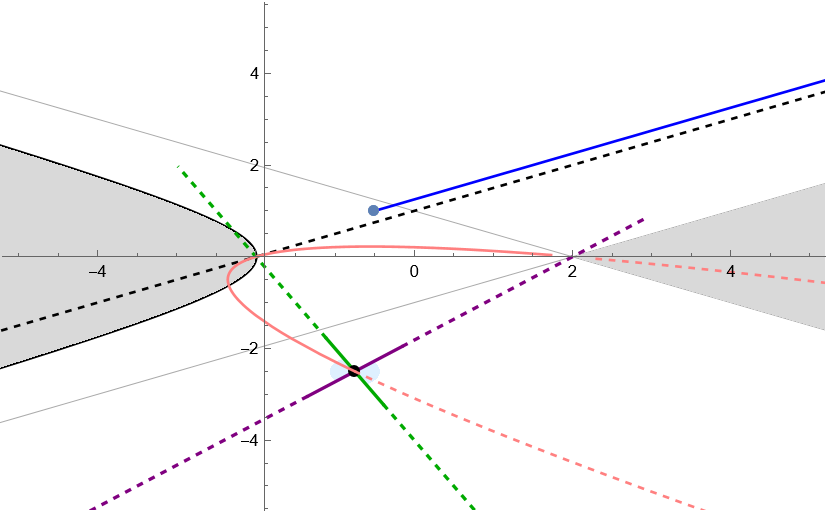}
\caption{ \label{caseCc1} Subcase C.c illustration in the case where $p = (-0.5,1)$ }
\end{figure}

In both situations, $F(L)$ is locally on the left hand side of $C$ and $C'$, hence $F(L) \in \mathcal{C}_{F(p)}^-$. 

$\bullet $ Case D: let us assume that $L \in \mathcal{P}_0'$ so that $p \in R_5 \cup R_4^+ \cup R_3^+ \cup R_2^+$. By assertion (iv) of Lemma \ref{lem_forward_curves}, $F(L)$ is an arc of a parabola $\Gamma$, starting from $F(p)$ going toward a point on $\bIz$. Note that the parabola passes through the point $\binfp$ and diverges toward infinity in the lower half plane.

$\bullet \bullet$ Subcase D.a (see Figure \ref{caseDa}): assume that $F(p) \in \bH_+$ so that $p \in R_2^+$. Note also that since $L$ does not intersect the collapsed line, the image $F(L)$ does not contain the point $\bfzp$. By Proposition \ref{prop_mapping_regions}, we have $F(p) \in R_2^- \cup R_3^- \cup R_4^-$. 
Take $C \in \mathcal{P}_{\infty}$ a line through $F(p)$.  By Proposition \ref{prop_transversal}, the curve $\Gamma$ and $C$ intersect transversely. In particular, $C$ cuts the convex set $Conv(\Gamma)$ into a bounded component and an unbounded one. 
As $\Gamma$ diverges to infinity in the lower half plane and as $\bfzp$ is above the line $C$, it must be in the bounded component. This shows that $F(L)$ belongs to the other component, hence $F(L) $ lies below the line $C$. 
Take $C'$ the other line defining the cone field at $F(p)$. Since $F(p) \in R_4^- \cup R_2^- \cup R_3^-$, this shows that $C' $ belongs to $   \mathcal{P}_0 $ or to $ \mathcal{P}_{\infty}'$. 

Let us show that $F(p) \in R_4^-$ so that $C' \in \mathcal{P}_\infty'$. 
Assume by contradiction that $F(p) \in R_2^- \cup R_3^-$. Take the line $l \in \mathcal{P}_\infty'$ through $F(p)$. Since $\Gamma$ is a parabola through $\binfp, F(p)$, by Bezout's theorem and by convexity, the portion of $\Gamma$ on the right hand side of $C$ is unbounded and intersects the horizontal line $\Lz$ at a point $q \notin \bIz$. However, $\Gamma$ already intersects $\Lz$ at the points $\bfzp$ and another point on $\bIz$, this contradicts Bezout's theorem and we conclude that $F(p) \in R_4^-$ so that $C' \in \mathcal{P}_\infty'$. 

Note that $\bfzp$ is on the left hand side of $C'$ but is not on $F(L)$. Hence by convexity, $F(L)$ is in the on the right hand side of $C'$. We have shown that $F(L)$ is below the line $C$ and on the right hand side of $C'$, hence $F(L) \in \mathcal{C}_{F(p)}^+$, as required. 

\begin{figure}
\centering
\includegraphics[scale=0.8]{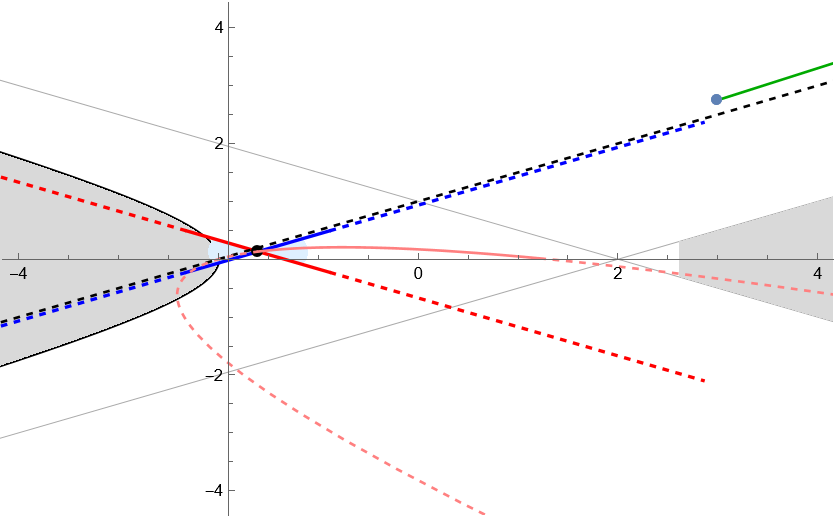}
\caption{\label{caseDa} Subcase D.a illustration where $p=(3,2.75)$ and $F(p) \simeq (-1.60, 0.13)$. The half-line $L$ is in solid green and its image in solid pink. The rest of the parabola $\Gamma$ is in dashed pink.}
\end{figure}

$\bullet \bullet$ Subcase D.b (see Figure \ref{caseDb}): assume that $F(p) \in \bH_-$ so that $p \in R_3^+$ and that the line $L$ is not vertical. 
 By assertion (iii) of Proposition \ref{prop_mapping_regions}, this implies that $F(p) \in \sigma(R_4 \cup R_5)$. Take $C \in \mathcal{P}_{\infty}$ a line determining the cone field at $F(p)$.  
By Proposition \ref{prop_transversal}, the curves $\Gamma$ and $C$ are transverse. In particular, they intersect at another point $r \neq F(p)$. By convexity and since $\Gamma$ diverges to infinity in the lower half plane, the subarc of $\Gamma$ joining $F(p)$ and $r$ is above the line $C$. This shows that $F(L)$ is locally near $F(p)$ above the line $C$. 
Let us take $C'$ the other line defining the cone field at $F(p)$. 
Since $F(p) \in \sigma(R_4 \cup R_5) $,  $C'$ belongs to $\mathcal{P}_\infty'$ or to $\mathcal{P}_0'$, so by Proposition \ref{prop_transversal}, $C'$ and $\Gamma$ are also transverse at $F(p)$.  

Assume that $C'\in \mathcal{P}_\infty'$ so that $F(p) \in \sigma(R_4^-)$. By assertion (v) of Proposition \ref{prop_mapping_regions}, the point $p$ is not in the band $[-2,2] \times \R$. As a result, the half line $L$ does not intersect $\Lc$, hence the arc $F(L)$ does not contain the point $\bfzp$ and is entirely contained in the lower half plane. 
Since the parabola $\Gamma$ passes to the point $\binfp$, the arc $F(L)$ remains on one side of $C'$. Furthermore, this arc cannot be in the right hand side of $C'$. Otherwise one considers   the continuation of arc $F(L)$ along $\Gamma$, after the point $q \in \bIz$, the parabola would continue and go toward $\bfzp$, which is on the left hand side of $C'$. This means that it would intersect $C'$ at a third point, which is impossible. This shows that $F(L)$ is on the left hand side of $C'$.     

Assume that $C' \in \mathcal{P}_0'$ so that $F(p) \in \sigma(R_4^+ \cup R_5)$. By assertion (v) of Proposition \ref{prop_mapping_regions}, this implies that $p \in [-2,2] \times \R$. As a result, the half-line $L$ intersects the collapsed line $\Lc$ and $F(L)$ passes through $\bfzp$. By convexity, the arc $F(L)$ must be locally near $F(p)$ on the left hand side of $C'$ because it ends at a point on $\bIz$ the right hand side of $C'$. 
In both situations, we have shown that $F(L)$ is locally near $F(p)$ above $C$ and on the left hand side of $C'$, hence $F(L) \in \mathcal{C}_{F(p)}^-$.
\begin{figure}
\centering
\includegraphics[scale=0.7]{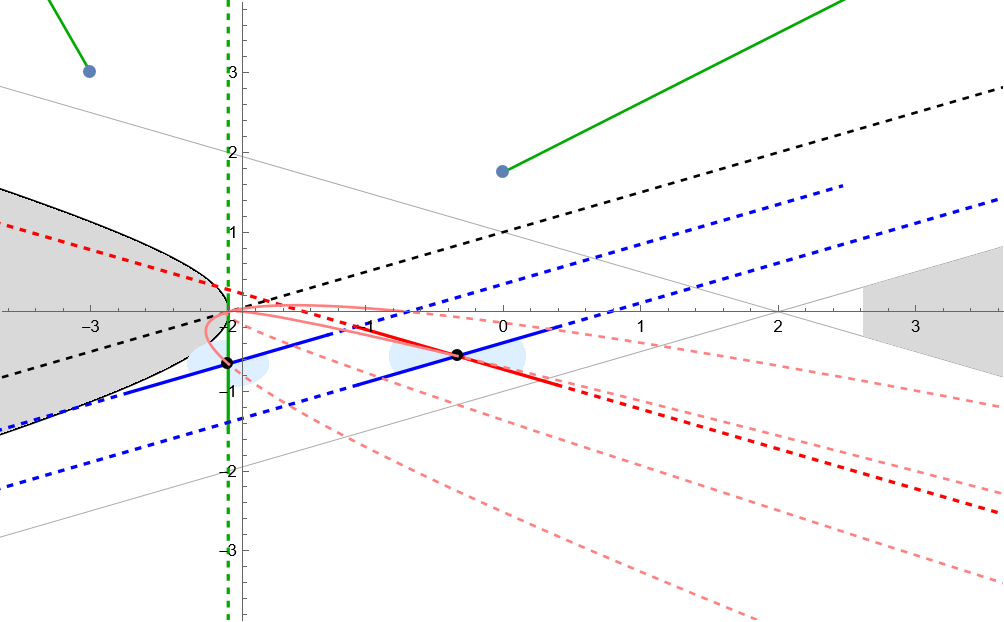}
\caption{\label{caseDb} Subcase D.b illustration in the case where $p= p_{D_b, 1} = (0,1.75)$, $F(p)\simeq (-2, 0.65)$ and where $p=p_{D_b,2} = (-3,3)$ and $F(p)\simeq (-1/3,-5/9)$. The half-line $L$ in each situation is in solid green and its image in solid pink while the rest of the parabola is in shaded pink.}
\end{figure}

$\bullet \bullet$ Subcase D.c: assume that $F(p) \in \bH_-$ and $p$ is on the vertical line through $\bfzp$ so that $L$ is a vertical half-line.  Using Remark \ref{rem_special_vert_bfzp}, we deduce that $F(L)$ is a segment $[F(p) \bfzp]$ contained in $\sigma(\Li)$. It is then immediate that $F(L) \in \mathcal{C}_{F(p)}^-$.
\bigskip

We now argue why $F^2$ maps the cone field strictly into itself. Observe that in case A and D.c, the lines are in the pencil are mapped to boundary lines in $\mathcal{P}_\infty \cup \mathcal{P}_\infty'$ of the cone field. In all the other cases, we showed that the image of the boundary line is transversal to any line defining the cone field. In other word, the boundary vectors are mapped strictly inside the cone field except in cases A and D.c. If one applies $F$ a second time in these two special situations, we have to apply cases C or B for which $F$ maps the cone field strictly into itself. This shows that $F^2$ maps strictly into itself.  
\end{proof}

The above result shows that one can consider an invariant family of curves, defined as follows. 
\begin{defi} \label{def_cL} The subset $\mathcal{L} = \overline{\cup_{n\geqslant 0} F^{-n}(\Lc)}$ defines a family of vertical curves for the cone field. 
\end{defi}

A priori, we do not know yet that $\cL$ is a lamination outside of $\Lz \cup \Linf$. We will establish later on using some expansion properties of the map.

\section{The Chebyshev attractors and Cantor structure}
\label{section_chebyshev}

\subsection{Cantor structure}

We show that any slice of $\mathcal{L}$ by a horizontal line of the form $\{y=c\}$ is a Cantor set. 

We will need  the following simple but crucial observation.
\begin{lem}
	The curve $\Lc(\R)$ is on the boundary of the basin $\cW^s(\pinf)$. 
\end{lem}

\begin{proof}
	For any point $p = (2,c) \in \Lc$ with $c \in \R^+ \setminus {0}$, 
the image of a horizontal line $\{y=c\}$ is a parabola passing through $\bfzp$ and tangent to $\Li$ at $\bfzp$ because $F$ maps $\R^2$ to the half-space below the inderminate line $\Li$.  
In particular, since $\overset{\circ}{R_6} \subset \cW^s(\pinf)$ and is a region delimited by the component of the hyperbola $F^{-1}(\bar \Lf)$ through $\bfzp$, whose tangent at $\bfzp$ is vertical, we deduce that an open interval  on one side of the form $]\alpha, 2[\times \{ c \}$ is mapped by $F$ inside $\overset{\circ}{R_6} $, while $F(p) =\bfzp \in \partial \cW^s(\pinf)$, as required.\end{proof}

Recall that we have subdivided the half-plane into $6$ regions. 

Next, we study the behavior of the family $\mathcal{L}$ near the Chebyshev interval. To that end, we make another  subdivision of $R_4^-$: 
\begin{equation}
R_4^- = R_{4,l}^- \cup R_{4,r}^-,
\end{equation} 
where $R_{4,l}^-$ is the closed region on the left hand side of $F^{-1}(\Lc)$ while $R_{4,r}^-$ is the one on the right hand side. 
We will use the following feature of the leaves of $\mathcal{L}$. 

\begin{prop} \label{prop_leaf_decreasing} Any leaf of $\mathcal{L}$ contained in the region $R_{4,l}^-$ is a graph of the form  $\{ (u(y),y) | y\in \R^+ \}$ where  $u$ is an decreasing function. 
\end{prop}

\begin{proof} Observe that by Proposition \ref{prop_mapping_regions}, the image by $F$ of $R_{4,l}^-$ is $\sigma(R_2)$. 
Take a leaf $C$ of $\mathcal{L}$ going through a point $p$. Take a vertical line $V$ through $p$, we assume that $C$ can be viewed as an arc starting on $\bIz$, passing through $p$ and  going toward $\binfp$. Let us show that the subarc $\gamma$ of $C$ going from $p$ to $\binfp$ is on the left hand side of $V$. 

Assume that $F(p)$ is above the line $\sigma(\Li)$, then observe that the image $F(C)$ is a vertical curve for the cone field $\mathcal{C}$, hence $F(\gamma)$ is a curve going from $F(p)$ to $\bfz$ which lies in the convex region bounded by the segment from $\bfz$ to $F(p)$ and the half-line of slope $-1/2$ starting at $F(p)$ and going to infinity on the upper-half plane. 
By Lemma \ref{lem_forward_curves}, the image $F(V)$ is a line joining $\bfzp$ and $F(p)$. 
In particular, $F(\gamma)$ is above the line $F(V)$, this shows that $\gamma$ is on the left hand side of $V$, as required. 

\smallskip

Assume now that $F(p)$ is below the line $\sigma(\Li)$, the lines through $\bfzp$ are the boundaries of the cone field $\mathcal{C}$. However a leaf of $\mathcal{L}$ is also  vertical for the cone field $\mathcal{C}^{(3)}$, this proves that the same argument shows that the image $F(\gamma)$ lies above the line $F(V)$. Hence $\gamma$ is also on the left hand side of $V$, as required. 

In both cases, we have shown that a leaf of $\mathcal{L}$ contained in $R_{4,l}^-$ is a graph of a decreasing function. 
\end{proof}

Recall that the Chebyshev map is semi-conjugate  to the multiplication by $2$ via the map $\theta \mapsto z=2 \cos(\theta)$.
Let us consider the connected subregion $\mathcal{U}$ contained in $R_{4,l}^-$ delimited by $\Li, \Lz$ and by the branch of $F^{-3}(\Lc)$ going through the point $(-2 \cos(\pi/8),0)$. 

In the following lemma, we obtain some good  control for certain orbits under $F^2$.
\begin{lem}
\label{lem_orbit_left_Lc}  
Consider a point $p= (x,y) \in R_4^-$, the following properties are satisfied. 
\begin{enumerate}
\item[(i)]  Assume that $ 0 \leqslant |x| \leqslant 2$, $ 0 < |y| \leqslant  1/2$ and $ |y/(x-2)| \leqslant 1/3$. Write by   
 $F^2(p)= (x',y')$, then one has $|y'/y| \leqslant  8/9 $. 
\item[(ii)] If $p\in R_{4,l}^-$, then $|y/(x-2)|\leqslant  1/4$.
\item[(iii)] If $p\in \mathcal{U}$, then $|y| < 1/8 < 1/2$ and $|y/(x-2)|< 1/24$.
\item[(iv)] If $p \in \mathcal{U}$, then $|y'/(x'-2)| \leqslant 1/49$.  
\end{enumerate}
\end{lem}

\begin{proof} Let us prove assertion (i).
We compute explicitly $F^2(x,y)$ and write $y'$ as: 
\begin{equation*}
y' = \dfrac{y^2(x^2 - 2 x - 4y^2)}{(x-2)^2}  = \dfrac{x}{x-2} y^2 - \dfrac{4y^4 }{(x-2)^2}.  
\end{equation*}

We get using the fact that $0 \leqslant y \leqslant 1/2$, $|y/(x-2)|\leqslant 1/3$ and $  0  \leqslant x \leqslant 2$:
\begin{equation*}
\left | \dfrac{y'}{y} \right |  \leqslant    \left  |\dfrac{x y }{x-2} \right |  +  \left |  4 y \dfrac{y^2}{(x-2)^2} \right |   \leqslant \dfrac{2}{3} + \dfrac{2}{9} = \dfrac{8}{9},
\end{equation*} 
as required.

\smallskip 

Let us now prove assertion (ii). Observe that the region $R_{4,l}^-$ is convex and bounded by the hyperbola $F^{-1}(\Lc)$ and $\Li$. Since the point $(-2/3,2/3)$ is in the intersection $F^{-1}(\Lc) \cap \Li$, the maximal value of $|y/(x-2)|$ is reached at this point, hence $|y/(x-2)| \leqslant 1/4$, as required. 

\smallskip 

We prove assertion (iii). By Proposition \ref{prop_leaf_decreasing}, any leaf of $\mathcal{L}$ in $\bar{\mathcal{U}}$ is the graph of a decreasing function. 
Since the boundary curve $F^{-3}(\Lc)$ is a leaf of $\mathcal{L}$ and is symmetric with respect to the reflection $\sigma$, it must lie on the left hand side of the vertical line $V$ through $(-2 \cos (\pi/8), 0) \in F^{-3}(\Lc)$, this shows that $\mathcal{U}$ is contained in the convex region $T$ delimited by the lines $\Li, \Lz$ and $V$. Let $q = (x',y') $ be the intersection point $\Li \cap V$. By definition, we have: 
\begin{equation*}
|y''|  = \dfrac{2 - 2 \cos(\pi/8)}{2} = 1 - \cos(\pi/8) \leqslant \dfrac{\pi^2}{2\cdot 8^2} < \dfrac{4}{2^7}= \dfrac{1}{8}.  
\end{equation*} 
This proves the first inequality. For the second, since $\mathcal{U} \subset T$ and by convexity of $T$, one observes that the maximal value of the function $|y/(x-2)|$ in the region $T$  is reached at the point $q$, this gives: 
\begin{equation*}
\left  | \dfrac{y}{(x-2)} \right | \leqslant \dfrac{1-\cos(\pi/8)}{2 + 2 \cos(\pi/8)} \leqslant \dfrac{\pi^2}{3\cdot 2 \cdot 8^2} < \dfrac{1}{24}.
\end{equation*}

\smallskip

We finally prove assertion (iv). Write explicitely 
\begin{align*}
\dfrac{y'}{x'-2} = \dfrac{y^2 (-2 x + x^2 - 4y^2)}{(x-2)^2 (-2 + h(x,y))},
\end{align*}
where 
\begin{align*}
h(x,y) &= \dfrac{x^4-4 x^3-6 x^2 y^2+2 x^2+12 x y^2+8 x+8 y^4-8}{(x-2)^2} \\
 &= \frac{x^4-4 x^3+2 x^2+8 x-8}{(x-2)^2}+ \frac{\left(12 x-6 x^2\right) y^2}{(x-2)^2}+\frac{8 y^4}{(x-2)^2} \\
 &= \dfrac{x^2 -2}{x-2} - \dfrac{6x y^2}{x-2} + \dfrac{8y^4}{(x-2)^2}.
\end{align*}
Since $x \in [-2, -\sqrt{2}]$, $y \leqslant 1/8$ and $|y/(x-2)| \leqslant 1/24$ as $p \in \mathcal{U}$ by assertion (iii), we get: 
\begin{align*}
|h(x,y)|  & \leqslant \left | \dfrac{x^2 -2}{x-2} \right | + |6xy| \left | \dfrac{y}{x-2} \right | + |8y^2| \left  | \dfrac{y^2}{(x-2)^2} \right | \\
  & \leqslant \dfrac{2}{2} + \dfrac{6 \cdot 2 }{8 \cdot 24} + \dfrac{8}{8^2 \cdot 24^2} \\
  & \leqslant 1+ \dfrac{1}{16} + \dfrac{1}{3^2 \cdot 8^3} \leqslant \dfrac{9}{8}.
\end{align*} 
This implies that $|-2 + h(x,y)| \geqslant 7/8$ and we get: 
\begin{align*}
	\left |\dfrac{y'}{x'-2}\right |  \leqslant \dfrac{1}{24^2} \cdot \dfrac{8^2}{7^2} |-2 x + x^2 - 4 y^2| .    
\end{align*}
Since $|y| < 1/8 $ and $x \in [-2, -\sqrt{2}]$, we have: 
\begin{equation*}
		\left |\dfrac{y'}{x'-2}\right |  \leqslant \dfrac{1}{3^2 \cdot 7^2 } (8 + 4/8^2) \leqslant \dfrac{1}{7^2}, 
\end{equation*}
as required.
\end{proof} 

Recall from Proposition \ref{prop_E0_dynamics} that the map $F_X^2$ maps $\Efz$ to itself and its restriction is given by $g_0 \colon l \mapsto l^2/(2-6l^2)$ where $l = y/(x-2)$. This map is conjugate to $z^2 -12$ and the point $l=0$ is a super-attracting fixed point, whose immediate basin is given by $|l| < 1/2$.    
In order to control the forward orbits of points in $R_{4,r}^-$, we shall introduce the  conic region $T_l$ below the lines of slope $\pm l$ through $\bfz$:   
\begin{equation*}
T_{l}= \left \{  (x,y) \in \R^2 | \left | \dfrac{y}{(x-2)} \right | \leqslant l \right \}, 
\end{equation*}
and where  $0 \leqslant l < 1/2 $. 

The next statement allows one to control the behavior under $F^2$ of points on $R_{4,r}^-$. 

\begin{lem} \label{lem_orbit_right_Lc}
For all $0 \leqslant l < 1/2 $, one has: 
\begin{equation*}
F^2(R_{4,r}^- \cap T_l) \subset \sigma(R_4 \cup R_5 \cup R_6) \cap T_{g_0(l)}. 
\end{equation*}
%
\end{lem}

\begin{proof} Take $p \in R_{4,r}^- \cap T_l$. 
Since $p\in R_{4,r}^-$, its image $F(p)$ belongs to $\sigma(R_2\cup R_3)$ by assertion (iv) of Proposition \ref{prop_mapping_regions}. 
Furthermore, since the image $F(R_{4,r}^-)$ is a region bounded by 
$\Lc$ and the line $\Lf$, we deduce that $F(p)$ belongs to $\sigma(R_3)$. Using the fact that $F\circ \sigma = F$ and  assertion (iii) of Proposition \ref{prop_mapping_regions}, this shows that $F^2(p) \in \sigma(R_4\cup R_5 \cup R_6)$.   

\smallskip 

We argue that the image is in $T_{g_0(l)}$. To do so, consider the line $L \in \mathcal{P}_0$ through $p$ and $\bfz$ of slope $-l$.   
By assertion (i) of Lemma \ref{lem_forward_curves}, the image is a line through $\binf$.  Assertion (ii) of Lemma \ref{lem_forward_curves} ensures that $F^2(L)$ is a parabola which is tangent to $\Li$ at $\bfzp$. 
Consider the half-line through contained in $F(L)$ starting from $\Lc$ and going to infinity in the lower half-plane. 
This half-line is contained in $\sigma(R_3)$ so its image by $F$ is a portion $\gamma$ of the parabola $F^2(L)$ contained in $\sigma(R_4\cup R_5 \cup R_6)$.
 Since $\Efz$ is preserved by $F^2$, the slope of the parabola $F^2(L)$ at $\bfz$ corresponds to $g_0(l)$. This shows that the region $R_{4,r}^- \cap T_l$ is mapped by $F^2$ to a region in the lower-half plane above $\gamma$. Since the parabola $F^2(L)$ is tangent to the line through $\bfz$ of slope $g_0(l)$ and lies above it, we deduce that $F^2(R_{4,r}^- \cap T_l) \subset T_{g_0(l)}$, as required.   
\end{proof}

Consider the line $L= \{ y =c \}$. We now show that sufficiently near the collapsed line, there are leaves that are pullback of leaves over half of the Chebyshev interval.  
\begin{prop} \label{prop_good_intervals} For any real number $c> 0$, there exists $h>0$ and some real numbers $r_n>0$ converging to zero and a sequence $k_n$  such that $I_n= [2 , 2 + r_n] \times [c-h, c+h]$ and a sequence of integers $k_n$ such that $F^{k_n}(I_n)$ converges to  $[0, 2] \times \{0 \}$.    
\end{prop}

\begin{proof}
We shall construct the intervals $I_n$ inductively. 
We start with $I_0 = \cup_{t\in [c-h,c+h]} [2,2 + 2 t] \times \{t\}$ which is a region bounded by $\Lf$, $\Lc$ and some horizontal lines. 
Denote by $L_t$ the horizontal line $[2,2+ 2t] \times \{t \}$ for $t\in [c-h,c+h]$. 
Since the line $L_t$ passes through $u_0 \in \Lc$ and through the point at infinity $\pinf$, the image of $F(I_0)$ is an arc of the parabola $F(L_t)$ passing through $\bfzp$ and through the point  $ F(v_0)=(6, 2) \in\Lf$.  
Since $I_0$ is contained in the region $\bar{R_2}$, by assertion (ii) of Proposition \ref{prop_mapping_regions}, $F(I_0)$ is a union of  arcs of the parabolas $F(L_t)$ contained in $R_2^- \cup R_3^- \cup R_4^-$.  In particular, by convexity, these arcs intersects the region $\mathcal{U}$. 
We set $I_1 = F^{-1}(\mathcal{U}) \cap I_0$. By assertion (iv) of Lemma \ref{lem_orbit_left_Lc}, $F^2(\mathcal{U}) \subset T_{l}$ with $l = 1/49$. 
This shows that $J_1=F^3(I_1) \subset T_{7^{-2}}$ and is an arc joining $\bfz$ and the curve $F^{-1}(\Lc)$ contained in $R_4^-$. Up to taking the preimage by the restriction of $F^3$ to $I_1$, we are reduced to finding some subregions $J_n \subset J_1$ whose images by $F^{2n}$ "cover half the Chebyshev interval and converge to it". Precisely, we find some subregions $J_n$ for which the image $F^{2n}(J_n)$ converge toward the segment $[0,2]\times \{ 0 \}$ and such that $F^{2n}(J_n)$ is a union of horizontal curve joining $\bfz$ and a point on $F^{-1}(\Lc)$.  
Note that $\bfz$ is a fixed point for  $F^2$, we set $J_n$ the smallest subregion contained in $J_1 = F^3(I_1)$ and containing $\bfz$ and bounded by $F^{-2n +1}(\Lc)$. 
Since the pullbacks $F^{-2n +1}(\Lc)$ cross the horizontal interval at the points $( 2 \cos(\pi/(2n-2)),0)$ and intersect $J_1$, we deduce that $J_{n+1} \subset J_n$.  
By construction $ F^{2n}(J_n)$ is subregion containing $\bfz$ and a point on $F^{-1}(\Lc)$. 
Furthermore, $F^{2n-2}(J_n)\subset \sigma(R_{4,r}^{-}) \cap T_{g_0^{n-1}(l)}$.
 Indeed, this follows using Lemma \ref{lem_orbit_right_Lc} inductively. Since $l=1/49$, $g_0^n(l)$ converges to zero and this shows that $F^{2n-2}(J_n)$ converges to the interval $[0,2] \times \{ 0\}$. Moreover, observe that the area of $J_n$ converges to zero as the distance between $F^{-1 - 2n}(\Lc)$ and $\bfz$ converge to zero.  
 Taking the pullback $F^{-3}(J_n) \cap I_1$, the result holds. 
\end{proof}

\begin{prop}
	For any $c \in \R\setminus \{0 \}$, the slice $\mathcal{L} \cap \{ y=c\}$ is a Cantor set. 
\end{prop}

\begin{proof} Fix $c \in \R\setminus \{ 0\}$ and denote by $L$ the horizontal line $\{ y =c \}$. It follows from the definition that $L \cap \mathcal{L}$ is a closed set. 
 
Let us first show that $\mathcal{L} \cap L$ has no interior. Assume by contradiction that it contains an open interval $I$. 
 By the previous lemma, any accumulation point of $\mathcal{L} \cap L$ belongs to the boundary $\partial \cW^s(\pinf)$. 
 Moreover, since $\Lc$ and all its pullbacks are vertical curves for the cone field $\mathcal{C}$, the boundary of the basin $\cW^s(\pinf)$ is directed by vertical curves for the cone field. 
  In particular $I$ contains a point $p$ on the boundary $\partial \cW^s(\pinf)$, but since $L$ is a horizontal line for the cone field, any leaf of $\mathcal{L}$ through $p$ is a local curve in $\partial  \cW^s(\pinf)$, hence it must be vertical for the cone field $\mathcal{C}$ and transversal to $L$. This shows that $I \cap \cW^s(\pinf) $ is non-empty, which contradicts the fact that $\mathcal{L} \subset \partial \cW^s(\pinf)$. 
  We have thus proved that $\mathcal{L} \cap L$ has no interior. 
  
  \smallskip 
  
  Let us now prove that $\mathcal{L} \cap L$ is perfect i.e all its points are accumulation points of the set. 
  Take a point $p$ on $\mathcal{L} \cap L $. If $p $ does not belong to $\cup_{k} F^{-k} \Lc$, then by construction, $p$ is accumulated by point on the preimages of $\Lc$, which belong to $\mathcal{L}$. 
  Otherwise, $p$ is one pullback of $\Lc$, up to taking a certain iterate, we are reduced to the proof in the case where $p \in \Lc$. Assume that $p$ is not accumulated by any pullbacks of $\Lc$, then we can take the largest open interval $I \subset L$ containing $p$ and no pullbacks of $\Lc$. 
  By construction, $F^3(I)$ is an arc of a horizontal curve for the cone field $\mathcal{C}$ going through $\bfzp$. 
  Let us show that $F^3(I)$ intercepts a pullback of $\Lc$. By Proposition \ref{prop_good_intervals}, there is a sequence of intervals $I_n$ contained in $L$ on the right hand side of $\Lc$ for which $F^{k_n}(I_n)$ converges to $[0,2]\times \{ 0\}$. Since the length of $I_n$ decreases to zero, we can choose $N$ such that $I_N\subset I$ and we have $I_n\subset I$ for $n\geqslant N$ and  $F^{k_n}(I_n) $ converge to $[0,2 ]\times \{ 0 \}$. In particular, $F^{k_n}(I_n)$ must intercept $F^{-3}(\Lc)$ as it passes through the point $(2\cos(\pi/4),0)$ for $n$ large enough. We have thus shown that $I_n$ is contained in $I$ and intercepts a pullback of $\Lc$, which is a contradiction with the definition of $I$. We have thus shown that there are no open intervals in $L$ around $L \cap \Lc$ which do not intercept any pullback of $\Lc$. Hence $L \cap \mathcal{L}$ is a perfect set.   
\end{proof}

\subsection{A subset of positive area}

The orbit of a point under $F$ is a priori hard to track once it intersects the indeterminacy locus. However, an orbit can behave well on the blow-up $X$ of $\bP^2$ so we introduce the following subset $U$, whose lift to $X$ is far from the indeterminacy point on $\Efz$. 

Take the image $\gamma'$ of $\Li$ by $F^2$, it locally passes through $\bfz$ and is tangent the line $\Lz$ at this point. Parametrize locally $\gamma'$ as the graph of a function $g(x)$.  
Consider $U$ the subset
\begin{equation}
U := \{ (x,y) | |y| \leqslant \min( g(x), \epsilon , |x+2|/2 ) \}.
\end{equation}
Note that both $\Li$ and $\sigma(\Li)$ have the same image by $F^2$ since $F\circ \sigma = F$. 
Observe that $U$ is not forward invariant, but we will remove a subset which will be forward invariant. 
To that end, we consider the conic region above the point $\bfzp$
\begin{equation} 
W= \{ (x,y) | |y| \geqslant |x+2|/2 \}. 
\end{equation}

\begin{prop} \label{prop_forward_invariant_U} The set  $U \setminus \cup_{n\geqslant 1} F^{-n}(W)$ is forward invariant by $F^2$.  
\end{prop}

\begin{proof} We first establish that $F^2$ maps the region below $W$ in $U$ inside $U$. Take a vertical segment $I$ from a point $p$ on $\Li$ ending at a point $q$ on $\bIz$ near $\bfzp$. Its image by $F$ is contained in  the line through $\bfzp$ and through the point $F(p)$ on the parabola  $F(L_i)$.
More precisely, it is a half line going from $F(p)$ diverging to infinity in the lower half plane toward $F(q)$. 
The second iterate $F^2(I)$ is an arc of a  parabola going from a point $F^2(p)$ on $\gamma'$ to a point on the Chebyshev interval $F^2(q)$, close to $\bfz$.  Since $F$ restricted to the upper half plane is injective onto its image, we deduce that the restriction of $F$ to the filled triangle $T'$ delimited by the segment $I$, the segment $[\bfzp, p]$ and $[\bfzp,q]$ is also injective. 
Since $T'$ is mapped to the lower half plane. We deduce that $F^2$ restricted to $T'$ is also injective. This shows that the arc $F^2(I)$ does not cross $\gamma'$ other than at the point $F^2(p)$. 
Hence the portion below $W$ is mapped by $F^2$ inside $U$.    
Removing the successive preimages of $W$ yield a forward invariant subset. 
\end{proof}

A priori, the pullbacks $\cup_{n \geqslant 0} F^{-n}(W)$ could fill in the whole space $U$. The next step is to estimate the area of these pullback in $U$. 
To that end, we first compute the first pullback $F^{-2}(W)$ by $F^2$. This lemma is somewhat surprising if one compares to the behaviour of the skew-product maps $(x,y) \mapsto (x^2 - 2, y^2)$ for which the pullback $F^{-2}(W)$ is a conic region. 

\begin{lem} \label{lem_first_pullback} The pullback $U\cap F^{-2}(W)$ is a region bounded by the vertical line of equation $ x = 0 $ and an algebraic curve whose power series expansion at zero is  of the form $$x(y) = -2 y^2  + \sum_{n\geqslant 2} (-1)^n 2^{n-1} y^{2n}.$$  
\end{lem}

\begin{proof}
Since the boundary of $W$ consists of the curves $\Li$ and $\sigma(\Li)$. We compute explicitly the pullback of the lines $\Li$ and $\sigma(\Li)$. 
Observe that the line $x=0$ passes through the point $(0,-1) \in \Lf$ and by assertion (iii) of Lemma \ref{lem_forward_curves}, its image is a line through $\bfzp$ and $F(0,-1) = (-2,-2)$. In particular, the image is  a vertical line through $\bfzp$. 
Applying again assertion (iii) of Lemma \ref{lem_forward_curves}, we deduce that the second iterate $F^2( \{ x=0\})$ is a line passing through $\bfzp$ and $F^2(0,-1)= (0,-1)$. This is exactly the line $\sigma(\Li)$. 

We now compute the preimage of  $\Li = \{ x+2 - 2y =0 \}$. 
$$ ((x+2) - 2  y) \circ F^2  = \frac{\left(x^2-2 x-4 y^2\right)^2}{(x-2)^2}.  $$
By a formal computation (see Lemma \ref{lem_expansion_explicit_1}), we obtain the required expansion of the curve of equation $x^2 - 2 x -4y^2=0$ near $(0,0)$. 
\end{proof}

\begin{lem} \label{lem_vertical_parabola} 
There exists a constant $C_-$ such that for any point $p\in U$ and for any $\epsilon \leqslant \epsilon_0$ small enough.  If  $\gamma$ is any vertical curve through $p$ for the first iterate of cone field $\mathcal{C}^{(1)}$. Then $\gamma$ is  in between a vertical line and a curve of the form $ q_-  -x = C_- y^{3/2} $ and $ q_+   =x $ where $q_{\pm} \in \bIz$.  
\end{lem}

\begin{proof}
The cone field $\mathcal{C}^{(1)}$ is the first pullback of the cone field at $p$.
Recall that $U$ is a region below the line $\Li$ so that $F(p)$ lies under the parabola $F(\Li)$. Since $\Li$ intersects the line $\bar \Lf$ at the point $M=(0,1)$ whose image is $F(M) = (-2,-2)$, we deduce that the half line from $F(M)$ to $\binfp$ going to infinity in the lower half plane is in the convex envelope of $F(\Li)$. 
Since particular, we deduce that the domain $F(U)$ lines below this half line, hence $F(U)$ is contained in the region $\sigma(R_2^+ \cup R_3^+)$, i.e below the line $\sigma(\Li)$. 
Observe that since $F(p)$ is below the line $\sigma(\Li)$, the lines delimiting the cone field $\mathcal{C}$ at $F(p)$ are by definition the lines through $\bfz,\bfzp$. The pullback by $F$ of the line through $\bfzp$ and $F(p)$ is a vertical line through $p$ by assertion (iii) of Lemma \ref{lem_forward_curves}. Similarly, the pullback $\gamma'$ by $F$ of a line through $\bfz$ and $F(p)$ is a degree two curve, which intersects the line at infinity at $\binf, \binfp$. In particular, it is a hyperbola through $p$ which admits a vertical tangent at $\bIz$. We thus control the local curvature of $\gamma'$ at the intersection point with $\bIz$. 
The line through $F(p)$ and $\bfzp$ is of the form: 
\begin{equation*}
(x-2) - \lambda y =0,
\end{equation*}
where $\lambda \in (-2,2)$. 
In particular, using \eqref{eq_def_F}, the equation of $\gamma'$ is given by: 
\begin{equation*}
\dfrac{4 y^2}{x-2} =  (x - \lambda).
\end{equation*}
Since $p \in U$, it lies in  certain parabolic region near $\bfz$, hence  $|y| \leqslant   \eta(x-2)^2$ and we get: 
\begin{equation*}
|x- \lambda| \leqslant \dfrac{4 \sqrt{\eta}y^2}{\sqrt{|y|}} = 4 \sqrt{ \eta} y^{3/2},
\end{equation*}
as required.
\end{proof}

We will now apply some distortion estimates, first we compare with the distortion for the Chebyshev map via the following lemma. 
Denote by $T$ the Chebyshev map. 

\begin{lem} \label{lem_Chebyshev_derivative} Let $x_n \in [-2,2]$ be a point such that $T^n(x_n)\in [-\epsilon, \epsilon]$. Then one has: 
\begin{equation}
|(T^{n-1})'(x_n)| \geqslant \sin(\pi/4-\epsilon)\dfrac{2^{n-1}}{\sqrt{1 - x_n^2/4}}.
\end{equation}
In particular, there exists a constant  $C(\epsilon)>0$ such that for all $n$, $|(T^{n-1})'(x_n)| \geqslant C(\epsilon) 2^{n-1}$. 
\end{lem}

\begin{proof}
Write $x_n = 2 \cos(\theta_n)$. Since the Chebyshev map is semi-conjugate via $\theta \mapsto z=2 \cos(\theta)$ to the doubling map on the circle and using the fact that $2\cos(2^p \theta_n) = \sin(2^{p+1}\theta_n)/\sin(2^p \theta_n)$, we have: 
\begin{multline*}
|(T^{n-1})'(x_n)| = \left  | \prod_{p=0}^{n-2} T'(T^p(x_n)) \right |  = \left  | \prod_{p=0}^{n-2} 4 \cos(2^{p} \theta_n) \right | = 2^{n-1} \left  | \prod_{p=0}^{n-2} \dfrac{\sin(2^{p+1}\theta_n)}{\sin(2^p\theta_n)} \right | \\ = \left | \dfrac{2^{n-1}\sin(2^{n-1} \theta_n)}{\sin(\theta_n)} \right |.
\end{multline*}

Since $T^{n}(x_n) \in [2 \cos(\pi/2)-\epsilon, 2\cos(\pi/2)+\epsilon]$, we get $2^{n} \theta_n \in  [\pi/2-\epsilon, \pi/2+\epsilon]$ hence $2^{n-1}\theta_n \in [\pi/4 - \epsilon, \pi/4 + \epsilon]$ and we obtain the required formula. 
The lower bound follows directly from the inequality $|\sin(\theta_n)| \leqslant 1$. 
\end{proof}

We can now estimate the rate of convergence toward the horizontal line. 
\begin{prop} \label{prop_good_decay}There exists two constants $C >0$ such that for all $n\geqslant 3$,  if $(x_n, y_n) \in  U$ with $F^2(x_{n+1}, y_{n+1}) = (x_n , y_n) \in U$, then one has for all $k\leqslant n-3$: 
\begin{equation}
|y_{n-k}| \leqslant (C  y_n)^{\alpha^{k+1}},
\end{equation} 
where $\alpha=3/2$.
In particular, for any point $p = (x_n , y_n)$ on $F^{-2n}(W)\cap U$ whose orbit under $F^2$ stay in $U$, their images converges super exponentially fast  toward $\bIz$.  
\end{prop}

\begin{proof}
Let us show the statement for $k=n-3$, as the proof is similar for all other cases.
We have using \eqref{eq_G} that: 
\begin{equation*}
|y_{p-1}| = |y_p^2| \left | \dfrac{x_p}{x_p -2}  - \dfrac{4 y_p^2}{(x_p - 2)^2}  \right |.
\end{equation*}
Recall that $U$ is a region bounded by the parabola through $\bfz$ given by $y = \pm \eta (x-2)^2$ where $\eta > 0$ is some constant. In particular, we have
$\sqrt{|y_p|} \leqslant \sqrt{\eta} |x_p -2|$ for all $p$. Hence, we obtain a constant $C>0$ such that for any $p \leqslant n$, one has: 
\begin{equation*}
|y_{p-1}| \leqslant C |y_p|^{3/2} 
\end{equation*}
 Raising at the power $\alpha^{p-3}$ with $\alpha=3/2$ and taking the product, one obtains a uniform constant $C'= C^3>0$ such that for all $n$: 
 \begin{equation*}
\left  | \dfrac{y_3}{y_n^{\alpha^{n-2}}} \right | =  \left |  \prod_{p=4}^{n} \dfrac{y_{p-1}^{\alpha^{p-3}}}{y_p^{\alpha^{p-2}}} \right |  \leqslant \prod_{p=4}^{n} C^{\alpha^{p-3}} = C^{3 (\alpha^{n-3}-1)} \leqslant  (C^3)^{\alpha^{n-2}},
\end{equation*}
as required.  
\end{proof}

We can now compare the horizontal derivative between the orbit of point on a leaf $\gamma_x$ and the one of $x$. This bound is not satisfied everywhere as it is violated above the critical point, but we will ensure its validity below a parabola. 

\begin{prop} \label{prop_distortion} 
Fix $0<\alpha<1$.
Consider any point  $p \in U$ and any integer $n$.
Let us denote by  $p_k = F^{2k}(p) = (x_k,y_k)$ for all $k \leqslant n$ the forward orbit of $p$. We consider the pullback $\gamma_{p_k,n}$ by $F^{2n-2k}$ of the vertical line through $F^{2n}(p)$ and consider $(q_k,0)$ the intersection of $\gamma_{p_k,n}$ with $\bIz$. 
If we assume that for all $k \leqslant n$, $p_k \in U$ and that for all $k\leqslant n-1$, $|x_k| \geqslant \alpha^k$.  
Then one has for all $k\leqslant n-1$:
\begin{equation}
e^{-2 \alpha^k}\leqslant  \left | \dfrac{ \partial_x (x\circ F^2)(p_k)}{2 q_k} \right |  \leqslant e^{  \alpha^k}.  
\end{equation}   
\end{prop} 

\begin{proof} Fix $\epsilon'>0$. 
Using \eqref{eq_G}, we have: 
\begin{equation} \label{eq_derivative}
\partial_x (x\circ G)(x,y) = 2x  - \dfrac{6x y^2}{(x-2)^2} - \dfrac{8 y^4}{(x-2)^3}.  
\end{equation}
Since $p_k$ belongs to $U$ it is below a parabola of the form $y = \eta (x-2)^2$ for some $\eta >0$ so we have $|y_k|/(x_k-2)^2 \leqslant \eta$. This gives: 
\begin{equation} \label{eq_error_1}
\left | \dfrac{6 x_k y_k^2}{(x_k-2)^2} - 8 \dfrac{y_k^4}{(x_k-2)^3} \right | \leqslant C |y_k|,
\end{equation}
where $C = 6 \eta + \eta^{3/2} 8$.
Using Proposition \ref{prop_good_decay}, we get $|y_k| \leqslant (C' M_{\epsilon'})^{1.5^k}$ for some constant $C'>0$. Shrinking $M_{\epsilon'}$ appropriately, we can assume that $C' M_{\epsilon'} < 1$ to get a super-exponential decay. 
Moreover, using the fact that  $|x_k| \geqslant \alpha^k$ for all $k \leqslant n-1$, we obtain that for $N \leqslant k \leqslant n-1$ large enough $|y_k| \leqslant  |x_k|^2$, hence $|x_k| \geqslant \sqrt{|y_k|}$.
 Moreover, since $\gamma_{p_k,n}$ is a vertical curve for the cone field $\mathcal{C}^{(1)}$,  Lemma \ref{lem_vertical_parabola} shows that each $\gamma_{p_k,n}$ is in between two curves of the form $ q_{k,-} - x = C_- y^2$ and $q_{k,+} = x$ where $q_{k, \pm } \in \bIz$.  
Since the curve $\gamma_{p_k, n}$ is in between these two curves, 
we conclude that the intersection point $(q_k,0)$ between $\bIz$ and $\gamma_{p_k,n}$ is in the segment $[q_{k,-}, q_{k,+}]$. 
This gives:
\begin{equation*}
 x_k + C_- y_k^{3/2} = q_{k,-} \leqslant q_k \leqslant q_{k,+ } =x_k.
\end{equation*}
By Proposition \ref{prop_good_decay}, $|y_k|$ decays super-exponentially fast toward zero, hence for $k\geqslant N$ large enough, we have $C_{-} y_k^{3/2} \leqslant \alpha^k/2 \leqslant  |x_k|$, so we obtain: 
\begin{equation} \label{eq_bound_qk}
|q_k | \geqslant \alpha^k - C_- y_k^{3/2} \geqslant \alpha^k/2.
\end{equation}
 We obtain by combining with \eqref{eq_error_1}, \eqref{eq_derivative},\eqref{eq_bound_qk} together with the fact that $(C+4)|y_k|/\alpha^k \leqslant   (C+4)(C' M_{\epsilon'})^{1.5^k}/\alpha^k  \leqslant \alpha^k$ 
 that for $n-1 \geqslant k \geqslant N$: 
\begin{multline*} \label{eq_upper_bound_distortion}
\Delta=\left  | \dfrac{\partial_x (x\circ G)(p_k)}{2q_k} \right | \leqslant  \dfrac{|2q_k|+  |(C+4) y_k|}{|2q_k|} = 1 + \dfrac{(C+4)|y_k|}{|2q_k|} \leqslant 1 + \dfrac{{(C+4) |y_k|}}{\alpha^k}  
 \leqslant 1 + \alpha^k \leqslant e^{\alpha^k}. 
\end{multline*}
Similarly, using the inequality $1 -  x \geqslant e^{-2  x} $ for all $0 \leqslant x \leqslant \log 2 / (2 )$,  we also get a lower bound: 
\begin{equation}
\Delta \geqslant \dfrac{||2q_k| - (C+4) |y_k| |}{|2q_k|} \geqslant 1 - \dfrac{(C+4)|y_k|}{|2q_k|} \geqslant 1 - \dfrac{(C+4) |y_k|}{\alpha^k} \geqslant 1 - \alpha^k  \geqslant e^{ -2 \alpha^k},
\end{equation}
provided the inequality $\alpha^k \leqslant \log 2 /2 $ holds true.
Since $\alpha<1$, this is always satisfied for $k$ large enough  and this gives the lower bound, as required.
\end{proof}

\begin{thm} \label{thm_positive_measure_subset} The subset $ U \setminus \cup_{n\geqslant 1} F^{-2n}(W)$ has positive Lebesgue measure.  Moreover, for any $\epsilon>0$, there exists $t >0$ such that for any $|h| < t$, the length of $ L_h \cap U\setminus \cup_{n\geqslant 1} F^{-2n}(W)$ is larger or equal to $4-\epsilon$. 
\end{thm}

\begin{proof}

For all $p \in L_h$, we denote by $(x_n,y_n)$ the coordinates of $F^{2n}(p) = p_n$ and we set:  
\begin{equation}
\xi_n = \{ p \in L_h \ | \ \forall k\leqslant n-1,  |x_k| \geqslant  \alpha^k, p_k \in U \text{ and } |x_n|^2< |y_n|   \}.
\end{equation}
Geometrically, the points in $\xi_n$ land after $n$ iterates on a region above the parabola $y= x^2$. Let us denote by $\tilde{W}$ this region, defined as $\tilde{W} = \{(x,y) \ | \ |y| \geqslant x^2 \}$. 
Since $\tilde{W}$ contains $F^{-2}(W)$, we get that $U \setminus \cup_{n\geqslant 1} F^{-2n}(\tilde{W}) \subset  U \setminus \cup_{n\geqslant 1} F^{-2n} (W)  $. In particular, Proposition \ref{prop_good_decay} shows that $U\setminus \cup_{n\geqslant 1} F^{-2n}(\tilde{W})$ is contained in the basin of $\bIz$.  

\smallskip 

It remains to show that $ L_h\cap U \setminus \cup_{n} \xi_n \subset \cW^s(\bIz)$ has positive length. 
To that end, we estimate the length of $\cup_{n\geqslant 0}\xi_n$.

\medskip

We first construct some approximate stable manifolds, then estimate the horizontal expansion of $F^{2n-2}$. 
For each $p \in \xi_n$, for each iterate $k\leqslant n$, we consider the pullback $\gamma_{F^{2k}(p),n}$ of the vertical line through $F^{2n}(p)$ by $F^{2n-2k}$. 
For each $k \leqslant n-1$, $\gamma_{F^{2k}(p),n}$ is a vertical curve through $p_k$ for the first iterate of the cone field $\mathcal{C}^{(1)}$, which intersects $\bIz$ at the point $(q_k,0)$ and we are in the setting of the Proposition \ref{prop_distortion}. 
Observe also that the base point $(q_n,0)$ lands in a neighborhood of the critical point of radius $\epsilon$ since $|x_n|^2 < |y_n| \leqslant \epsilon^2$ by assumption and since $y_k$ decays super-exponentially fast toward zero. 
Lemma \ref{lem_Chebyshev_derivative} then gives: 
\begin{equation}
\left  | \prod_{k=0}^{n-1} 2 q_k \right | \geqslant C(\epsilon) 2^{n-1} 
\end{equation}
 Combining with Lemma \ref{lem_Chebyshev_derivative}, we get:
\begin{equation} \label{eq_expansion}
| \partial_x (x\circ F^{2n-2})(p)| \geqslant e^{-2 \sum_{k=0}^{n-1} \alpha^k} \left | \prod_{k=0}^{n-2} 2q_k \right  | \geqslant C_1  C(\epsilon) 2^{n-1} = C'' 2^{n-1},    
\end{equation}
where $C_1 = e^{-2/(1-\alpha)}$ and $C''= C_1 C(\epsilon)$. 
We have thus shown that $F^{2n-2}$ expands in the $x$-direction uniformly by a factor $2^{n-1}$, hence $F^{2n}$ does the same as well.  

\medskip

The portion $\xi_n$ of $L_h$ enters the portion $\tilde{W}$ above the parabola after exactly $2n$-steps, so $F^{2n}(\xi_n) \in \tilde{W}$. 
Let $\pi$ be the horizontal projection and consider $\pi(F^{2n}(\xi_n))$. 
Using the uniform expansion of $x\circ F^{2n}$ given in \eqref{eq_expansion} together with the fact that $F^{2n}(\xi_n) \subset \tilde{W} \cap U$, we get:
\begin{equation*}
C'' 2^n l(\xi_n) \leqslant l(\pi (F^{2n}(\xi_n)))  \leqslant l(\pi(\tilde{W}\cap U)).
\end{equation*}
This shows that $l(\xi_n) \leqslant K 2^{-n}$ decays exponentially fast where $K = l(\pi(\tilde{W}\cap U))/C''$. 
In particular, we have:
\begin{equation}
l(\cup \xi_n) \leqslant \sum_{n\geqslant 0} 2 \dfrac{l(\pi(\tilde{W}\cap U))}{C''}.
\end{equation}
Recall that $U$ depends on the choise of a height $\epsilon$ and as $\epsilon$ goes to zero, the length of $\pi(\tilde{W} \cap U)$ converges toward zero. In particular, this shows that as $L_h$ approaches the length of  $L_h\cap U\setminus \cup_{n\geqslant 0} \xi_n$ approaches the horizontal line, the length of $L_h \cap U$  converges to $4$, as required.

This shows that $U\setminus \cup_{n\geqslant 1} F^{-2n}(W)$ has positive measure using Fubini's theorem.
\end{proof}

At this stage, one cannot conclude that $\cL$ has positive area, namely because even though we have shown that $U \setminus \cup_n F^{-n}(W)$ is invariant and has positive area. There are some natural stable manifolds through each point $p$ of this set. We will show later on that these stable manifold belong to $\cL$, but to claim this, we will need to show that a subset of $\cL$ is a lamination.

\subsection{Kobayashi expanding region and laminar structure}
\label{section_kobayashi}

We define the elements of the family $\mathcal{L}$ which intersect a point on the Chebyshev interval whose orbit avoids $\bfz$ as elements of $\mathcal{L}^*$.

Consider the complex manifold $H(\C) = \bP^2(\C) \setminus (\Lf \cup \bar \Lf \cup \Lz \cup \Linf)$, the interior $\mathcal{D} \subset \R^2$ of $R_2 \cup R_3 \cup \sigma(R_2 \cup R_3)$ and the punctured sphere $\mathfrak{h} = \Linf(\C) \setminus \{ \pinf,\binf, \binfp\} $.  
There is a  natural radial projection $\pi_0$ from $\bfz$ to the line at infinity whose fibers are elements of the pencil $\mathcal{P}_0$ of lines through $\bfz$. We view the quotient $N(\mathcal{P}_0) =TH(\C)/ T\mathcal{P}_0$  as a bundle over $\mathfrak{h}$ via $\pi_0$.  

We endow $H(\C)$ with the Kobayashi pseudo-metric denoted $d_{Kob}$ (see \cite[p.45]{kobayashi_book} for a definition) and by $|| \cdot ||_{Kob}$ the infinitesimal Kobayashi pseudo-metric.

In order to obtain  the expansion of $F$ for a particular metric, we first show that the Kobayashi pseudo-metric is non-vanishing on a large class of tangent vectors on $\mathcal{D}$. Then prove that these vectors expand for the pseudo-metric. 
 For the first statement, we will work with $\mathcal{C}^{(2)}$  the second iterate  of the cone field.  
\begin{lem} \label{lem_non_deg} For any point real point $p \in \mathcal{D}$ and any real non-zero tangent vector $v \in \mathcal{C}_p^{(2)}$, $|| v||_{Kob}$ is non-vanishing. 
\end{lem}

\begin{proof}
Take $p$ such a point and a tangent vector in the cone field. Observe that $v$ is in the kernel of $\pi_0$ if and only if the vector $v$ belongs to  $T\mathcal{P}_0$. By Theorem \ref{thm_invariant_cone_field}, the cone field $\mathcal{C}^{(2)}_p$ is strictly inside the cone $\mathcal{C}_p$ for which one boundary vector is exactly in the kernel of the projection. This implies that $v$ is not in the pencil $\mathcal{P}_0$, hence $D\pi_0(v)$ is a non-zero vector on $N(\mathcal{P}_0)$.
We prove that on $N(\mathcal{P}_0)$, the Kobayashi pseudo-metric coincides with the metric induced by the hyperbolic metric on $\mathfrak{h}$.   
Following \cite[\S~3.1]{kobayashi_book}, a disc between $p,q \in N(\mathcal{P}_0)$ is a holomorphic embedding $\gamma \colon D \to N(\mathcal{P}_0)$ whose image contains $p,q$. 
 The Kobayashi pseudo-metric between $p, q$ is the infimum over the sum of the Poincare metric:
\begin{equation}
d_{Kob}(p,q) = \inf \sum_{i=0}^{n-1} \rho( \gamma_i^{-1}(p_i) , \gamma_i^{-1} (p_{i+1}) )
\end{equation} 
where the infimum is taken over all chains of holomorphic discs $\gamma_i$ between $p_0 = p , p_1, \ldots , p_{n} = q$. 
Since the Kobayashi pseudo-metric is decreasing under holomorphic mappings, for any two points $p,q \in N(\mathcal{P}_0)$, $d_{Kob}(p,q) \geqslant d_{Kob}(\pi_0(p), \pi_0(q))$. 
Furthermore, for any two points $p, q \in N(\mathcal{P}_0)$, we choose an open set $U \subset \mathfrak{h}$ containing $\pi_0(p), \pi_0(q)$ which trivializes the bundle over $U$, i.e $\pi_0^{-1}(U) \sim U \times \C$. We join the points $p$ and $q$ by a chain of two discs, a horizontal  disc $\gamma_1$ contained in $U \times \{ \lambda\} \subset \pi_0^{-1}(U)$ from $p$ to $p_1 \in \pi_0^{-1}(\pi_0(q))$ and a vertical $\gamma_2$ contained in a fiber over $\pi_0(q)$ from $p_1$ to $q$.  
Since the second disc is contained on a fiber over $\pi_0^{-1}(q)$ which is isomorphic to $\C$, the Kobayashi pseudo-metric is zero on the fiber. Since the first disc is horizontal, we have $\rho(\gamma_1^{-1}(p), \gamma_1^{-1}(p_1)) = d_{Kob}(\pi_0(p), \pi_0(q)) = \rho(\pi_0(p), \pi_0(q))$ by \cite[assertion (2) of Proposition 3.1.7]{kobayashi_book}.
This shows that $d_{Kob}(p,q) \leqslant d_{Kob}(\pi_0(p), \pi_0(q))$, hence $d_{Kob}(p,q) = d_{Kob}(\pi_0(p), \pi_0(q))= \rho(\pi_0(p), \pi_0(q))$.
Finally, a tangent vector $v \in \mathcal{C}^{(2)}$ maps to a non-vertical vector on $N(\mathcal{P}_0)$ so its infinitesimal Kobayashi pseudo-metric is non-zero.    
\end{proof}

\begin{lem} \label{lem_recurrence} Any point $p \in \overline{R_2 \cup R_3} \cup \sigma(\overline{R_2\cup R_3})$ which does not belong to the basin of $\pinf$ and whose orbit avoids the indeterminacy locus has a recurrent orbit in $\overline{R_2 \cup R_3} \cup \sigma(\overline{R_2\cup R_3})$.  
\end{lem}

\begin{proof}
Setting $V = \overline{R_2 \cup R_3} \cup \sigma(\overline{R_2\cup R_3})$. 
 Assume by contradiction that for some iterate, $F^n(p) \notin V$ for all $n \geqslant N$. 
Since $\R^2 = V \cup \overline{R_4 \cup R_5 \cup R_1 \cup R_6 \cup \sigma(R_4 \cup R_5 \cup R_1 \cup R_6)}$ and since $R_1$ and $R_6$ are contained in the basin of $\pinf$,
 this means that $F^N(p) \in \overline{R_4 \cup R_5 \cup \sigma(R_4 \cup R_5)} \setminus \{ \bfz \}$ since the orbit of $p$ avoid the indeterminacy locus. 
 By assertion (iv) of Proposition \ref{prop_mapping_regions}, $F^{N+1}(p) \in \sigma(\overline{R_2 \cup R_3 }) \subset V$, which contradicts our hypothesis, as required. 
\end{proof}

Let us denote by $\mathcal{R}$ the return map in the set $V=R_2 \cup R_3 \cup \sigma(R_2 \cup R_3)$. 
\begin{prop} \label{prop_expansion_kob} There exists a constant $M>1$ such that for any point $p \in V $ whose orbit avoids the indeterminacy locus and any vector $v \in \mathcal{C}_p^{(2)}$, one has,  $||\mathcal{R}^n v||_{Kob} \geqslant M^n || v ||_{Kob}$ for all $n \geqslant 1$.   
\end{prop}

\begin{proof} Note that the Kobayashi pseudo-metric is decreasing under holomorphic maps. Fix a point $p$ and a tangent vector $v $ in $\mathcal{C}_p^{(2)}$. Observe that the set $E = \Lf \cup \bar \Lf \cup \Lz \cup \Linf$ is forward invariant by $F$. Furthermore, 
 consider the restriction $\psi$ of $F$ \begin{equation*}
\psi = F_{| H(\C) \setminus F^{-1}(E)} \colon H(\C) \setminus F^{-1}(E) \mapsto H(\C).
\end{equation*} 
It is holomorphic and the preimage of $E$ contains $E$, the line $\Lc$ as well as the hyperbola $F^{-1} \bar \Lf$, so in particular $\Psi$ is a covering map. Hence it realizes an isometry between $H(\C) \setminus F^{-1}(E)) $ and $H(\C)$ for the Kobayashi pseudo-metric.  
The injection of $\bP^2(\C) \setminus F^{-1}(E) \hookrightarrow \bP^2(\C) \setminus E$ is a contraction of rate $M^{-1}$ for the Kobayashi pseudo-metric for some $M>1$, hence its inverse is expanding. This shows that the map $\psi$ is expanding for the Kobayashi pseudo-metric on $H(\C)$ of rate $M>1$. 
Using the first return map on $V$ which is well-defined by Lemma \ref{lem_recurrence}, we get the uniform expansion.
\end{proof}

We shall now introduce a good region where the pseudo-metric is bounded. This curves will bound a good region contained in $R_2\cup R_3 \cup \sigma(R_2\cup R_3)$. 

\begin{lem} \label{lem_good_curves} For any $\epsilon>0$ there exists two curves $C_1,C_2$ in $\mathcal{L}$ or  in $\mathcal{L}^*$ satisfying the following conditions: 
\begin{enumerate}
\item[(i)] The curve $C_1 $ is contained in $R_3 \cup \sigma(R_3)$ while $C_2$ is  in $R_2 \cup \sigma(R_2)$. 
\item[(ii)] The curve $C_1$ starts from infinity at a point on $p=[a:1:0]$ where $a \in (-2, -2 + \epsilon)$, then goes through $\bfz$ with a slope distinct from $-1/2$ and diverges to infinity in the lower half plane toward the point $\sigma(p)$.
\item[(iii)] The curve $C_2$ starts from infinity at a point $p' =[b:1:0] $ with $b \in [2- \epsilon,2)$ and passes through $\bfz$ and diverges to infinity toward $\sigma(p')$ in the lower half plane. 
\end{enumerate}
\end{lem}

\begin{proof}
The proof is constructive. Observe that the pullback of $\Lc$ by $F$ is a hyperbola and contains a branch in $R_4\cup R_5$ going from $[0:0:1]$ and diverging toward $\binf$ in the upper half plane. 
One checks that the next pullback by $F$ of this curve yields a curve contained in $F^{-2}(\Lc)$ which is not of slope $\pm 1/2$ at $\bfz$. 
Fix $\epsilon>0$, there exists an iterate $N$ and some  preimages $q,q'$ in $\bIz$ of $[0:0:1]$ by $F^{2N}$ such that $q = [a: 0:1]$ and $q'= [b:0:1]$ with $a \in (-2, -2 + \epsilon)$ and $b\in (2-\epsilon, 2)$.   
The preimage of $\Lc$ by $F^{2N+1}$ contains thus two branches $\gamma,\gamma'$ in $\sigma(R_4\cup R_5) \cup R_4\cup R_5$ going through $\binf$ and through $q, q'$ respectively. 
By Proposition \ref{prop_mapping_regions}, the pullback of $\gamma \cup \sigma(R_4\cup R_5) $ by $F$ yields a curve $C_1$ in $R_3\cup \sigma(R_3)$ which is not of slope $-1/2$ at $\bfz$ and goes through $p=[a:1:0]= F(q)$. 
Indeed, the map induced by $F^2$ on the blow-up $\Efz$ at $\bfz$ is conjugate to $z^2 - 12$. If the tangent at $\bfz$ was equal to $-1/2$, then it would correspond to the point $z=-4$ in the blow-up $\Efz$. Since this point is prefixed, we would obtain that $F^{2N-2}(C_3)\subset F^{-2}(\Lc)$ is tangent to $\Lf$ at $\bfz$, which is not possible. 
Similarly, the pullback of $\gamma' \cap R_4\cup R_5$ yields a curve which has a vertical tangent at $\bfz$ and goes toward $p'= F(q')= [b:1:0]$.
We have thus found some pullbacks of $\Lc$ satisfying the conditions of the statement. Taking a limit of these yield some leafs of $\mathcal{L}^*$. 
\end{proof}

In the next statement, we will consider \textit{some regions bounded by two 
curves $C_1, C_2$} provided by the previous lemma. They will precisely consist in the connected component of $R_2\cup R_3\cup \sigma(R_2\cup R_3) \setminus (C_2\cup C_3)$ containing the collapsing line. 

\begin{prop} \label{prop_laminar_good_region} For any $\epsilon> 0$, choose two curves $C_1,C_2$ in $\mathcal{L}^*$ satisfying the conclusion of Lemma \ref{lem_good_curves} and a region $K_\epsilon \subset R_2 \cup R_3 \cup \sigma(R_2 \cup R_3)$ bounded by these two curves.  
Then for any point $p $ in the interior of $K_\epsilon$ whose orbit is regular and  returns infinitely many times in $K_\epsilon$, there exists a unique element of $\mathcal{L}$ through $p$. 
\end{prop}

\begin{proof}
We first observe that the pullback of $\Lc$ are two by two disjoint away from $\Lz \cup \Linf$.
Also, any leaf of $\mathcal{L}$, except $\Lc$ is mapped under $F$ to another leaf of $\mathcal{L}$.  
We claim that any two element of $\mathcal{L}$ \textit{do not cross} (i.e one curve intersects both connected component of the complement of the second curve). Indeed, if two elements cross  at a point $p$, then we would find two nearby pullback of $\Lc$ that would cross, which is impossible. 

\smallskip

Assume by contradiction that there exists two curves $C,C'$ in $\mathcal{L}$ through $p$. In particular, both $C$ and $C'$ are in $K_\epsilon$. Indeed, since $C,C'$ cannot cross the boundary curves $C_1,C_2$, they are on the right hand side of $C_1$ and on the left hand side of $C_2$, hence in $K_\epsilon$..
For any iterate $F^{n_i}$ such that $F^{n_i}(p) \in K_\epsilon$, the same argument shows  that the curves $F^{n_i}(C)$ and $F^{n_i}(C')$ are also in $K_\epsilon$. 
Consider the bubble $B$ consisting of the connected component of $\R^2\setminus (C\cup C')$ in between $C,C'$, then its image $F^{n_i}(B)$ also belongs in $K_\epsilon$. Take any horizontal line $L$ in $B$, its image $F^{n_i}(L)$ is in $K_\epsilon$, hence its length for the Kobayashi pseudo-metric is bounded. 
However, Proposition \ref{prop_expansion_kob} shows that the length  of $F^{n_i}(L)$ grows exponentially in $i$. This gives a contradiction and there exists at most one curve through $p$, as required.  
\end{proof}

The above statement shows that the family $\mathcal{L}$ is well-behaved for orbits that are recurrent in a good region $K_\epsilon$. Our aim is to analyze whether orbits avoiding consistently $K_\epsilon $ exist.

\begin{prop} \label{prop_recurrence_good_orbit} There exists $\epsilon>0$ and a region $K_\epsilon$ bounded by two curves satisfying the conclusion of Lemma \ref{lem_good_curves} such that any point $p\in \R^2$ whose orbit is regular is recurrent in $K_\epsilon$. 
\end{prop} 

\begin{proof}
Note an orbit that is recurrent in $K_\epsilon \cup F^{-1}(K_\epsilon)$ is recurrent in $K_\epsilon$. 
Denote by $K' = K_\epsilon \cup F^{-1}(K_\epsilon)$. 

Let us first choose $\epsilon$ accordingly. 
Note that the leave of $\mathcal{L}^*$ are recurrent on any domain $K_\epsilon$ because their orbit land on generic point of the Chebyshev interval at infinity. 
For $\epsilon$ small enough, the point $\bfz$ is repelling on $\bIz$ so we can choose a leaf $C_1' \in \mathcal{L}^*$ landing at  $[2-\epsilon : 0 :1]$, any leaf in $R_4\cup R_5\cup \sigma(R_4\cup R_5)$ on the left hand side of $C_1'$ maps by $F^2$ to curves in the same domain on the left hand side of $F^2(C_1')$.   
This shows that any point in $R_4\cup R_5 \cup \sigma(R_4\cup R_5) \setminus F^{-1}(K_\epsilon)$ eventually lands in $F^{-1}(K_\epsilon)$ under some iterates of $F^2$. Indeed, given any point $p$ near $\bar \Lf \cup \Lf$ in $R_4\cup R_5 \cup \sigma(R_4\cup R_5) \setminus F^{-1}(K_\epsilon)$. 
 We choose a curve $C_1'$ on the right hand side of $p$ and a element $\mathcal{L}_p$ through $p$. But for some iterate $F^{2n}(C_1')$ will land near $[\sqrt{2}:0:1]$, hence in $F^{-1}(K_\epsilon)$. So for that first iterate $n$ , $F^{2n}(C_1')$ is on the left hand side of $F^{2n}(C_1')$, hence also in $K_\epsilon$.  
This shows any point $p \in R_4\cup R_5\cup \sigma(R_4\cup R_5) \setminus F^{-1}(K_\epsilon)$ near $\bar \Lf \cup \Lf$ eventually lands in $K_\epsilon$. 
The case where $p$ is in $R_4\cup R_5 \cup \sigma(R_4\cup R_5)$ is near $F^{-1}(\bar \Lf) \cup \Lf$ is similar as this region is mapped by $F^2 $ near $\bar \Lf$ for which the same argument applies. 

\smallskip 
Assume that $p \in R_2\setminus K_\epsilon$. By Proposition \ref{prop_mapping_regions}, its image is either in $R_2$, $R_3$ or in $R_4$. We shall analyze each of these three cases. 
 If its image $F(p)$ belongs to $R_4 \setminus F^{-1}(K_\epsilon)$, then this case was treated previously.

\smallskip

Assume that $p \in R_2\setminus K_\epsilon$ and  that its image $F(p)$ is  in $R_3 \setminus K_\epsilon$ in a neighborhood of $\bar \Lf$. Then its second iterate $F^2(p)$ belongs to $\sigma(R_4\cup R_5)$ by Proposition \ref{prop_mapping_regions} and is in a neighborhood of $\Lf$. This case was treated previously. 

\smallskip 

Finally, if $p$ is in $R_2\setminus K_\epsilon$ and its image $F(p)$ remains in $R_2\setminus K_\epsilon$.
Our aim is to show that $p$ belongs to $\Lf$. Assume the contrary.
 By the previous argument, all iterates of $p$ must remain in $R_2\setminus K_\epsilon$. 
Consider a horizontal line $H$ through $p$, it intersects the line $\Lf$ at a point $q$, which is fixed by $F^2$. In particular, $F^{2n}(H)$ is a sequence of horizontal curves for the cone field and $F^{2n}(H) \cap R_2 \setminus K_\epsilon$ is then a sequence of horizontal curves through $q$. Since these curves have bounded slope, this shows that $F^{2n}(p)$ is in a compact set. By applying $F$, we can assume that $q$ is below the horizontal line $y=1$ (since $F$ is an involution on $\Lf$ with fixed point at $[2:1:1]$).
In particular, the image $F^{2n}(p)$ is below $y=1$. We estimate the horizontal derivative $\partial_x x\circ F^2$ on this compact set. 
Let us first observe that if $p \in R_2$ and $F(p) \in R_2$, then this implies that $p$ is in the region bounded by the hyperbola $F^{-1}(\Lc)$ and $\Lf$. 
Since the equation of the hyperbola $F^{-1}(\Lc)$ is given by $4y^2 = x(x-2)$, we get for any point $p= (x,y) \in R_2 \cap F^{-1}(R_2)$, we have $(x-2)^2 = x(x-2) + 4 \geqslant  4 + 4y^2 $ and $(x-2) \geqslant 4 y^2 / x$ together with the fact that $x\geqslant 2$, we get:
\begin{multline*}
\partial_x (x\circ F^2)(x,y) = 2x + \dfrac{12 x y^2}{x-2} - \dfrac{16 y^4}{(x-2)^3} \geqslant 2x  - \dfrac{16 y^4}{4(x-2)(1+y^2)} \geqslant  2x - \dfrac{4 y^4}{(x-2)} \\ \geqslant 2x - \dfrac{4 y^4x}{4y^2} \geqslant 2x - x y^2 = x(2 - y^2) \geqslant 2. 
\end{multline*} 
This shows that $F^2$ is expanding exponentially in the horizontal direction for the euclidean metric for all infinite orbits that remain below the line $y=1$ and in $R_2\setminus K_\epsilon$. In particular, this region cannot be invariant by $F^2$ and any point $p$ not on the line $\Lf$ must escape $R_2\setminus K_\epsilon$, as required.

This shows that there are no regular orbits which is consistently in the complement of $K_\epsilon \cup F^{-1}(K_\epsilon)$, as required. 
\end{proof}

\begin{thm} \label{thm_lamination} The family $\mathcal{L}$ is a lamination away from the lines $\Lz \cup \Linf$.  
\end{thm}

\begin{proof}
By Proposition \ref{prop_laminar_good_region}, $\mathcal{L}$ defines a lamination on any region $K_\epsilon$ for $\epsilon>0$. By Proposition \ref{prop_recurrence_good_orbit}, there exists a region $K_\epsilon$ such that every regular orbit is recurrent in $K_\epsilon$, hence $\mathcal{L}$ is a lamination.   
\end{proof}

\subsection{Positive area of $\cL$}

In order to show that $\cL$ has positive area, we will use the following consequence of the previous results. 

\begin{cor} \label{cor_L_star_leaf} Any leaf in $\mathcal{L}^*$ lands uniquely at a point of the Chebyshev interval. Moreover, given any point $p\in \R^2$ and any stable manifold $\cW^s(p)$ which intersect $\bIz$ at a regular point belongs to $\mathcal{L}^*$. 
\end{cor}

\begin{proof}
Take two leaves $C, C'$ in $\mathcal{L}^*$ landing on the same regular point $q$ of the Chebyshev interval $\bIz\cup \bIinf$. Without loss of generality (up to considering its image by $F$), we can assume that  $q \in \bIinf$. By definition, its orbit is dense on the Chebyshev interval, we can choose a subsequence $n_k$ such that $F^{2n_k}(q)$ is arbitrarily close to $\Oinf$. In particular, $F^{2n_k}(C), F^{2n_k}(C')$ will belong to the region constructed in Proposition \ref{prop_laminar_good_region}. The same proposition shows that the two leaves must be equal.
For the second statement, we observe that any stable manifold $\cW^s(p)$ is a vertical curve for the cone field $\mathcal{C}^{(1)}$. Furthermore, assuming that the landing point $q$ on $\bIz$ is regular, then we can compare the orbit of $\cW^s(p)$ and of the unique leaf $C$ through $q$ in $\mathcal{L}^*$. The same argument shows that the orbit of these two curves enter infinitely many times the region constructed in Proposition \ref{prop_laminar_good_region} and so they coincide, as required.  
\end{proof}

\begin{thm} \label{thm_L_positive} The support of $\cL$ has positive area and any slice of $\mathcal{L}$ by a horizontal line $L_h = \{ y =h \}$ with $h \in \R$ has positive length.  
\end{thm}

\begin{proof}
By Theorem \ref{thm_positive_measure_subset}, there exists a $F$-invariant subset $U \setminus \cup_{n\geqslant 0} F^{-n}(W)$ of positive area. 
For any point $p $ in that set, its orbit converges to $\bIz$ and we can construct a stable manifold $\cW^s(p)$ through it. Observe that $\cW^s(p)$ is a vertical curve for the cone field. Let $q$ be the intersection point between $\cW^s(p)$ and $\bIz$. By definition, the point $q$ has an infinite orbit which avoids the critical point and $\bfz$. 
By Corollary \ref{cor_L_star_leaf}, each stable manifold $\cW^s(p)$ then belongs to $\mathcal{L}^*$. In particular, this shows that $U \setminus \cup_{n\geqslant 0} F^{-n}(W)$ belongs to the stable set of $\bIz$ and that it consists of leaves in $\mathcal{L}^*$. 
This shows that $\mathcal{L}^* \cap U\setminus \cup_{n\geqslant 0} F^{-n}(W)$ has positive area, hence $\mathcal{L}$ has positive area.
The next statement follows also from Theorem \ref{thm_positive_measure_subset}. Indeed, for small $h$, the length of $ \mathcal{L}^*\cap L_h \cap U \setminus \cup_{n\geqslant 1} F^{-2n}(W) $ is positive and converges to $4$. 
If $h$ is large, we proceed as follows. By Proposition \ref{prop_good_intervals}, we can find a sequence of subintervals $I_n$ in $L_h$ and some iterates $a_n$ such that $F^{a_n}(I_n)$ maps into the region $U$ and converging toward the interval $[0,2] \times \{0\}$. In particular, these arcs will intersect a subset of positive area $\mathcal{L}^* \cap U \setminus \cup_{n\geqslant 1} F^{-2n}(W)$, so their preimage by $F^{a_n}$ will have positive length, as required.   
\end{proof}

\section{The real trace of the Green current}
\label{section_real_green}

\subsection{Combinatorial description of $\mathcal{L}$}
\label{section_combinatorics}
In this section, our aim is to give a combinatorial description of the leaves of $\mathcal{L}$. Recall from \S~\ref{section_kobayashi} that we have established that $\mathcal{L}$ is a lamination. 
In the following statements, we will call a point on $\bIz\cup \bIinf$ \textit{precritical} if some iterates of $F$ map to $\Oinf$, otherwise, it will be called \textit{regular}. 
 
We first subdivide certain leaves into 9 classes (see Figure \ref{fig_coding}) and defined as follows: 
\begin{enumerate}
\item[$\bullet$] \textbf{Type $I$}: A leaf is of type $I$ if it is $\sigma$-symmetric and is contained in $R_5\cup \sigma(R_5)$. More precisely, it starts from infinity in the upper-half-plane at $\binfp$ then goes to $\bfzp$ with a vertical tangent and diverges to infinity in the lower half plane toward the point $\binf$. 
\item[$\bullet$] \textbf{Type $II$}: A  leaf is of type $II$ if it is $\sigma$-symmetric, is contained in $R_4\cup R_5 \cup \sigma(R_4 \cup R_5)$ and passes through a point on $\bIz$ which is precritical. 
\item[$\bullet$] \textbf{Type $III_-$}: A leaf is of type $III_-$ if it is $\sigma$ symmetric, is contained in $R_3^+ \cup R_3^- \cup \sigma(R_3^+ \cup R_3^-)$ and passes through $\bfz$ and a precritical point on $\bIinf$  distinct from $\Oinf$. 

\item[$\bullet$] \textbf{Type $III_0$}: A leaf is of type $III_0$ if it is $\sigma$ symmetric, is contained in $R_3^+ \cup R_3^- \cup \sigma(R_3^+ \cup R_3^-)$ and passes through the point $\Oinf$ on $\bIinf$.

\item[$\bullet$] \textbf{Type $III_+$}: A leaf is of type $III_+$ if it is $\sigma$ symmetric, is contained in $R_2^+ \cup R_2^- \cup \sigma(R_2^+ \cup R_2^-)$ and passes through a precritical point on $\bIinf$  distinct from $\Oinf$. 

\item[$\bullet$] \textbf{Type $IV$}: A leaf is of type $IV$ if it is $\sigma$ symmetric, is contained in $R_2^- \cup \sigma(R_2^-)$  and goes from $\binf$ in the upper half-plane toward $\bfz$ with a vertical tangent and diverges to infinity toward $\binfp$ in the lower half-plane. 

\item[$\bullet$] \textbf{Type $\mathcal{V}_{II}$}: A leaf is of type $\mathcal{V}_{II}$ if it is $\sigma$-symmetric, is contained in $R_4\cup R_5 \cup \sigma(R_4 \cup R_5)$ and passes through a regular point on $\bIz$. 

\item[$\bullet$] \textbf{Type $\mathcal{V}_{III_-}$}: A leaf is of type $\mathcal{V}_{III_-}$ if it is $\sigma$-symmetric, is contained in $R_3^+\cup R_3^- \cup \sigma(R_3^+\cup R_3^-)$ and passes through $\bfz$ and a regular point on $\bIinf$.
 
\item[$\bullet$] \textbf{Type $\mathcal{V}_{III_+}$}: A leaf is of type $\mathcal{V}_{III_+}$ if it is $\sigma$-symmetric, is contained in $R_2^+\cup R_2^- \cup \sigma(R_2^+\cup R_2^-)$ and passes through $\bfz$ and a regular point on $\bIinf$.
 
\end{enumerate}
The first $6$ classes of leaves will be related to the Green current, while the last $3$ consists of curves which create a forward invariant subsystem of leaves. 
Recall from \S~\ref{section_kobayashi} that we have defined a subset $\mathcal{L}^*$ of $\cL$, which are all the leaves going through a regular point on the Chebyshev intervals. 
The reason why we will distinguish these last three classes is because they belong to $\mathcal{L}^*$ while the first 6 are on the complement.

\begin{figure}[h!]
\centering
\includegraphics[scale=0.35]{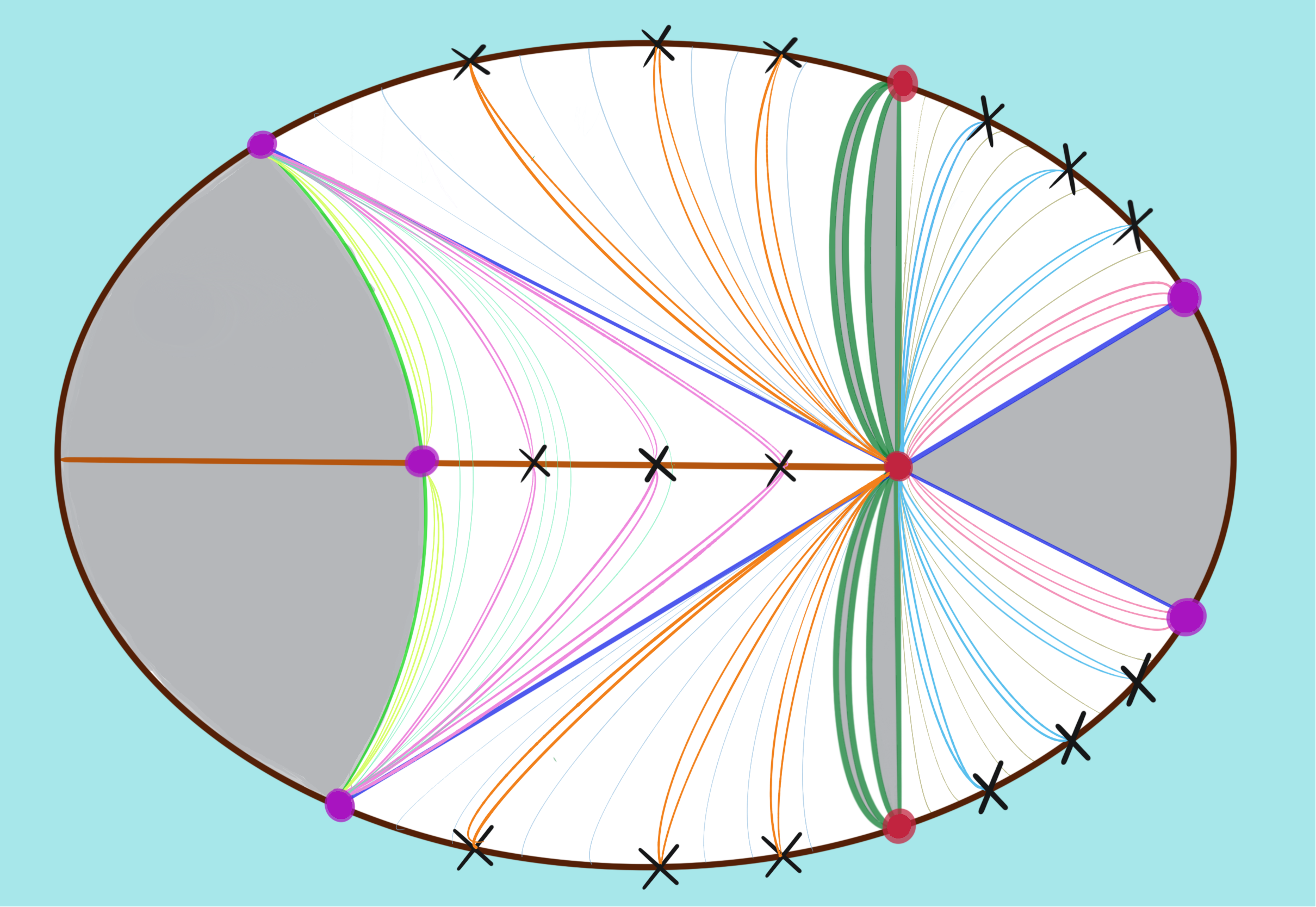}
\caption{\label{fig_coding} Schematic picture of the coding of the leaves of $\mathcal{L}$ in $\bP^2(\R)$, the circled points are the $\beta$-fixed points and the critical point at infinity, the ellipsis in black is the line at infinity. The crosses are precritical points and the curves through it of purple, orange and light blue color are of type $II$, $III_-$, $III_+$ respectively. The curves of type $I$, $III_0$, and $IV$   are in yellow, dark green and pink respectively. Last, the curves (in between the crosses and dotes) are of type $\mathcal{V}_{II}, \mathcal{V}_{III_-}$ and $\mathcal{V}_{III_+}$ are represented in light green, blue, green/gray.}
\end{figure}


 One would expect that the leaves of $\mathcal{L}^*$  to be charged by the Green current, but we will show that this is not the case. 

\begin{thm} \label{thm_markov} Given any leaf $\mathcal{L}_x$ of $\mathcal{L}$ and any connected component of $C \subset \mathcal{L}_x \setminus \Linf$, then the closure $\overline{C} $ in $\bP^2(\R)$ belongs  to exactly  one of the classes  $I, II, III_-, III_0, III_+ , IV , \mathcal{V}_{II}, \mathcal{V}_{III_-}$ and  $\mathcal{V}_{III_+}$. 
Moreover, fix a bijection $\varphi \colon \{1 ,\ldots , 9\} \to \{ I , II, III_-, III_0, III_+, IV, \mathcal{V}_{II}, \mathcal{V}_{III_-},\mathcal{V}_{III_+} \}$, then the  $9 \times 9$ matrix $M$ given by: 
\begin{equation} \label{eq_matrix}
M_* = (m_{ij})_{i,j \leqslant 9} := \left  ( 
\begin{array}{lllllllllll}
  0 & 0 &  0     & 1     & 0     &  0 & 0    & 0      & 0 \\ 
  0 & 0  &  1    & 0     & 1     & 0  & 0    & 0      & 0   \\ 
0   & 1  & 0     & 0     & 0     & 1  & 0    & 0      & 0   \\
0   & 1  & 0     & 0     & 0     & 1  & 0    & 0      & 0   \\
0   & 1  & 0     & 0     & 0     & 1  & 0    & 0      & 0   \\
  1 & 0  &0      & 0     & 0     & 1  & 0    & 0      & 0 \\ 
 0  & 0  & 0     & 0     & 0     & 0  & 0    & 1      & 1   \\
 0  & 0  & 0     & 0     & 0     & 1  & 1    & 0      & 0   \\
 0  & 0  & 0     & 0     & 0     & 1  & 1    & 0      & 0   \\
\end{array}
  \right )
  \begin{array}{l}
  I \\
  II \\
  III_- \\
  III_0 \\
  III_+ \\
  IV \\
  \mathcal{V}_{II}\\
  \mathcal{V}_{III_-} \\ 
  \mathcal{V}_{III_+}
  \end{array}
\end{equation}
 satisfies the following properties: 
\begin{enumerate}
\item[(i)] (Compatibility with $F_*$) For any $j \leqslant 9$, given any curve $C$ of $\mathcal{L}$ of type $\varphi(j)$, its image by $F$ is a leaf of $\mathcal{L}$ of type $\varphi(i)$ for some index $i \leqslant 9$ with $m_{ij} \neq 0$.  
\item[(ii)] (Surjection) For any $i \leqslant 9$, for any leaf $C$ of $\mathcal{L}$ of type $\varphi(i)$,  for any $j \leqslant 9$ with $m_{ij} \neq 0$, there exists a leaf of type $\varphi(j)$ whose image by $F$ is $C$. 
\end{enumerate}  
In other words, these classes of curves form a Markov partition of $\mathcal{L}$ whose matrix is given by $M$. 
Moreover, the pullback by $F_{\R}$ is encoded by the transpose of the matrix $M_*$. 
\end{thm}

\begin{proof}
Let us first prove that these give a partition of $\mathcal{L}$. 
Recall that $\Lf$ and $\Lc$ are vertical curves for the cone field and that $\mathcal{L}$ is a lamination away from $\Lz \cup \Linf$ by Theorem \ref{thm_lamination}. 
In particular, any leaf of $\mathcal{L}$ does not intersect $\Lf $, $\Lc$ and its preimages away from $\Lz \cup \Linf$. 
Recall also that the plane $\R^2$ is subdivided into regions bounded by $\bar \Lf$, $\Lf$, a component of the preimage $\gamma$ of $\bar \Lf$ by $F$. Moreover, every pullback of $\Lc$ is $\sigma$-symmetric since $F\circ \sigma = F$, hence every leaf of $\mathcal{L}$ is $\sigma$-symmetric and is in a region bounded by $\gamma$, $\Lf, \Lc , \bar \Lf$. Furthermore, the interval at infinity $\bIinf$ is partitionned into points that are precritical/regular points on the interval $(-2, 0)$, precritical/regular points in the interval $(0,2)$ and the singleton $\{ 0\}$. As a result, the leaves that diverge to infinity toward a point distinct from $\binfp$ must be intersect $\bIinf$ at one of these classes of points. This shows that the leaves in $R_3^+ \cup R_3^- \cup \sigma(R_3^+ \cup R_3^-)$ are of type $III_-,III_0$ or $\mathcal{V}_{III_-}$. Similarly, for all the other curves in each region, this shows that these 9 classes of curves form a partition.
 
\medskip
 
We will prove assertion (ii) directly by showing that the pullback of a curve of type $\varphi(i)$ is a union of curves of type $\varphi(j)$ for all $j$ with $m_{ij}\neq 0$. We prove it case by case. 

$\bullet$ Let $C$ be a leaf of type $I$. Then the arc $C\cap \sigma(R_5)$ admits a preimage, which is a vertical $\sigma$-symmetric curve contained in $R_3 \cup \sigma(R_3)$ by assertion (iii) of Proposition \ref{prop_mapping_regions}. Moreover, since $C$ passes through the point $\bfzp$, its preimage must pass through the point $\Oinf$, and this proves that the pullback of $C\cap \sigma(R_5)$ is a curve of type $III_0$, as required. 

\smallskip

$\bullet$ Let $C$ be a leaf of type $II$. Then by assertion (ii) of Proposition \ref{prop_mapping_regions}, the  pullback of $R_4^- \cap C$ is a symmetric leaf in $R_2\cup \sigma(R_2)$. Since the map induced by $F$ from $\bIinf$ to $\bIz$ is the Chebyshev map, this pullback must intersect of point at infinity which is precritical, but distinct from $\Oinf$ because $C$ does not pass through $\bfzp,\bfz$. 
This shows that the pullback of $C\cap R_4^-$ is a curve of type $III_+$. 
Next we analyze the pullback of $C\cap \sigma(R_4 \cup R_5)$.  
By assertion (iii) of Proposition \ref{prop_mapping_regions}, the pullback of $C\cap \sigma(R_4 \cup R_5)$ is contained in $R_3\cup \sigma(R_3)$. By the same argument, since $C$ intersects $\bIz$ at a precritical point distinct from $\bfz, \bfzp$, the pullback intersects $\bIinf$ at a precritical point distinct from $\Oinf$. This pullback is thus a curve of type $III_-$.  This shows that the pullback of $C$ is a curve of type $III_-$ and a curve of type $III_+$. 

\smallskip 

$\bullet$ The argument is similar if $C$ is a leaf of type $\mathcal{V}_{II}$, the only difference is that the pullback curves is always going to intersect $\bIinf$ at a regular point. So the pullback of a leaf of type $\mathcal{V}_{II}$ is one curve of $\mathcal{V}_{III_-}$ and a curve of type $\mathcal{V}_{III_+}$.

\smallskip

$\bullet$ Let $C$ be a leaf of type $III_-$. Observe that the map induced by $F$ from $\bIz$ to $\bIinf$ is the identity map, so the pullback of $C\cap \sigma(R_3)$ is a curve going through a precritical point on $\bIz$ which is distinct from $[0:0:1]$. 
By assertion (iv) of Proposition \ref{prop_mapping_regions}, the pullback of $C\cap \sigma(R_3)$ is  $\sigma$-symmetric leaf of $\mathcal{L}$ contained in $R_4\cup R_5 \cup\sigma(R_4\cup R_5)$. It is a curve of type $II$. 
Similarly assertion(ii) of Proposition \ref{prop_mapping_regions} shows that the pullback of $C\cap R_3^-$ is contained in $R_2\cup \sigma(R_2)$ and yields a curve joining $\bfz$ (the preimage of the intersection point of $\Li$ with $C$) and $\binf,\binfp$. This is a curve of type $IV$. This shows that the pullback of a curve of type $III_-$ is a curve of type $IV$ and a curve of type $II$.

\smallskip 

$\bullet$ The case where $C$ is a leaf of type $\mathcal{V}_{III_-}$ is similar, except that the pullback of the part in $C\cap \sigma(R_3)$ by $F$ is a curve going through a regular point in $\bIz$. In particular, the pullback of $C$ is a curve of type $\mathcal{V}_{II}$ and a curve of type $IV$. 

\smallskip

$\bullet$ Let $C$ be a leaf of type $III_0$. This is similar to the previous two cases. The pullback is one curve of type $II$ and a curve of type $IV$. 

\smallskip 

$\bullet$ If $C$ be a leaf of type $III_+$. Then this case is similar to the previous cases.

\smallskip

$\bullet$ If $C$ is a leaf of type $\mathcal{V}_{III_+}$, then the argument is similar except that one curve will be a regular leaf, in $\mathcal{V}_{II}$ and the other will be a curve of type $IV$. 

\smallskip 

$\bullet$ Let $C$ be a leaf of type $IV$. Then by assertion (ii) of Proposition \ref{prop_mapping_regions}, the pullback of $C\cap R_2^-$ is a curve  in $R_2\cup \sigma(R_2)$ which passes through $\bfz$ and $\binf,\binfp$. In particular, this is a curve of type $IV$. 
Note that $C$ is contained in $R_2^- \cup \sigma(R_2^-)$, so the pullback of $C\cap \sigma(R_2)$ is by assertion (vii) of Proposition \ref{prop_mapping_regions} a $\sigma$-symmetric curve in $ R_5 \cup \sigma( R_5)$, which passes through $\bfzp$ and diverges toward infinity toward the points $\binf,\binfp$. This pullback is a curve of type $I$. The pullback of $C$ is thus one curve of type $I$ and a curve  of type $IV$, as required.   

This finishes the proof of assertion (ii) and yields a matrix encoding the pullback on $\mathcal{L}$, its transpose is exactly the matrix $M$. We conclude as assertion (ii) implies assertion (i).   
\end{proof}

\begin{cor} \label{cor_count_real_curves} The matrix $M_*$ from Theorem \ref{thm_markov} has characteristic polynomial $x^4 (x^2 -2)(x^3 -x^2 - 2 x +1)$. The vectors $ \pm \sqrt{2} e_7 + e_8 + e_9 $ are the eigenvectors of the eigenvalue $\pm \sqrt{2}$.  
Moreover, the eigenvector for the transpose of $M_*$ associated with the eigenvalue $\delta$ belongs to the span $\R e_1 + \ldots + \R e_6$.
\end{cor}

\begin{proof}
This follows from a direct computation. As for the second statement, the transpose of $M_*$ is a block-upper triangular matrix where the first block is $6\times 6$ while the other $3\times 3$ block is of the form 
\begin{equation*}
N = \left  ( \begin{array}{lll}
0&1&1 \\
1 & 0 &0 \\
1& 0 & 0 
\end{array} \right ).
\end{equation*} 
In particular, it is symmetric and its eigenvalues are $0, \pm \sqrt{2}$. 
As a result, the spectral radius  of the $6\times 6$ block is equal to $\delta$ and its associated eigenvector gives an eigenvector of the transpose of $M_*$ in $\sum_{i=1}^6 \R e_i$, as required.
\end{proof}

\subsection{Counting real pullback and complex pullback of curves}

In this section, we investigate the properties of the pullback by $F$ in the real domain. Recall that the Green current $T_F$ is a complex current obtained by pulling back the line $\Lc$ and divided by $\delta^n$. 
First we analyze how the  real pullback  behaves.
For each type of vertical curves indexed via the bijection $\varphi \colon \{1, \ldots, 9 \} \to \{ I, II, III_-, III_0, III_+, IV, \mathcal{V}_{II}, \mathcal{V}_{III_-} , \mathcal{V}_{III_+}  \}$, we will associated a corresponding vector $v_i \in H^{1,1}(X, \C)$ as follows: 
\begin{align*}
v_1 &  = L - \Efzp - \dfrac{\Einf}{2} - \dfrac{\Einfp}{2}, & v_2=v_7 & = L - \dfrac{\Einf}{2} - \dfrac{\Einfp}{2} , & v_3 = v_4 = v_5 = v_8=v_9 &  =  L - \Efz, \\
v_6 &= L - \Efz - \dfrac{\Einf}{2} - \dfrac{\Einfp}{2}.  
\end{align*}

\begin{thm} \label{thm_magic} Take $C= \Lf \cup \bar \Lf$ or $C= \Lc$.
 For any $n \geqslant 2$,  the real  pullback $(F_X^n)^* C$ is a union of real curves $C \cup \cup_{i=1}^p  C_i$ where each curve $C_k$ is a leaf of type $\varphi(i_k)$. Then the cohomology of $(F_X^n)^* C$ is equal to $\sum_{i=1}^p v_{i_k} + \{ C\}$.  
\end{thm}

\begin{proof} Since the two arguments are similar, we assume that $C = \Lf \cup \bar \Lf$. 
We proceed by induction on $n\geqslant 1$. For $n=1$, observe that the pullback of $\Lf$ is the union $\Lf \cup \bar \Lf \cup \gamma$ where $\gamma = F^{-1}(\bar\Lf)$. 
Since $\gamma$ is a hyperbola going through $\binf, \binfp, \bfz, \bfzp$, its cohomology is equal to $2 L - \Efz - \Efzp - \Einf - \Einfp = v_1 + v_6$. This shows that the equality holds for $n=1$. 
\smallskip 

Let us assume that  $n \geqslant 1$ holds. Note that pullback is controlled by the transpose of matrix $M_*$ given in Theorem \ref{thm_markov}. 
Take any  curve $V$ in $(F^{n}_X)^* \Lf \setminus (F_X^{n-1})^* \Lf$.
Note that since the cohomology class of $V$ has integer coefficients, this means that its real part is a union of $n_k$ leaves of type $\varphi(k)$ and by the induction hypothesis, its  cohomology class is: 
\begin{equation*}
\{ V\} = \sum_{k=1}^9 n_k v_{k}. 
\end{equation*} 
Since $\{ V\}$ is integral, we deduce that  
$n_1 + n_2 + n_6 + n_7$ is an even integer. 
The real pullback of $V$ consists of $n_k'$ leaves of type $\varphi(k)$ where $ ^t (n_1', n_2', \ldots, n_9') = ^t M_* \  ^t (n_1, \ldots, n_9)$.
 We thus need to show  $F_X^* \{V \} = \sum_{k=1}^9 n_k' v_k $. 
To that end, we take the difference of matrices $\Delta = U \ ^t M_* - (F_X)^* U$ where $U$ is the matrix whose column $j$ correspond to the cohomology class  $v_j$ for each $j$, where $ \ ^t M_*$ is the transpose of the matrix $M_*$ corresponding to the pullback action on real curves and where $(F_X)^*$ is the matrix of the pullback action in $H^{1,1}(X,\C)$. 
  We have the matrix equality 
\begin{multline*}
\Delta =\left(
\begin{array}{cccccccccc}
 1    & 1    & 1  & 1  & 1  & 1    & 1    & 1  & 1\\
 -1   & 0    & 0  & 0  & 0  & 0    & 0    & 0  & 0\\
 -1/2 & -1/2 & 0  & 0  & 0  & -1/2 & -1/2 & 0  & 0\\
 0    & 0    & -1 & -1 & -1 & -1   & 0    & -1  & -1\\
 -1/2 & -1/2 & 0  & 0  & 0  & -1/2 & -1/2 & 0   & 0 \\
\end{array}
\right) \left(
\begin{array}{ccccccccc}
 0 & 0 & 0 & 0 & 0 & 1 & 0 & 0 & 0 \\
 0 & 0 & 1 & 1 & 1 & 0 & 0 & 0 & 0 \\
 0 & 1 & 0 & 0 & 0 & 0 & 0 & 0 & 0 \\
 1 & 0 & 0 & 0 & 0 & 0 & 0 & 0 & 0 \\
 0 & 1 & 0 & 0 & 0 & 0 & 0 & 0 & 0 \\
 0 & 0 & 1 & 1 & 1 & 1 & 0 & 1 & 1 \\
 0 & 0 & 0 & 0 & 0 & 0 & 0 & 1 & 1 \\
 0 & 0 & 0 & 0 & 0 & 0 & 1 & 0 & 0 \\
 0 & 0 & 0 & 0 & 0 & 0 & 1 & 0 & 0 \\
\end{array}
\right)
 \\ -  \left(
\begin{array}{ccccc}
 2 & 1 & 0 & 0 & 0 \\
 0 & 0 & 1 & 0 & 0 \\
 0 & 0 & 0 & 1 & 0 \\
 -1 & -1 & 0 & 0 & 1 \\
 0 & 0 & 0 & 1 & 0 \\
\end{array}
\right) \left(
\begin{array}{cccccccccc}
 1    & 1    & 1  & 1  & 1  & 1    & 1    & 1  & 1\\
 -1   & 0    & 0  & 0  & 0  & 0    & 0    & 0  & 0\\
 -1/2 & -1/2 & 0  & 0  & 0  & -1/2 & -1/2 & 0  & 0\\
 0    & 0    & -1 & -1 & -1 & -1   & 0    & -1  & -1\\
 -1/2 & -1/2 & 0  & 0  & 0  & -1/2 & -1/2 & 0   & 0 \\
\end{array}
\right)
\\ = \left(
\begin{array}{ccccccccc}
 0 & 0 & 0 & 0 & 0 & 0 & 0 & 0 & 0 \\
 \frac{1}{2} & \frac{1}{2} & 0 & 0 & 0 & -\frac{1}{2} & \frac{1}{2} & 0 & 0 \\
 0 & 0 & 0 & 0 & 0 & 0 & 0 & 0 & 0 \\
 -\frac{1}{2} & -\frac{1}{2} & 0 & 0 & 0 & \frac{1}{2} & -\frac{1}{2} & 0 & 0 \\
 0 & 0 & 0 & 0 & 0 & 0 & 0 & 0 & 0 \\
\end{array}
\right).
\end{multline*}
Most of the entries of $\Delta$ are zero, but there are some non-zero entries, hence the identity $U \ ^t M_* = (F_X)^* U$ is not always true. 
The key point is that the range $W = \ ^t M_* (\R e_1 + \ldots + \R e_9)$ is in the span of $\R e_3 + \R e_4 + \R e_5 + \R (e_6 + e_1) + \R (e_6 + e_2) + \R (e_8 + e_9) + \R(e_6 + e_7)$. 
  In particular, since all the vectors $e_3, e_4, e_5, (e_6 + e_1), (e_6 + e_2) , (e_6 + e_7), e_8, e_9$ are in the kernel of $\Delta$, we get that the equality 
  $U  \ ^t M_* v = (F_X)^* U v$ holds for any vector $v \in W$. 
Observe that when $n=1$, the pullback $F^{-1}(\bar \Lf)$ yields a union of two curves, one of type $I$ and the other of type $IV$. 
This shows that the class of $\{ F^{-1}(\bar \Lf)\}$ belongs to $W$.
By an immediate induction, since $W$ is stable by $\ ^t M_*$, we get that $\{ V\} = \sum_{k} n_k v_k$ where the vector $ \underline{n} =  ^t (n_1, \ldots, n_9) \in W$. In particular, we get $U\ ^t M_* \underline{n} = (F_X)^* U \underline{n}$, and this gives the required equality. 
\end{proof}

\subsection{Signed pullback $F^{\sharp}$ of real vertical currents}

In this section, we show that a real Green current potentially exists. While the pullback of complex closed positive currents by arbitrary rational maps is well-defined, the pullback of real currents is not automatically well-behaved in a general setting. 
The main difficulty in our case is that the line $\Lz$ is transversally critical and that the line $\Lc$ has a collapsing behavior.   
To that end, we endow each curve on $\R^2$ with an orientation for which the \textit{positive} direction is upward, i.e the $y$ increasing direction. 
When an oriented curve is not oriented positively, we say the orientation is \textit{negative}. 

\medskip 

Let us now define an oriented pullback as follows. 
Given any positively oriented vertical curve $\gamma$ in $\mathcal{L}$, its strict transform by $F_{\R}$ is a union of vertical curves. These curves are $\sigma$-symmetric. 
For each component $\gamma'$ of the strict transform of $\gamma$ by $F$, we say that the oriented pullback by $F_{\R}$ of $\gamma$ is $\gamma'_\epsilon$ where $\epsilon$ corresponds to the orientation of $\gamma'$ and where  the sign of $\epsilon$ coincides with the one of the jacobian of $F$ (on the canonical basis of $\R^2$) at every point of $\gamma'$.
We denote by $ \Fs$ this oriented pullback. 

\medskip

We now establish a current of integration on an oriented curve $\gamma$. Given any oriented curve, we take a parametrization $\varphi \colon (0,1) \to \gamma$ whose orientation coincides with the one of $\gamma$. 
Given any smooth $1$-form $\alpha$ on $\R^2$ with compact support, the current of integration along $\gamma$, denoted $[\gamma]$, applied at $\alpha$ is defined as: 
\begin{equation}
\langle [\gamma], \alpha \rangle = \int_{(0,1)} \varphi^* \alpha. 
\end{equation}
When $\gamma$ is a vertical oriented curve for the cone field $\mathcal{C}$, the parametrization $\varphi$ is Lipschitz, hence the derivative of $\varphi$ exists almost everywhere by Rademacher's theorem (see \cite[\S~3.2 p.21]{morgan_gmt}) and $[\gamma]$ is a rectifiable current (see \cite[p. 39]{morgan_gmt}). 

\smallskip

Any rectifiable $1$-current $T$ has a notion of mass, denoted as $M(T)$ and defined as: 
\begin{equation}
M(T) = \sup \{ |\langle T, \alpha \rangle | \  || \alpha||_\infty \leqslant 1 \},
\end{equation}
where the supremum is taken over all smooth compactly supported $1$ forms $\alpha$ which are uniformly bounded by $1$ at all points. 
For a graph of a Lischitz function, the mass of the corresponding integration current is the one dimensional Hausdorff measure of the corresponding set. 
 
 Naturally, we endow the space of rectifiable currents with the weak topology, the topology of pointwise convergence on each compactly supported forms.    
  
First, we get some mass estimates. 
\begin{prop} \label{prop_bound_mass} Let $V$ be a curve in $\mathcal{L}$, we consider its positively oriented current of integration $[V]$. 
 Then the currents: 
\begin{equation}
 T_n =\dfrac{ \Fs^n [V]}{\delta^n} 
\end{equation}
have uniformly bounded mass (i.e $M(T_n) < C < +\infty$ for all $n$ and some $C>0$). 
\end{prop}

\begin{proof}
Fix a compact subset $K$ in $\R^2$. For any vertical curve $\gamma$ for the cone field $\mathcal{C}$, the length of $\gamma \cap K$ is bounded since $\gamma$ is vertical for the cone field and the bound is uniform, i.e there exists a constant $C_K> 0$ such that for any curve $\gamma$, the length of $\gamma \cap K$, or the $1$ Hausdorff measure is bounded by $C_K$. 
 By Theorem \ref{thm_markov}, the number of vertical curves in $F^{-n}(V)$ can be computed using the transpose of the matrix $M_*^n$. As the spectral radius of $M_*$ is $\delta$, the number of vertical curves is bounded by $ C_1 \delta^n$ where $C_1 > 0$. 
 This shows that the mass of $\Fs^n [V]$ is bounded above by $C C_1 \delta^n$, hence the quotient has uniformly bounded mass, as required.
\end{proof}

\subsection{Coding of real vertical currents} \label{section_coding}
 
 Proposition \ref{prop_bound_mass} shows that the sequence of pullbacks has bounded mass, hence weak limits for subsequence exists. 
Now, we construct some components of the limit candidate as follows. 
These components will be very singular near the Chebyshev interval, namely they will consists of bouquets of curves going through every precritical point. The part of the currents going through the same point $p$ is what we call the \textit{wakes} at $p$.  
We start with the construction of the pieces of the fundamental wake at $(0,0)$, as it will generate all the other wakes and we fix an alphabet $A = \{ I, II, III_-, III_0, III_+ , IV\}$.

\begin{lem} For any generic infinite admissible word $\underline{w}= (w_1, \ldots, w_n , \ldots)$ in the symbol $A$ such that $w_1 = II$ and $w_2 = III_0$, there exists a current $[\gamma_{\underline{w}}]$ supported by one vertical curve in $\mathcal{L}$ whose coding under the image by $F$  coincides with the word $\underline{w}$.  
\end{lem}   

\begin{rem} Concretely, if the coding of a curve starts with $II, III_0$, then this means the curve lies in $R_4$ but maps to a curve that goes through the point $\Oinf$. In particular, it goes through the point $p=(0,0)$. 
\end{rem}

\begin{proof} Take $\underline{w}$ a generic word satisfying the hypothesis of the Lemma. 
Let us construct $\gamma_{\underline{w}}$ explicitly. Since $\underline{w}$ is a generic word, the letter $III_0$ is recurrent. In particular, we can choose infinitely many indices $n$ such that $w_n = III_0$. For such $n$, we take the component $\gamma_n$ of $F_{\R}^{-n}(\Lc)$ whose orbit coincides with the first $n$ words. By Theorem \ref{thm_markov}, there is only one such component as the entries of the matrices $M_*$ and its transpose are only zeros and ones. 
This means that on a neighborhood of $p=(0,0)$ or radius $R$, we have: 
 $$ \lim_{R \rightarrow 0} \dfrac{l(\gamma_n \cap B(p,R))}{2R} = 1.  $$
 The sequence of curves $\gamma_n$ converges to a curve $\gamma_{\underline{w}}$ as $\gamma_n$ is vertical for the cone field and since this cone field is invariant.          
\end{proof}

We now construct a current associated to an arbitrary generic word. 

\begin{prop} \label{prop_coding_single} For any generic admissible word $\underline{w}$ in $A$, there exists a vertical curve $\gamma_{\underline{w}}$ in  $\cL$ whose orbit under $F$ coincide with $\underline{w}$.  
\end{prop} 
 
\begin{proof} Since $\underline{w}$ is generic, we can take  the first moment $n$  for which the subword $II, III_0$ appears. Indeed, if this sequence never appears, then $\underline{w}$ is not generic as this means that the word subword before this first occurrence is of the form $(w_{n-k},\ldots, w_{n-1} ,w_n, w_{n+1}) = (III_0 , I , IV , \ldots,  IV , III_0)$.      
As this must happen everytime $III_0$ appears, the word $\underline{w}$ is not generic. 
Assume that $i_n$ is the $n$-th moment the subword $ III_0$ appears, let $\tilde{w}$ the shift of the $\underline{w}$ by $n$ letters. 
We set $\gamma_{\underline{w}}$ to be the component of the pullback of $\gamma_{\tilde{w}}$ by $F^{i_n}$ whose coding coincide with $\underline{w}$ for the first $i_n$ iterates, this exists by Theorem \ref{thm_markov}. We then take a weak limit as $n \rightarrow +\infty$ 
\end{proof}

Recall that $\mathcal{L}$ is a lamination away from $\bIz \cup \bIinf$ so in particular, so if two curves have a distinct coding for their orbit, then they are distinct curves because for some iterate, the curves are not in the same domain any more. 
We analyze the situation near the intervals and control the difference between these curves. Recall that the difference of two currents $T_1, T_2$ can be controlled via the flat norm $N_F(T_1 - T_2)$ which is defined as: 
\begin{equation*}
N_F(T_1 - T_2) = \inf \{  M(A) + M(B) \ | \ T_1 - T_2 = A + \partial B \} ,
\end{equation*} 
where the infimum is taken over all such decompositions (see \cite[p.41]{morgan_gmt}).

\begin{lem}
Consider two admissible words $\underline{w},\underline{w'}$ whose first $n$ letters coincide. Let $p$ be the intersection point between $\gamma_{\underline{w}}, \gamma_{\underline{w'}}$ on $\bIz$, then for a  bump function $\chi$ near $p$,
\begin{equation}
N_F(\chi [\gamma_{\underline{w}}] - \chi[\gamma_{\underline{w'}}]) \leqslant  C \theta^n,
\end{equation}
 where $\theta < 1$, where $C$ is a uniform constant and where $N_F$ is the flat norm.
\end{lem} 

\begin{proof}
This is a consequence of the contraction of the cone field under the pullback of $F$ at the point $p$. 
\end{proof}

We obtain some weaker estimates outside of the Chebyshev intervals as follows. 

\begin{prop} \label{prop_flat_norm_est} For any $k \leqslant n$, for any admissible word $w$ of length $n$ for which the symbol $III_0$ appears $k$ times, for any compactly supported function $h$ on $\R^2$, there exists a constant $C_{w,h}>0$ (depending only on $w$ and $h$) and $\theta<1$ such that for any two admissible words $w',w''$ with the same prefix $w$, one has:
\begin{equation}
N_F(h ([\gamma_{w'}] - [\gamma_{w''}])) \leqslant C_{w,h} \theta^k. 
\end{equation} 
Moreover, the same statement holds for any two vertical curves in $\mathcal{L}$ whose orbit coding under $F$ coincide in the first $n$ letters and for which the letter $III_0$ appears $k$ times.
\end{prop}

\begin{proof} The proof is general and applies to any two vertical curves in $\mathcal{L}$. For simplicity, let us assume the two curves are $\gamma_{w'}, \gamma_{w''}$. 
Denote by $s$ the shift in one letter on words.  
Let $p,q$ be the first and last moment among the indices $i$ such that $w_i = III_0$. Observe that  the image of $\gamma_{ w'}$ by $F^{q-1}$ is exactly $\gamma_{s^{q-1} (w')}$. In particular, the curves $F^{p-1} \gamma_{w'} ,F^{q-1} \gamma_{w'} , F^{p-1} \gamma_{w''} ,F^{q-1} \gamma_{w''}$ are all of type $III_0$.
Recall that we have defined in Proposition \ref{prop_laminar_good_region} a region $K_\epsilon$ bounded by two curves in $\mathcal{L}^*$ that by definition contains all curves of type $III_0$. By Proposition \ref{prop_expansion_kob}, the Kobayashi length of the horizontal curves for the cone field grows  exponentially fast under the first return map.
Set  $\gamma' = F^{p-1} \gamma_{w'}$ and $\gamma'' = F^{p-1} \gamma_{w''}$. One has $\mathcal{R}^k \gamma' = F^{q-p} \gamma'$ and $\mathcal{R}^k \gamma'' = F^{q-p} \gamma''$.  
Consider $V$ the union of all the horizontal segment joining $\gamma'$ and $\gamma''$ contained in the support of $h$. By construction, this region return $k$ times to $K_\epsilon$, choose any compact set $V'$ in $V$ which is mapped to a compact set $V''$ in $K_\epsilon$.
Any horizontal segment $I$ in $V''$ joining $\gamma'$ with $\gamma''$ is mapped to a horizontal curve for the cone field, using the comparison between the angular metric with the Kobayashi pseudo metric in $V', V''$, we have
\begin{equation*}
l(F^{q-p}(I)) \geqslant C M^k l(I),
\end{equation*}
where $M>1$
In particular, using Fubini's theorem, the area of the $V \cap V'$ is bounded by: 
\begin{equation} \label{eq_decay_area}
Area(V\cap V') \leqslant  \dfrac{1}{C M^n} C' Area(F^{q-p}(V')) \leqslant \dfrac{C'}{C M^k} Area(K_\epsilon),
\end{equation}
where $C'>0$ is a constant which controls the distortion between the angular length  of a horizontal line and a horizontal curve in $K_\epsilon$. 
Similarly, the length of the boundary horizontal curve $I$ in $\partial V$ is controlled using the same estimates: 
\begin{equation} \label{eq_decay_length}
l (I ) \leqslant \dfrac{C'}{C M^n} l(F^{q-p}(I)) \leqslant \dfrac{C''}{M^k}, 
\end{equation}   
where $C''$ is a uniform constant.  
The two inequalities \eqref{eq_decay_area} and \eqref{eq_decay_length} yields: 
\begin{equation*}
N_F(h [\gamma'] - h[\gamma'']) \leqslant \dfrac{C'''}{M^k},
\end{equation*}
where $C'''>0$ is a uniform constant. We then conclude by taking the pullback by $F^{q-1}$.
\end{proof}

We investigate the properties of the associated function $\underline{w} \mapsto \gamma_{\underline{w}}$.
We shall define a transfer operator for the  first return map on the symbol $III_0$. 
Consider $\tilde{A}$ the set containing subwords $w \in A^{\N}$,   starting with $III_0$, such that no other symbol is equal to $III_0$  and such that the concatenation $w III_0$ is also admissible. 
We set $\tilde{A}^{\N}$ the words in $\tilde{A}$, which can be also embedded inside $A^{\N}$. 
Fix $\theta < 1$. We recall the definition of Holder regular functions as in \cite[p. 12]{parry_pollicott}.

A real function $g$ on $\tilde{A}^{\N}$ is $\theta$-\textit{Holder regular}   if it is continuous and there exists a constant $C>0$ such that for any $n$ and any two words $w',w'' \in \tilde{A}^{\N}$ coinciding in the first $n$ subwords, one has: 
\begin{equation*}
\dfrac{|g(w')- g(w'')|}{\theta^{n}} \leqslant C.
\end{equation*}  
A function from $A^{\N}$ to $\R$ induces a function on $\tilde{A}^{\N}$ via the embedding from $\tilde{A}^{\N}$ into $A^{\N}$. 
This space is denoted $\mathcal{F}_{\theta}$ and it is a Banach space when endowed with the norm: 
\begin{equation}
||g||_{\theta} = \sup |g| + \sup_n \sup \left \{\dfrac{|g(w')- g(w'')|}{\theta^{n}} \ | \ w_i'=w_i'' \in \tilde{A} , \forall i \leqslant n \right  \} 
\end{equation}
We extend the notion of Holder regularity to current valued functions, namely  any function $g \colon A^{\N} \to \mathcal{D}^1(\R^2) $ is in $\mathcal{F}_\theta'$ if for any test smooth $1$-form $\alpha$ with compact support, the function $w \mapsto \langle g(w),\alpha \rangle$ belongs to $\mathcal{F}_\theta$. 

Recall that the curves $\gamma_w$ are constructed by pulling back the curve $\Lc$ and taking the component with the appropriate coding. 
As a corollary of the above statements, we show that the resulting function is universal, in the following sense.

\begin{cor} \label{cor_def_gamma} The function $ \gamma : \underline{w} \mapsto  [\gamma_{\underline{w}}]$ belongs to $\mathcal{F}_\theta'$ and its value for $m$-almost every point $w$ does not depend on the choice of any starting vertical curve in $\mathcal{L}$, where $m$ is the invariant measure associated to the matrix $M_*$. 
\end{cor}

\begin{proof}
 Proposition \ref{prop_flat_norm_est} shows that the second term in the definition of the norm is bounded when one pairs with a compactly supported test-form. As for the first term, it follows from the fact that   the mass of each element $[\gamma_w]$ is one.
 If one constructed another function $\gamma'$ at generic points by pulling back another vertical curve $V$. Then Proposition \ref{prop_flat_norm_est} shows that for generic word $w$, the flat norm $N_F(h[\gamma_w] - h[\gamma'_{w}])$ tends to zero, hence $[\gamma_w]$ and $[\gamma'_w]$ coincide. 
\end{proof}

\subsection{Transfer operator, first return and decay of correlations}
\label{section_transfer}

We shall relate the properties of the function $\gamma$ defined previously with a transfer operator associated to the shift. Then we investigate the spectral properties of the transfer operator associated to a first return map. 

Consider the operator $L$ given by: 
\begin{equation}
L g(w) = \sum_{\substack{a \in A \\ m_{a, w_1} \neq 0}} \dfrac{1}{\delta} g(a w).
\end{equation}
Denote by $C$ the cylinder associated to the symbol $III_0$. 
Similarly, we consider $R$  the transfer
operator such that for any function $g \colon C \to \mathcal{D}^1(\R^2)$: 
\begin{equation}
R g(w) =  \sum_{\alpha \in \tilde{A}} \dfrac{1}{\delta^{|\alpha|}} g(\alpha w),
\end{equation} 
where $w= (w_1, \ldots, w_n , \ldots) \in C$.

The operator $L$ is the transfer operator associated with the shift $\tau$ while $R$ is associated to the first return map $\tau_{C}$ to the symbol $III_0$.

\begin{prop} \label{prop_main_transfer} Given any curve $V$ of type $III_0$ in $\mathcal{L}$.
Let $\gamma_n \colon  A^{\N} \to \mathcal{D}^1(\R^2)$ be the  constant function on each cylinder of length $n$ such that for each admissible word $\underline{w} \in A^{\N}$, $u_n$ is equal to  the component $\gamma_{\underline{w},n}$ of $F^{-n}(V)$ whose orbit coincides with the first $n$ letters of $\underline{w} $.
The functions $\gamma : \underline{w} \mapsto [\gamma_{\underline{w}}] \in \mathcal{D}^1(\R^2)$ and $\gamma_n \colon  A^{\N} \to \mathcal{D}^1(\R^2)$ satisfy the following properties: 
\begin{enumerate}
\item[(i)] one has for any $\underline{w} \in A^{\N}$: 
\begin{equation}
L \gamma (\underline{w}) = \dfrac{1}{\delta} F^\sharp\gamma (\underline{w}),
\end{equation}
\item[(ii)] For any compactly supported, smooth $1$-form $\alpha$, the functions $ \langle \gamma_n, \alpha\rangle$ converge in $\mathcal{F}_\theta$ toward $\langle \gamma,\alpha\rangle$ as $n$ tends to $\infty$. 

\item[(iii)] one has for all $n$ and all admissible $\underline{w} \in A^{\N}$: 
\begin{equation}
L \gamma_{n+1} (\underline{w}) = \dfrac{1}{\delta} F^{\sharp} \gamma_n(\underline{w}).
\end{equation}
\end{enumerate}
\end{prop}

\begin{proof}
Observe that assertion (ii) and (iii) imply the first one. Assertion (iii) follows by definition. Since for each word $w$ of length $n$, the restriction of $\gamma_n$ to the cylinder $C_w$ is equal to the positive current of integration along the component of $F^{-n}(V)$ whose itinerary is given by $w$. Its pullback by $F$ is are  all the components of $F^{-n-1}(V)$ whose itinerary are $aw$ where $aw$ is admissible and $a \in A$, which correspond exactly the value of $\gamma_{n+1}$ on the cylinder $C_{aw}$. We then conclude by dividing by $\delta$.  As for assertion (ii), it follows directly from Proposition \ref{prop_flat_norm_est}.    
\end{proof}

We endow  the space $A^{\N}$  with a natural probability measure $m$ (see e.g \cite[Remark 1 p.27]{parry_pollicott}) defined  as follows. Let $v= (v_{I}, \ldots, v_{IV}) $ be the eigenvector for the matrix $M$ associated to the eigenvalue $\delta$, then for any cylinder $C_{\alpha}$ where $\alpha = (\alpha_1, \ldots , \alpha_n) \in A^n$ is an admissible word, one has:
\begin{equation*}
m(C_{\alpha}) = \dfrac{v_{\alpha_1}}{ v_{\alpha_n} \delta^{n-1}} p_{\alpha_n},
\end{equation*}
where $p = (p_I, \ldots, p_{IV})$ is the normalized eigenvector for the stochastic matrix $(p_{ij})_{i,j \in A} = (m_{ij} v_i/v_j)_{i,j \in A} $.
We thus get a probability space $(A^{\N}, m)$ on which the shift $\tau$ acts.
The measure $m$ is not invariant by the transpose of $L$ but by a conjugate $\tilde{L}$. 
Consider the conjugate matrix $P = D^{-1} M_* D$ where $D$ is the diagonal matrix whose entries are the one of the eigenvector  $(v_i)_{i\in A}$. Then the transfer operator associated to $P$, denoted $\tilde{L}$ is conjugate to $L$ and the measure $m$ is invariant by $^t \tilde{L}$. 
As a consequence, the measure $\mu=\ ^t D^{-1} m$ is finite and invariant by $ ^t L$.  

\smallskip

 The total space can be partitioned into cylinders of length $1$ such that the shift maps each cylinder $C_{\alpha_1}$ into the union of all the cylinders $C_\beta$ for which $\alpha_1 \beta$ is admissible. Moreover, the restriction of the shift to each cylinder is invertible onto its image. 
In other words, the shift $\tau$ acting on $(A^{\N}, m)$ is a \textit{Markov map}. We consider $\tau_{III_0}$ the return map to the symbol $III_0$. The following statement is a direct application of some results due to Saarig, Gouezel on the decay of correlations for certain transfer operators.  

\begin{thm} \label{thm_decay_correlation} Let $\tau $ be a the shift map acting on a measured space $(A^{\N},\mu)$ associated with the matrix $M_*$. Then for any $0 <\theta<1$, for any H\"older function $f \in \mathcal{F}_\theta$ supported in the cylinder $C_{III_0}$ and for any $g \in L^{\infty}(\mu)$, one has: 
\begin{equation*}
\lim_{n\rightarrow +\infty }\int (L^n f)(w) g(w) d\mu(w) = \int fd\mu \int g dm,
\end{equation*}
where $m$ is the probability measure associated with the matrix $M$. 
\end{thm} 

\begin{rem} Observe the slight discrepancy with usual statements on the convergence of the correlation. This discrepancy comes from the fact that $L$ is not a normalized operator, i.e the constant function $1$ is not invariant by $L$, so its natural measure is $\mu$. 
\end{rem}

\begin{proof} We first observe that since the symbol is finite and since  $m(C_{III_0})>0$, the image $\tau_{III_0}(C_{III_0, \alpha , III_0})$ where $\alpha$ is any admissible word in $A^{\N}$ with  $III_0 \alpha III_0$ admissible has mass $\geqslant \epsilon >0$. 
This shows that the first return map $\tau_{III_0}$ has the big image property (see e.g \cite[p.52]{gouezel} for a definition).

\smallskip 

We now study the properties of the potential associated to the transfer operator and its iterate. 
Let us consider $\phi_{k}$ the Radon-Nikodyn derivative at $\underline{w} = w_1 \ldots w_n \ldots \in A^{\N}$ and any integer $k$
\begin{equation} \label{eq_radon_nikodyn}
\phi_k (\underline{w}) = \dfrac{dm}{d \tau^{k}_* m} (\underline{w})= \dfrac{v_{w_1} \delta^{k}}{v_{w_{k+1}}}.  
\end{equation}    
In particular, this shows that $\phi_k$ is piecewise-constant, hence it is $\theta$ H\"older for any $\theta \in (0,1)$. 

\smallskip 

We now compute the rate of decay of the first return time $\varphi$ toward $III_0$. This is a consequence of the following lemma. 
\begin{lem} \label{lem_return_decay} If $\varphi$ is the first return time toward $III_0$, then $m( \{\varphi > N \}) \leqslant C \alpha^N$ for some uniform constant $C>0$ and for $\alpha <1$. 
\end{lem}

\begin{proof} Let us consider the generating function for the first return time to $III_0$, we set $R(z) = \sum m(\{\varphi =k \}) z^k$. 
We consider the generating function for the return  time to $III_0$, given by $T(z) = \sum m(\{w_1 = w_n = III_0\}) z^n$. 
Since each path starting at $III_0$ and returning to $III_0$ in $n$-step can be decomposed into a path returning to $III_0$ in $k\leqslant n-1$ step and a first return in $n-k$ step, we obtain the equation $R(z) = 1 - \dfrac{1}{T(z)}$. 
Observe that the general term of the series $T$ has an asymptotic of the form: 
\begin{equation}
m(\{w_1 = w_n = III_0\}) = \dfrac{ p_{III_0}\langle M^n e_{III_0} , e_{III_0}\rangle}{\delta^n} =  C + O\left (  \left (\dfrac{\delta- \epsilon}{\delta} \right )^n\right ),
\end{equation}
where $\epsilon > 0$ and $C$ is a constant. 
This shows that $T(z)$  is meromorphic with a pole of order $1$ at $z=1$ and we get that $R(z)$ is holomorphic at $z=1$ and in particular, the series $R(z)$ converges for $|z| > 1 + \epsilon'$ with $\epsilon'>0$. 
This shows that $m (\{ \varphi =n \}) \leqslant C' / (1+ \epsilon')^n $ for some constant $C'>0$ and we obtain that $m( \{\varphi > N\}) = O ((1+\epsilon')^{-N})$, as required. 
\end{proof}

Lemma \ref{lem_return_decay} shows that the measures $m(\{ \varphi > N\})$ decay faster than $1/N^\beta$ for $\beta >1$.  
We are thus in the setting of \cite[Theorem 6.9]{gouezel}, which gives: 
\begin{equation*}
\left  |\int (\tilde{L}^n f)(w) g(w) dm  - \left ( \int f dm \right ) \left (\int g dm\right ) \right | \leqslant \dfrac{C}{n^{\beta}} || f||_{\theta} ||g||_\infty .
\end{equation*}
Let us apply this to $L$ instead of $\tilde{L}$. Let $D$ the operator induced by the diagonal matrix whose entries are given by $(v_i)$ the components of the eigenvector of $M$. 
We have using the previous estimate together with the definition of $\mu$: 
\begin{multline*}
\lim_{n\rightarrow +\infty } \int (L^n f) g d\mu =\lim_{n\rightarrow +\infty } \langle (L^n f) g, \mu \rangle =\lim_{n\rightarrow +\infty }  \langle (D \tilde{L}^n D^{-1} f) g , \ ^t D^{-1} m \rangle \\= \lim_{n\rightarrow +\infty } \langle (\tilde{L}^n D^{-1 } f ) g  , m \rangle  = \langle D^{-1} f ,m \rangle \langle g , m  \rangle = \langle f, \  ^t D^{-1} m \rangle \langle g, m \rangle = \langle f, \mu \rangle \langle g, m \rangle,
\end{multline*}
where we write the pairing $\langle f, \mu \rangle$ as the integral $\int f d\mu$. 
The above formula yields the desired conclusion. 
\end{proof}

\subsection{Equidistribution of vertical curves (Proof of Theorem E)}

Following the construction in Proposition \ref{prop_coding_single} as well as the definition of the invariant measure $\mu$ for $L$, 
we  define the limiting candidate $T_{\R}$ by setting: 
\begin{equation} \label{eq_real_current}
T_{\R} := \int \gamma(w) d\mu(w).
\end{equation}

We state the main properties of this current. Recall that we have defined  $\mathcal{L}^*$ as the subset of leaves of $\mathcal{L}$ that are regular (of type $\mathcal{V}_{II}, \mathcal{V}_{III_{\pm}}$).  
\begin{prop} \label{prop_real_current} The current $T_{\R}$ is a real positive woven current, invariant by $F^{\sharp}/\delta$ i.e $F^{\sharp} T_{\R} = \delta T_{\R} $. Moreover, this current does not charge any leaf of $\mathcal{L}^*$. 
\end{prop}

\begin{rem} Away from the horizontal line through the origin, the current $T_{\R}$ is geometric in the sense of Ruelle and Sullivan \cite{ruelle_sullivan}.
\end{rem}

\begin{proof}
The invariance follows directly from assertion (i) of Proposition \ref{prop_main_transfer} and the fact that the measure is invariant by $ ^t L$. The positivity comes from the fact that it is a sum of currents of integration along positively oriented curves in $\mathcal{L}$. Finally, the fact that it does not charge any leaf of $\mathcal{L}^*$ comes from the fact that the measure $\mu$ puts no mass on any type $\mathcal{V}_{II}$ or $\mathcal{V}_{III,\pm}$ curves.
\end{proof}

For any subset $C\subset A^{\N}$, we denote by $\chi_C$ the characteristic function on that subset. 

\begin{thm} \label{thm_equid_real}
Let $V$ be any vertical curve in $\mathcal{L}$, we consider the  positively oriented current of integration
 $[V]$ on it. Then there exists a constant $M>0$ such that
\begin{equation}
 \lim_{n\rightarrow +\infty} \dfrac{\Fs^n[V]}{\delta^n} =  C T_{\R}.  
\end{equation}
Precisely, if $V$ is a curve of type $a \in A$, then $C = v_a$, where $v_a$ is the coefficient of the normalized eigenvector $v= (v_i)_{i\in A}$ of $M_*$. 
\end{thm}

\begin{proof} Fix $\epsilon >0$. 
Observe that by Proposition \ref{prop_bound_mass},
 the currents $T_n = \Fs^n [V]/\delta^n$ admit some weak limit along particular subsequence. 
  We follow the notation and construction of Proposition \ref{prop_main_transfer}. 
  The curve $V$ has of certain type $a_1 \in A$. 
  Recall that $\gamma_0$ is the constant on the cylinder $C_{a}$ and equal to $[V]$ on $C_{\alpha}$ and zero elsewhere. 
  Fix a compactly supported smooth $1$-form $\alpha$. 
We have by assertion (iii) of Proposition \ref{prop_main_transfer}: 
\begin{equation*}
\left \langle \dfrac{\Fs^n [V]}{\delta^n} ,\alpha \right \rangle = \langle \dfrac{1}{\mu(C_{a})} \dfrac{\Fs^n}{\delta^n}  \int_{C_{a}} \gamma_0(w) d\mu(w) , \alpha\rangle =  \dfrac{1}{\mu(C_{a})} \int_{C_{a}} \langle L^n \gamma_n , \alpha \rangle d\mu.
\end{equation*}
Moreover, since $\langle \gamma_n ,\alpha\rangle$ converges in $\mathcal{F}_\theta$ toward $\langle \gamma,\alpha \rangle$ and since $L$ is a bounded operator (of norm bounded by $1$ on $L^\infty(m)$), there exists $N_1$ such that for any $n\geqslant N_1$, one has 
\begin{equation*} 
|\langle L^n \gamma_n - L^n\gamma, \alpha \rangle| \leqslant \epsilon.
\end{equation*}
We thus get using the two previous equation: 
\begin{equation} \label{eq_approx_1} 
\Delta_1 = \left  | \left \langle \dfrac{\Fs^n [V]}{\delta^n}  - \dfrac{1}{\mu(C_{a})} \int_{C_{a}} L^n \gamma d\mu , \alpha \right  \rangle \right | \leqslant \epsilon,
\end{equation}
 for all $n \geqslant N_1$. 
Recall that we denoted by $\varphi$ the first return time function to $III_0$. Set $g = \chi_{C_{a}}$, we have for all $N \geqslant 0$ and using the fact that $\gamma = \sum_{k\leqslant N}  \gamma \chi_{\varphi =k} + \gamma \chi_{\varphi > N}$: 
\begin{equation*}
 \int   L^n  \langle\gamma , \alpha\rangle g  d\mu    =  \sum_{k\leqslant N}  \int  (L^{n-k+1}  L^{k-1} \langle \gamma \chi_{\varphi = k} ,  \alpha \rangle) g d\mu  + \int \langle L^n(\gamma \chi_{\varphi > N}) , \alpha \rangle g d\mu.  
\end{equation*}
Observe that $ \left | \int L^n \langle  \gamma \chi_{\varphi > N} , \alpha \rangle g d\mu \right |  =  \left | \int \langle \gamma \chi_{\varphi> N} , \alpha \rangle  g\circ \tau^n d\mu \right | \leqslant C_1|| g||_\infty \mu(\{ \varphi > N \} )$ with $C_1 > 0$ and the previous inequality gives for $N$ large enough: 
\begin{equation} \label{eq_approx_2}
\Delta_2 = \left  | \int  \langle L^n \gamma, \alpha\rangle  g d\mu - \sum_{k\leqslant N}  \int L^{n-k+1} \langle L^{k-1} (\gamma \chi_{\varphi = k}) ,\alpha\rangle g d\mu  \right | \leqslant C_1 \mu(\{\varphi > N \}) \leqslant \epsilon, 
\end{equation}
where we have used the exponential decay of the first return time from Lemma \ref{lem_return_decay} and the fact that $\mu = ^t D^{-1 }m$. 
From now on, $N$ is fixed and chosen so that the previous inequality holds. 
By construction, the functions $ \gamma \chi_{\varphi=k} $ are supported on the union of  cylinders $\cup_{|\beta| = k-1 }C_{\beta III_0}$. As a consequence,  $L^{k-1} \gamma \chi_{\varphi = k} $ is supported in $C_{III_0}$, so Theorem \ref{thm_decay_correlation} gives for all $k \leqslant N$ and all $n \geqslant N_2 \geqslant N$:  
\begin{equation*}
\left | \int L^{n-k+1} \langle L^{k-1} (\gamma \chi_{\varphi =k}), \alpha \rangle g d\mu - \int \langle L^{k-1} (\gamma\chi_{\varphi =k }) , \alpha \rangle d\mu \int g dm \right | \leqslant \dfrac{\epsilon}{ N Tr(M^N)}.  
\end{equation*}
Summing over all $k\leqslant N$ gives for all $n \geqslant N_2$: 
\begin{equation} \label{eq_approx_3}
\Delta_3 = \left  | \sum_{k \leqslant N} \left ( \int L^{n-k+1} \langle L^{k-1} (\gamma \chi_{\varphi = k}), \alpha \rangle g d\mu   -  \int \langle L^{k-1} (\gamma \chi_{\varphi=k}) , \alpha \rangle d\mu \int g dm \right )   \right | \leqslant \epsilon,
\end{equation}
where we used that $\sum_{k\leqslant N} \sum_{|\beta|=k-1} \epsilon/ (N Tr(M^N)) \leqslant \sum_{k\leqslant N}\epsilon  Tr(M^k)/(NTr(M^N)) \leqslant \epsilon $.
Using the invariance of $\mu$ by $^t L$, we have that 
\begin{equation*}
T_{\R} - \sum_{k\leqslant N} \int L^{k-1} (\gamma \chi_{\varphi = k}) d\mu = \int \gamma \mu - \int \gamma \chi_{\varphi \leqslant N} d\mu = \int \gamma \chi_{\varphi > N} d\mu.
\end{equation*}
  In particular, the mass of the difference is controlled: 
  \begin{equation*}
  Mass \left (T_{\R} - \sum_{k\leqslant N} \int  L^{k-1} (\gamma \chi_{\varphi = k}) d\mu \right ) \leqslant \mu (\{ \varphi > N \}) \leqslant \epsilon   
  \end{equation*}
  Pairing with $\alpha$ then yields: 
  \begin{equation} \label{eq_approx_4}
  \Delta_4 = \left ( \int g dm \right ) \left | \langle T_{\R} , \alpha \rangle -  \sum_{k\leqslant N} \int  \langle L^{k-1} (\gamma \chi_{\varphi=k}) , \alpha \rangle d\mu  \right |  \leqslant \epsilon ||\alpha ||_\infty m(C_{a})
  \end{equation}
Combining \eqref{eq_approx_1}, \eqref{eq_approx_2}, \eqref{eq_approx_3} and \eqref{eq_approx_4} together with the triangle inequality gives: 
\begin{equation*}
 \left | \langle \dfrac{\Fs^n [V]}{\delta^n}  - \dfrac{m(C_{a})}{\mu(C_{a})} T_{\R}  ,\alpha \rangle\right | \leqslant \Delta_1 + \dfrac{\Delta_2 + \Delta_3 + \Delta_4}{\mu(C_{a})}  \leqslant \left (1 + \dfrac{2}{m(C_{a})} + ||\alpha||_\infty \right ) \epsilon.
\end{equation*}
The above inequality allows us to conclude as $m(C_{a})/\mu(C_{a}) = v_{a}$. 
\end{proof}

\subsection{Closedness of $T_{\R}$}

Let $\pi \colon X \to \bP^2$ be the morphism induced by the blow-down of the four exceptional divisors in $X$. In this section, we shall view $\pi$ as a real mapping between $X(\R)$ and $\bP^2(\R)$. 
Given a vertical curve $C$ in $\bP^2(\R)$, its strict-transform admits a compatible orientation and we denote by $[\pi^o C]$ the induced current of integration on $X(\R)$. 
The main result in this section concerns the closedness of the strict transform $\pi^o T_{\R}$. 

\begin{cor} \label{cor_closed_TR} The strict transform $\pi^o T_{\R}$ is a closed current in $X(\R)$. 
\end{cor}

\begin{proof}
Observe that the strict transform by $F_X$ preserves closed curves in $X(\R)$. In particular, any weak limit of the strict transform by $F_X^n$ of $\Lc$ normalized by $\delta^n$ are closed. One observes that the strict transform by  $F_X^n$ of $\Lc$ coincide with the strict transform of $\Fs^n[\Lc]$ by $\pi$. By Theorem \ref{thm_equid_real}, the sequence $\Fs^n[\Lc]/\delta^n$ converges to $T_{\R}$, so the strict transform of $T_{\R}$ by $\pi$ is a limit of closed currents, which remains closed, as required. 
\end{proof}

\section{The real analog of the Sibony current}
\label{section_real_sibony}
We define an analog of the Sibony current $S$, which is a real current whose behaviour mimics the complex one. 
We consider: 
\begin{equation} \label{eq_real_sibony}
S_{\R} = \dfrac{1}{12} \sum_{k \geqslant 0} \dfrac{1}{2^k} \Fs^k [\Lc]
\end{equation}

We state its main features. 

\begin{prop} The current $S_{\R}$ is a real woven current, supported on $\mathcal{L}$ and which satisfies the functional equation: 
\begin{equation}
\left (Id - \dfrac{F^{\sharp}}{2} \right ) S_{\R} = \dfrac{1}{12} [\Lc]. 
\end{equation}  
\end{prop}

\begin{proof}
The functional equation follows from the definition of $S_{\R}$. As for the woven property, since $\Lc$ is a vertical curve for the cone field, of type $III_0$, Theorem \ref{thm_markov} shows  that the mass of $\Fs^n [\Lc]$ grows like $\delta^n = o((2 - \epsilon)^n)$, hence the series is summable and the current is woven. 
\end{proof}

We combine the above Proposition with the results from the last section. 
\begin{cor} \label{cor_asymp_SR} For any smooth $1$-form with compact support $\alpha$, the following holds: 
\begin{equation}
\langle S_{\R}, \alpha \rangle = \dfrac{1}{12}  \sum_{k=0}^n   \left \langle \dfrac{\Fs^k}{2^k} [\Lc], \alpha \right \rangle +  \dfrac{  v_{III_0}}{12 (1- (\delta/2))} \dfrac{\delta^n}{2^n} \langle T_{\R}, \alpha\rangle + o \left ( \dfrac{\delta^n}{3^n} \right ),
\end{equation}
where $v_{III_0}$ is the coefficient corresponding to the entry $III_0$ of the eigenvector of $M_*$ for the largest eigenvalue $\delta$. 
\end{cor}

\begin{proof}
Let us truncate the series defining $S_{\R}$ at the $n$-th stage, the tail of the series is given by $5/12 \sum_{k \geqslant n} 2^{-k} \Fs^k [\Lc]$. 
By Theorem \ref{thm_equid_real} applied to the curve $\Lc$, one has: 
\begin{equation*} 
 \dfrac{\Fs^k [\Lc]}{\delta^k} = v_{III_0} T_{\R} + R_k,
\end{equation*}
where $R_k$ is a sequence of currents converging weakly to zero.
 Pairing with $\alpha$ in the previous equation together with taking the sum as $k\geqslant n$ yields: 
 \begin{equation*}
\dfrac{1}{12} \sum_{k\geqslant n}  \langle \dfrac{\Fs^k[\Lc]}{2^k} , \alpha\rangle  = \dfrac{1}{12} \sum_{k\geqslant n} \dfrac{\delta^k}{2^k} ( v_{III_0} \langle T_R, \alpha\rangle + \langle R_k ,\alpha\rangle )= \dfrac{v_{III_0}\delta^n	}{12\cdot 2^n (1- (\delta/2))} \langle T_{\R} , \alpha \rangle  + o\left ( \dfrac{\delta^n}{2^n} \right ), 
 \end{equation*}
 where we have used the fact that $\langle R_k , \alpha\rangle$ converges to zero as $k \rightarrow +\infty$. The proof is done. 
\end{proof}

\section{Slicing of real and complex currents}
\label{section_slicing}
\subsection{Slicing of the Green current}

Recall that one hand, one can define the slice by a line $L$ of the Green current as the $dd^c$ of the restriction of $G_F$ to $L$.
In particular, for any $h \in \R$, we let $\nu_h$ be the slice of $T_F$ by the line $\{y=h\}$.  
On the other hand, given any smooth kernel $\psi \colon \R \to \R$, the $(1,1)$-current 
\begin{equation}
R_\psi = \int [y=h] \psi(h) dh
\end{equation}
admits a continuous potential, given by $U_{\psi} =\int_{\R} \log(|y-h|) \psi(h)dh$. So the Bedford-Taylor product $  R_{\psi} \wedge T_F = dd^c (U_{\psi} T_F) $ 
is well-defined. 
This suggests that there exists a relationship between $\nu_h$ and the real smearing. 

\begin{prop} \label{prop_switch} For Lebesgue almost all $h\in \R$, the slice of $T_F$ by the line $\{y=h\}$ exists and one has: 
\begin{equation*}
R_{\psi} \wedge T_F = \int \nu_h \psi(h) dh. 
\end{equation*} 
\end{prop}

While we think this statement must be standard but for the sake of completeness, we give a proof. 
\begin{proof}
First let us fix $h \in \R$. 
We compare the measure $\nu_h $ with the measure given by $ \nu_h' =dd^c G_F \wedge dd^c \log |y-h|$. 
One first regularize these two functions by setting $G_j = \max (G_F , -j)$ and $u_j = \max (\log |y-h|, -j)$. 
The two sequences of functions $G_j$ and $u_j$ are decreasing plurisubharmonic functions converging toward $G_F$, $\log |y-h|$ respectively as $j\rightarrow +\infty$.
 As each of the functions $G_j, v_j$ are bounded, we have by \cite[Corollary 3.10]{demailly_agbook} the equality: 
 \begin{equation} \label{eq_commute}
 dd^c G_j \wedge dd^c v_j = dd^c v_j \wedge dd^c G_j
 \end{equation}
 for all $j$. 
 Using \cite[Corollary 4.3]{demailly_agbook} twice, the left hand side converges toward $\nu_h=T_F \wedge [y=h]$ while the right hand side converges toward $\nu_h' = dd^c \log |y-h| \wedge dd^c G_F$. 
We have thus shown that $ dd^c G_F \wedge dd^c \log |y-h| = dd^c \log |y-h| \wedge dd^c G_F$. 

\smallskip

Summing over $h$, one obtains: 
\begin{equation}
R_{\psi} \wedge T_F = \int dd^c(\log |y-h|) \wedge T_F \psi(h) dh = \int \nu_h \psi(h) dh.
\end{equation}

\end{proof}

We state now a sliced version of the equidistribution. 
\begin{prop} \label{prop_sliced_equidistribution}
For Lebesgue almost $h \in \R$, for $V = \Lf , \Lc$ one has: 
\begin{equation*}
\lim_{n\rightarrow +\infty}  [L_h] \wedge \dfrac{1}{\dyndeg^n} (F_X^n)^* [V]  = \nu_h,
\end{equation*}
where $L_h$ is the pullback of the  horizontal line $\{ y=h\}$ by $\pi$. 
\end{prop}

\begin{proof}
By Theorem \ref{thm_equidistribution}, the line $C$ equidistributes toward $T_F$ under pullback. Since the pullback $\dfrac{1}{\dyndeg^n} (F_X^n)^* [C]$ and $T_F$ are  positive currents, their coefficients are measures and the equidistribution result implies that the measure coefficients converge weakly.
In particular for any smooth kernel $\varphi$, since the potential $U_{\psi}$ is  continuous, the following limit holds: 
\begin{equation*}
\lim_{n\rightarrow +\infty} U_{\psi} \dfrac{1}{\dyndeg^n} (F_X^n)^* [C]  = U_{\psi} T_F, 
\end{equation*}  
in the sense of currents. 
Taking the $dd^c $ and applying Proposition \ref{prop_switch}, one obtains: 
\begin{equation*}
\lim_{n\rightarrow +\infty }R_{\psi} \wedge \dfrac{1}{\dyndeg^n} (F_X^n)^* [C] = R_{\psi} \wedge T_F = \int \nu_h \psi(h) dh. 
\end{equation*}
The left hand side can now be interpreted geometrically. Namely, each pullback of $C$ is a finite union of curves so defines a complex web. Moreover, these curves  are real symmetric and obey the vertical cone field induced by $\mathcal{C}$.
Denote by $\nu_{n,h}$ the measure $[L_h] \wedge \dfrac{1}{\dyndeg^n} (F_X^n)^* [C]$.
Let $\sigma_{h,h'} \colon L_h \to L_{h'}$ the holonomy from $L_h$ to $L_{h'}$ induced by the real lamination. 
We have for any function $\psi$ whose support is on $\R_+$, for any $h_0$ with $\psi(h_0) \neq 0$. 
\begin{equation}
R_{\psi} \wedge \dfrac{1}{\dyndeg^n} (F_X^n)^* [C] = \int \nu_{n,h} \psi(h) dh = \int (\sigma_{h_0,h})_* \mu_{n,h_0} \psi(h) dh
\end{equation}
   This gives that for all smooth kernel $\psi$, one has: 
   \begin{equation}
\lim_{n\rightarrow +\infty}   \int \psi(h) (\sigma_{h})_* \mu_{n,h_0} dh = \int \psi(h) \nu_h dh,
   \end{equation}
   and we claim that this implies that for almost all $h$, the measures $\nu_{n,h}$ converge to $\nu_h$. 
   Indeed, for any non-negative function $f(x,y) $ on $\R^2$, the previous formula gives 
\begin{equation}
\lim_{n\rightarrow +\infty} \langle \nu_{n,h}, f\rangle = \langle \nu_h, f \rangle  \in \mathcal{D}'(\R),
\end{equation}   
   where the convergence is as distributions in $h$. 
   This implies that the sequence of functions $h \mapsto \langle \nu_{n,h}, f \rangle $ converge almost surely to $h \mapsto \langle \nu_{h}, f \rangle $. Since this is for any non-negative functions $f$, this gives that $\nu_{n,h}$ converges to $\nu_h$ for almost all $h \in \R$, as required.  
\end{proof}

\subsection{Slicing of the real current $T_{\R}$}

The real analysis of $T_{\R}$ yields stronger results. Namely, due to the woven structure of $T_{\R}$, it is well structured and we can intersect it with a current which is transversal to the leaves of $T_{\R}$. 
Like the complex situation, there are instances for which real currents can have a well-defined wedge product. Since we don't have any potential theoretic way, one needs to proceed using theory of Hormander. 
A current can be viewed as a form whose coefficient are distributions. In our case, those coefficients are measures and the wedge product is well-defined if the coefficients can be multiplied. 
In \cite[Theorem 8.2.10]{hormander}, H\"ormander proved that distributions can be multiplied if  their \textit{wavefronts} satisfy a certain condition. 
We will use a special case of this result in the following way.

\begin{lem} Take the function $\gamma \colon A^{\N} \to \mathcal{D}^{1}(\R^2)$ provided by Proposition \ref{prop_coding_single}.  For any $h \in \R$ and any generic word $w \in A^{\N}$, the wedge product $\gamma(w) \wedge_{\R} [L_h]$ is well-defined and equals the dirac mass at the corresponding intersection point. 
\end{lem}

\begin{proof}
Fix a word $w \in A^{\N}$, the curve $\gamma_w$ is by construction vertical for the cone field $\mathcal{C}$, while $L_h$ is horizontal for the cone field. This means that at the intersection point $\gamma_w$ and $L_h$ are transversal. 
Choose two $\mathcal{C}^1$ arc-length parametrization $f,g \colon \R \to \R^2$ of $\gamma_w$ and $L_h$ such that $f(0) = g(0)$ is equal to the intersection point $p$ of $\gamma_w \cap L_h$. 
The map $u (x, y) = f(x) + g(y)$ defines a local diffeomorphism from a neighborhood of $(0,0)$ to a neighborhood $V$ of $p$. 
 Given any smooth function $\varphi$ supported on $V$ which is identically $1$ near $p$, the image of $u^{-1}_* ( \varphi [\gamma_w])$ and $u^{-1}_* (\varphi [L_h])$ are of the form: 
 \begin{equation}
 u^{-1}_* (\varphi [\gamma_w]) = (\varphi \circ u) \delta(y) dx , \ u^{-1}_* (\varphi [L_h]) = (\varphi \circ u) \delta(x) dy, 
 \end{equation}
 where $\delta$ denotes the dirac mass at zero. 
 By \cite[Example 19]{dang_helein_brouder}, the distribution $\delta(x)$ and $\delta(y)$ can be multiplied and their product is $\delta(x,y)$. 
 In particular, this shows that the product of $[\gamma_w] \wedge [L_h]$ coincides with  $u_* (\delta(0,0) dx \wedge dy)$. This is exactly the definition of the dirac mass at $p$.    
\end{proof}

\begin{prop} \label{prop_slice_TR} The wedge product $T_{\R} \wedge [L_h]$ is a well-defined  positive measure, satisfying 
\begin{equation*}
T_{\R} \wedge_{\R} [L_h] = \int \gamma(w) \wedge_{\R} [L_h] d\mu(w).
\end{equation*}
\end{prop}

\begin{proof} 
For finite sums, the wedge product with $[L_h]$ is linear, when it is valid. This shows that the approximations $\gamma_n$ which are constant on each cylinder satisfy: 
\begin{equation*}
\int \gamma_n(w) d\mu(w) \wedge_{\R} [L_h] = \int \gamma_n(w)\wedge_{\R} [L_h] d\mu(w). 
\end{equation*}
Since $\gamma_n$ converges toward $\gamma$ and the mass of $\gamma_n$ is bounded by $1$, the dominated convergence implies that the limit is linear, hence: 
\begin{equation*}
T_{\R} \wedge_{\R} [L_h] = \int \gamma(w) d\mu(w) \wedge_{\R} [L_h] = \int \gamma(w) \wedge_{\R} [L_h] d\mu(w), 
\end{equation*}
as required. 
\end{proof}

\begin{thm} \label{thm_continuous_slice} For any $h \in \R$,  let $\nu_h$ be the sliced measure $[L_h] \wedge T_F$. Then the function $h \in \R \mapsto \nu_h$ is continuous. Moreover, one has:
\begin{equation}
T_{\R} \wedge_{\R} [L_h] = T_F \wedge [L_h] = \nu_h. 
\end{equation} 
\end{thm}

\begin{proof}
We first establish the equality between the sliced measures. 
Recall that for each $n$, the pullback of the complex line $F^{-n}(\Lc(\C))$ intersects $L_h$ at only real points, these points coincide with the intersection of the real preimage $F^{-n}(\Lc) \cap \R^2$ with $L_h$. 
This shows that for all $n \in \N$, one has: 
\begin{equation*}
\dfrac{(F^n)^* [\Lc(\C)]}{\delta^n} \wedge [L_h(\C)] = \int \gamma_n(w) \wedge [L_h] d\mu(w). 
\end{equation*}
By Proposition \ref{prop_sliced_equidistribution}, the limit as $n \rightarrow +\infty$ gives for almost all $h \in \R$: 
\begin{equation*}
T_F \wedge [L_h(\C)] = \int \gamma(w) \wedge_{\R} [L_h] d\mu(w) = T_{\R} \wedge_{\R} [L_h].
\end{equation*}
Fix a typical point $h_0$ for which the previous equality holds. Let $h$ be an arbitrary point in $\R$, since the limit $T_F \wedge [L_h(\C)]$ is also invariant by holonomy from $L_{h_0}$ to $L_h$. We deduce that $T_F \wedge [L_h]$ is the image of $T_F \wedge [L_{h_0}]$ via the holonomy map. Moreover, by construction, the holonomy of $T_{\R} \wedge_{\R} [L_{h_0}]$ is exactly $T_{\R} \wedge_{\R} [L_h]$. 
In particular, the previous formula holds for all $h$. 
Since the right hand side of the equality is continuous in $h$, we deduce that $\nu_h$ is continuous. 
\end{proof}

\subsection{Slicing of the Sibony current}

Following the same notation as in the previous sections, we state a direct consequence of the previous results together with the asymptotic formula of $S_{\R}$. 

\begin{thm} \label{thm_slice_SR} For any $h \in \R$, denote by $L_h$ the horizontal line $\{ y=h\}$, then one has for any continuous function $\psi$ on the real line: 
\begin{multline*}
\langle S \wedge [L_h], \psi\rangle = \langle S_{\R} \wedge_{\R} [L_h],\psi \rangle \\ = \dfrac{1}{12} \sum_{k \leqslant n } \dfrac{1}{2^k} \langle \Fs^k[\Lc] \wedge_{\R} [L_h], \psi \rangle + \dfrac{ v_{III_0}}{12(1-(\delta/2))} \dfrac{\delta^n}{2^n} \int_{A^{\mathbb{N}}} \langle   [L_h]\wedge_{\R}\gamma(w),\psi\rangle d\mu(w)  + o\left ( \dfrac{\delta^n}{2^n}\right ).
\end{multline*}
\end{thm}

\begin{proof}
The second equality is a direct consequence of Corollary \ref{cor_asymp_SR} together with Theorem \ref{thm_continuous_slice} and Proposition \ref{prop_slice_TR}. We claim that the number of intersection points $N_{\R}$ between the real pullback for $F^n$ of $\Lc$ and the line $L_h$ coincides with the number of complex intersections $N_{\C}$ between the pullback  $(F_X^n)^* \pi^o\Lc$ with the strict transform of $L_h$. 
The number $N_{\C}$ is equal  to the intersection number $((F_X^n)^* \{\pi^o\Lc \}\cdot \{\pi^o L_h\} )$ where the intersection product in computed in $H^{1,1}(X(\C))$  
while the number of vertical curves for the real pullback by $F^n$ is $N_{\R}$ as each vertical curve intersects $L_h$ exactly once. 
By Theorem \ref{thm_magic}, each real curve in $\cL$ yields a complex cohomology class which intersect once the class $\{ \pi^o(L_h)\}$. This shows that $N_{\R} = N_{\C}$. In particular, all intersection between $S$ and $L_h$ are real and we get $S \wedge [L_h] =   [L_h]\wedge_{\R} S_{\R}$, as required.   
\end{proof}

\section{Application to the density of states measures}
\label{section_application}

In this section, we apply all the dynamical results to the initial spectral problem. 
The relationship between the spectral measures $\omega_h$ and the currents is given by the following statement. Recall from \S~\ref{section_sibony} that we have defined the Sibony current $S$, associated with $F$.  

\begin{prop} \label{prop_sibony_slice_spectral} For any $h \in \R$, the spectral measure $\omega_h$ is equal to the slice $S \wedge [L_h]$, where $S$ is the Sibony current associated with $F$.
\end{prop}

\begin{proof}
Recall that the density of states measures $\omega_h$ is the limit of the counting of certain polynomials $\tilde{P}_n(\cdot, h)$ as $n\rightarrow +\infty$ where $\tilde{P}_n$ is defined in Lemma \ref{lem_recursion}. By Theorem \ref{thm_sibony_spectral}, the integration current on $\tilde{P}_n =0$ normalized by $2^n$ converge toward the Sibony current $S$. In particular, the density of states measure $\omega_h$ coincide with the slide of $S$ with the line $L_h$.  
\end{proof}

\subsection{Almost parallel leaves near the line of fixed points}

We shall describe the regularity  of the leaves of $\mathcal{L}$ as they approach the point $p=(4,1)$. 
Recall that $F^2$ fixes the line $\Lf$ and acts as identity on it. 
The differental of $F^2$ is diagonalizable with $1$ and $8$ as eigenvalues. In particular, the leaves of $\mathcal{L}$ are almost parallel to the line $\Lf$, which is what we are going to show. 

Fix $m$ a point near $p$ and choose some inverse branches $m_n$ of $m$ by $F^2$. Assume that $m_n$ converges to a point on $\Lf$. Consider the set $\mathcal{W}^u(m)$ of points $q$ such that there exists a sequence $x_n$ with $d(x_n,m_n) \rightarrow 0$ for which $F^{2n}(x_n)=q$. 
 
\begin{prop} For all $u$ in a neighborhood of $p$, the set $\mathcal{W}^u(m)$ is a real manifold of dimension $1$. 
\end{prop}

\begin{proof} Fix a neighborhood $U$ of $p$. Locally, $F^2$ restricted to $F^{-2}(U)$ to $U$ is a diffeomorphism. The set $\mathcal{W}^u(m)$ is then an unstable set, which is a manifold (see e.g \cite[Theorem 4.1]{hirsch_pugh_shub}).   
\end{proof}

We fix a ball of radius $\epsilon$ near $p$ and denote by $\mathcal{L}(B(p,\epsilon))$ the leaves of $\mathcal{L}$ intersecting this ball. 

\begin{prop}
\label{prop_C1_convergence}
For some $\epsilon_0> 0$. If a sequence of leaves $\mathcal{L}_{n} \in \mathcal{L}(B(p,\epsilon_0))$ containing some point $x_n \in \mathcal{L}_n$ converging toward $p$, then  $\mathcal{L}_n$  
   converge with respect to the $\mathcal{C}^1$ topology toward the line $\Lf$.   
\end{prop}

\begin{proof} 
Denote by $\varphi$ the local diffeomorphism induced by $F^2$ on a neighborhood of $p$. 
First, the leaves of $\mathcal{L}$ are directed by the vertical cone field $\mathcal{C}$. Secondly, the direction of the cone field avoids a neighborhood of the orthogonal direction to $\Lf$ in a neighborhood of $p$, hence the leaves of $\mathcal{L}$ near $p$ can be viewed as local graphs over the line $\Lf$. 
 Recall that the vertical cone field is  decreasing strictly by $F^{-2}$, hence near the point $p$, it decreases uniformly and this shows that the slope of the leaves of $\mathcal{L}_n$ near $p$ converge uniformly toward the slope of $\Lf$ whenever $\mathcal{L}_n$ converges toward $\Lf$. 

We thus need to show that near $p$, the leaves $\mathcal{L}_n$ converge toward $\Lf$.  
We choose $\epsilon_0> 0$ so that for any leaf $L \in \mathcal{L}(B(p,\epsilon_0))$, $\varphi^{-1}(L)$ also remains in $\mathcal{L}(B(p,\epsilon_0))$. This is possible because $\Lf$ is transversally repelling for $F^2$. 
Since $\mathcal{L}_n$ contains a point $x_n$ converging to $p$, this shows that for $n$ large enough, $\mathcal{L}_n$ belongs to $\mathcal{L}(B(p,\epsilon_0))$. 
Fix $q \in \Lf\cap B(p,\epsilon_0)$.  
Denote by $\chi_{h,n,m}$ the holonomy map from $\mathcal{L}_n$ to $\mathcal{L}_m$ along the unstable foliations. Let $\chi_{L}$ be the holonomy map along the vertical foliation toward the transversal $L= \{ y=1 \}$. Since $x_n$ converges to $p$, then $\chi_L(x_n)$ also converges to $p$. Moreover, for each $n$, take  $y_n$ the intersection point  $\mathcal{W}^u(q) \cap L_n$, when it exists.   We claim  that $y_n $ converges toward $q$. Indeed, $\chi_L(y_n)= \chi_L(x_n)$ converges to $p$. Moreover, take $h$ the restriction $h = (\chi_L)_{| \mathcal{W}^u(q)}$ 
 then we have $ h \circ  \chi_{h,n,m} = \chi_L \circ (\chi_{h,n,m})_{|L} \circ  h$.   
 Applying the previous relation to the point $\alpha =h(y_n ) = h(x_n)$, one obtains: 
 \begin{equation}
 y_m =  \chi_{h,n,m}(y_n) = (h^{-1} \circ \chi_L \circ \chi_{h,n,m}) (h(y_n)) = h^{-1} (\chi_L(x_m))
 \end{equation}
Taking the limit as $m \rightarrow +\infty$ yields that $y_m$ converges toward $h^{-1}(p) =q$. 
 Indeed, for each point $y \in \Lf$, the points $y_n = \mathcal{W}^u(y) \cap \mathcal{L}_n$ (whenever they are defined) converge toward $y$.   
 Hence the sequence of leaves $\mathcal{L}_{n}$ $\mathcal{C}^1$-converge toward $\Lf$.   
\end{proof}

\subsection{Spectral density at the maximal value}

Recall from Definition \ref{def_density} that the density of states measure is the slice of the spectral current by the line $L=\{y=1\}$:
\begin{equation} \label{eq_slice_formula}
\omega_L = S_{\R} \wedge [L]. 
\end{equation}
The farthest point in the support is the point $p=(4,1)$.  
We will compare the previous measure with  the slice $\omega_C = S_{\R} \wedge [C]$ by the invariant curve $C$ defined in the previous section.

For $\alpha \in \{ L,C \}$, the spectral density of the measure $\omega_\alpha$ is by definition: 
\begin{equation} \label{eq_shubin_like}
\eta_\alpha = \limsup_{\epsilon \rightarrow 0^+} \dfrac{\log \omega_\alpha(B(p, \epsilon))}{\log \epsilon}. 
\end{equation}

\begin{thm} \label{thm_spectral_density_unstable} One has: 
\begin{equation}
\eta_C = \lim_{\epsilon \rightarrow 0^+ } \dfrac{\log \omega_C (B(p,\epsilon))}{\log \epsilon} = \dfrac{2}{3}. 
\end{equation}
\end{thm}

\begin{proof}
Let us consider $\eta'_C = \liminf_{\epsilon \rightarrow 0^+ } \dfrac{\log \omega_C ([4-\epsilon, 4])}{\log \epsilon}$. 
Recall that the current $S_{\R}$ satisfies the functional equation: 
\begin{equation*}
(Id - \dfrac{\Fs}{2} ) S_{\R} = \dfrac{1}{12} [\Lc]. 
\end{equation*}
Applying $Id + \Fs/2$ and slicing by $[C]$, one obtains: 
\begin{equation*}
S_{\R} \wedge_{\R} [C] - \dfrac{\Fs^2 S_{\R} \wedge_{\R} [C] }{4} =\dfrac{1}{12} \left  ( [\Lc] +  \dfrac{\Fs [\Lc]}{2}\right ) \wedge [C]
\end{equation*}
Choosing small enough $\epsilon_0> 0$ so that for all $\epsilon < \epsilon_0$ the arc $C \cap B(p,\epsilon)$ does not contain the points $C\cap \Lc$ and $C\cap F^{-1}(\Lc)$, the previous relation yields: 
\begin{equation} \label{eq_pullback_measure}
(F^2)^* \omega_C( B(p, \epsilon)) = 4 \omega_C(B(p,\epsilon)), 
\end{equation} 
for all $\epsilon < \epsilon_0$. 

Let us now describe the measure $(F_X^2)^* \omega_C$. Consider an neighborhood $U$ of the point $C\cap \Lf$ and denote by $\mathcal{F}(B(p,\epsilon))$ the local curves of $\cup_{n \geqslant 0}(F_X^n)^* \Lc$ contained in $U$ which intersect the arc $C\cap B(p,\epsilon)$.
These curves are contained in the support of the current $S_{\R}$. Moreover, we claim that for each leaf $\mathcal{L}(x) \in \mathcal{F}(\mathcal{B}(p,\epsilon)$, $U \cap F_X^{-2}(\mathcal{L}(x))$ also belongs to $\mathcal{F}(B(p,\epsilon))$. 
Moreover, since $C$ is an unstable manifold, for any $x \in C \cap B(p, \epsilon)$, there exists $y \in C\cap B(p,\epsilon)$ such that $F^2(y) = x$. 
Observe also that the line $\Lf$ is a line of fixed points for $F^2$ and the differential  $F^2$ at the point $p_0=\Lf \cap L$ is diagonal of the form: 
\begin{equation*}
DF^2 (p_0) = \left  ( \begin{array}{ll}
1 & 0 \\
0 & 8
\end{array} \right  ). 
\end{equation*}
As a result, for $\epsilon$ small enough, the leaf $\mathcal{L}(y)$ belongs to $\mathcal{F}(\mathcal{B}(p, \epsilon)$ if and only if $F^2(\mathcal{L}(y))$ belongs to $\mathcal{F}(\mathcal{B}(p, 8 u(\epsilon)\epsilon))$, 
 where $u(\epsilon) = 1 + o(\epsilon)$ as $\epsilon \rightarrow 0$. 
 Such point $y $ belongs to the support of $(F^2)^* \omega_C ([4- \epsilon, 4])$, hence 
\begin{equation*} \label{eq_scaling_pullback}
(F^2)^* \omega_C([4 -\epsilon, 4]) = \omega_C([4 - 8 u(\epsilon) \epsilon, 4]).
\end{equation*}
Using the previous relation together with \eqref{eq_pullback_measure} and the definition of $\eta, \eta'$ respectively, we get:  
\begin{equation*}
8^{\eta} u(\epsilon)^{\eta}\epsilon^{\eta} \leqslant 4 \omega_C([4 - \epsilon, 4]) = \omega_C([4 - 8 u(\epsilon) \epsilon, 4])  \leqslant 8^{\eta'} u(\epsilon)^{\eta'} \epsilon^{\eta'} 
   \end{equation*}
Taking the logarithm and dividing by $\log(\epsilon)$ gives: 
\begin{equation} \label{eq_bound_1}
 \dfrac{  \eta' (\log (8u(\epsilon)) }{\log(\epsilon)} + \eta' \leqslant \dfrac{\log(4)}{\log(\epsilon)} + \dfrac{\log \omega_C([4 - \epsilon, 4])}{\log \epsilon} \leqslant \dfrac{\eta \log (8u(\epsilon))}{\log(\epsilon)} + \eta , 
\end{equation}
 for all $\epsilon < \epsilon_0$. 
 Choosing as sequence $\epsilon_n$ such that
$\lim_{n\rightarrow +\infty} \log \omega_C([4 - \epsilon_n, 4])/ \epsilon_n = \eta'$, the lower bound in the previous inequality gives
$-\eta'  {\log(8 u(\epsilon_n))}/{\log (\epsilon_n)} \geqslant -{\log 4}/{\log (\epsilon_n)}
$. In particular, since $u(\epsilon) \rightarrow 1$ as $\epsilon \rightarrow 0$, we get:
\begin{equation} \label{eq_liminf_bound}
\eta' = \liminf_{\epsilon \rightarrow 0^+} \dfrac{\log \omega_C([4 - \epsilon, 4])}{\log (\epsilon)}   \geqslant \dfrac{\log(4)}{\log(8)} = \dfrac{2}{3}. 
\end{equation}
Similarly, using  \eqref{eq_bound_1} to a sequence $\epsilon_n'$ for which $\lim_{n\rightarrow} \log(\omega([4 - \epsilon_n']))/\log(\epsilon_n') = \eta$, we get $- \eta \log(8 C(\epsilon_n'))/\log(\epsilon_n') \leqslant - \log(4)/\log(\epsilon_n')$. This gives:
 
\begin{equation} \label{eq_limsup_bound}
\eta = \limsup_{\epsilon \rightarrow 0^+} \dfrac{\log \omega_C([4 - \epsilon, 4])}{\log \epsilon} \leqslant \dfrac{\log 4}{\log 8} = \dfrac{2}{3}.
\end{equation} 
  The two inequalities \eqref{eq_liminf_bound} and \eqref{eq_limsup_bound} show that $2/3 \leqslant \eta' \leqslant \eta \leqslant 2/3$, hence the limit as $\epsilon$ converges to zero exists and equal $2/3$, as required. 
\end{proof}

\begin{cor} The density $\eta_L$ also satisfies the relation: 
\begin{equation*}
\eta_L = \lim_{\epsilon\rightarrow 0^+} \dfrac{\log \omega_L([4-\epsilon, 4])}{\log \epsilon}.
\end{equation*}
\end{cor}

\begin{proof} 
Consider the holonomy $\chi$ along the leaves of $\mathcal{L}$ from $C$ toward the transversal $L$. We have 
$\omega_C(B(p,\epsilon)) = \omega_L (\chi (C \cap B(p,\epsilon)))$. 
We thus need to establish how the restriction of $\chi$ to $C$ deforms the lengths. More precisely, we show that the holonomy is differentiable at $p$.   
Take $x \in C \cap B(p, \epsilon) $ contained in a leaf $\mathcal{L}_x$. By Proposition \ref{prop_C1_convergence},  the slopes of $\mathcal{L}_x$ converges to $1/2$ as $x$ converges to $p$. 
Recall that $C$ is a $\mathcal{C}^1$ curve, so we can consider $-\tan \theta$ the slope of the curve $C$ at $p$ where $0 < \theta<\pi/2$.   
For any $\epsilon > 0$, there exists $t> 0$ such that if $x \in C \cap B(p, t)$, the slope of $\mathcal{L}_x$ when it exists belongs to $[1/2 - \epsilon, 1/2 + \epsilon]$. Take $x\in C \cap B(p,t)$ at distance $r$ to $p$ on the left hand side of $\Lf$, 
\begin{equation*}
\cos(\theta) + \dfrac{2 \sin \theta}{1 + 2\epsilon}\leqslant \dfrac{ d(\chi(x),p)}{r} \leqslant \cos(\theta) + \dfrac{2 \sin \theta}{1 - 2\epsilon}.
\end{equation*}
Taking the limit as $\epsilon \rightarrow 0$, we obtain that the holonomy map is differentiable at $p$. 
We thus obtain:
\begin{align*}
\dfrac{\log\omega_C(B(p,\epsilon))}{\log \epsilon} & = \dfrac{\log \omega_L(\chi(C\cap B(p,\epsilon)))}{\log \epsilon} \\
&=  \dfrac{\log \omega_L(B(p, \chi'(p) u(\epsilon) \epsilon))}{\log \epsilon} \\
& = \dfrac{\log \omega_L([4 - \chi'(p)u(\epsilon) \epsilon, 4])}{\log (\chi'(p) u(\epsilon)\epsilon)} + O\left ( \dfrac{1}{\log (\epsilon)} \right )
\end{align*}
with $u(\epsilon) = 1 + o(\epsilon)$ as $\epsilon \rightarrow 0$. 
Taking the limit as $\epsilon \rightarrow 0$ in the previous equality, the left hand side converges to $2/3$ by Theorem \ref{thm_spectral_density_unstable}, and this shows that the limit in the right hand side converges to the same value.
\end{proof}

\subsection{Support of the density of states measures}

\begin{thm} For any $h \in \R$, the density of states measures $\omega_h$ are supported on a subset of positive length for the Lebesgue measure. 
\end{thm}

\begin{proof}   By Proposition \ref{prop_sibony_slice_spectral}, the measure $\omega_h$ is equal to $S\wedge [L_h]$ where $S$ is the Sibony current. 
The support of $\omega_h$ is thus the support of $S$ intersected with the line $L_h$. By Corollary \ref{cor_sibony_series}, the support of $S$ is exactly  $\mathcal{L}$. In particular, the support of $\omega_h$ is $\mathcal{L} \cap L_h$, which has positive length by  Theorem \ref{thm_L_positive}. 
\end{proof}

\subsection{Behavior of the density of states measures}

A direct consequence of the main result of Theorem \ref{thm_continuous_slice} yields an error term for the density of states measures. 
Recall from \S~\ref{section_coding} that we have defined an alphabet $A$ and that we have defined a function $\gamma$ from $A^{\mathbb{N}} $ with values in $1$-current supported on vertical curves in $\mathcal{L}$. As described in \S~\ref{section_transfer}, there is a natural measure $\mu$ invariant by the transfer operator induced by the Markov matrix introduced in \S~\ref{section_combinatorics}.  

\begin{thm} For any $h \in \R$, one has for any continuous function $\psi$ on the real line: 
\begin{multline*}
\langle \omega_h , \psi \rangle  = \langle S \wedge [L_h] , \psi \rangle = \langle S_{\R} \wedge_{\R} [L_h], \psi \rangle  = \\ \dfrac{1}{12} \sum_{k \leqslant n } \dfrac{1}{2^k} \langle \Fs^k[\Lc] \wedge_{\R} [L_h], \psi \rangle  + \dfrac{ v_{III_0}}{12(1-(\delta/2))} \dfrac{\delta^n}{2^n} \int_{A^{\mathbb{N}}} \langle \gamma(w) \wedge_{\R} [L_h],\psi \rangle d\mu(w)  + o\left ( \dfrac{\delta^n}{2^n}\right ),
\end{multline*}
as $n\rightarrow +\infty$. 
\end{thm}

\begin{proof}
The first equality follows from Proposition \ref{prop_sibony_slice_spectral} and the others from Theorem \ref{thm_slice_SR}. 
\end{proof}


\section{Appendix A: The iterated monodromy group of the Basilica}
\label{section_appendix_A}

In this section, we recall the construction of the iterated monodromy group of the Basilica group and the relationship between the density of states and the dynamical map $F$ is explained in \S 

Let us consider the basilica map  $f \colon z \mapsto z^2 -1$. The polynomial map $f$ is a
postcritically  since the orbit of the critical point $0$ 
is finite, $ 0\mapsto -1 \mapsto 0$. Below  we will recap the
construction of the iterated monodromy group associated to $f$,
from the lamination viewpoint, 
and of the corresponding sequence of Schreier  graphs converging to the Julia set of $f$.

\subsection
{Basilica group as the  monodromy group on the Riemann surface lamination}
The {\em natural extension}  (or the {\em inverse limit})  of the basilica map  
is an invertible map $\hat f: \NN \ra \NN$ naturally acting on the space of
backward orbits
$$ \hat z = (\dots z_{-2} \mapsto z_{-1} \mapsto z_0) $$ of $f$. 
The natural projetion $\pi: \NN \ra \hat C$, $\hat z\mapsto z_0$,  semi-conjugates $\hat f$
to $f$. Moreover, $\pi$ is a Cantor set fibration over $\C\sm \{0, -1\}$. 
We let $F_z$ be its fiber over $z\not=0,1$. 

The space $\NN$ has the structure of a Riemann surface lamination
whose leaves  $\LL= \LL(\hat z)$  are isomorphic to the complex plane
$\C$ (see \cite {lyubich_minsky}). 
The projection $\pi$ restricted to any $\LL \sm \pi^{-1}  \{0, -1\}$
is a covering map.

Given a path  $\si\subset \C\sm \{0,1\}$  connecting any two points
$a,b\not=0,1$,  there is a {\em holonomy map} $h_\gamma: F_a \ra F_b$
obtained by lifting $\si$ to the leaves of the lamination $\NN$.
This map depends only on the homotopy class of $\si\subset \C\sm
\{0,-1\}$ rel the endpoints. In particular, we obtain an action of the
fundamental group $\pi_1(\C\sm \{0, -1\}, a)$ on $F_a$. The image of
this action is called the {\em iterated monodromy group} ({\em IMG})
of the Basilica, or just the  {\em Basilica group} $G$. 

\subsection{Action on the coding tree}
Take a base point $\alpha\in \C \sm \Jul$ and connect it to its
preimages, $\alpha^1_0$ and $\alpha^1_1$, by paths $l_0$ and $l_1$. 
Call this tree $\TT^0$ 
Then lift the above  path by $f$ to paths $\l_{ij}^1$ labeled so that $l^1_{ij}$ 
is the lift of $l_j$ that begins at $\alpha^1_i$.  We label the endpoints of
the $l^1_{ij}$  as $\alpha^2_{ij}$. Call the corresponding tree
$\Gamma^1$/

Continue this lifting process: Lift the paths $l^1_{jk}$ to paths
$l^2_{ijk}$ beginning at $\alpha^2_{ij}$,
and label their endpoints as $\alpha^3_{ijk}$. Call the corresponding
graph $\TT^2$. And so on.  In this way we obtain a hierarchical 
sequence of binary trees $\TT^n$, naturally related by inclusions
and the $f$-action.

The above labeling of the vertices of the tree matches with the
natural labeling of the dyadic B\"ottcher  boxes, implying that the
trees converge to the Basilica Julia set $\Jul$.
The Basilica group $G$ naturally acts on these trees
by the monodromy respecting both relations.

We let $\TT_i^{n-1}$ be the subtree of $\TT^n$ rooted at $\alpha_i^1$.
(More generally, $\alpha_i^1$ can be replaced with an arbitrary vertex
$\alpha^m_{i_0\dots i_{m}-1}$.)
The corresponding infinite trees will be denoted $\TT$ and $\TT_i$,
respectively. 

In what follows we always select $\alpha$ as the dividing fixed point
for $f$,  $l_0$ as a simple positively oriented loop in
$\C\sm \{-0, -1\}$ based at $\alpha$,
and $l_1$ as a simple arc connecting $\alpha$ to $-\alpha$
and going above the real axis.

\subsection{Self-similarity of the Basilica group}

The group  $G$ is   self-similar. Indeed, let $b, b'$ be the
preimages of $a$, and let $\si. \si'$ be two paths connecting $b,b'$
to $a$. Then the  fibers  $F_b$ and $F_{b'}$ are related to $F_a$ in
two ways:

\ssk \nin (i) $F_a$ is partitioned into two cylinders,
$C= \hat f(F_b$ and $C'= \hat f (F_{b'}) $;

\ssk \nin (ii)  $F_b$ and $F_{b'}$ can be identified with $F_a$
via the holonomy along the paths $\si$ and $\si'$, respectively. 

\ssk By (i), the IMG action on  $F_a$  induces the action on the above cylinders $C$ and $C'$:
each monodromy  map of $G$ either keeps both cylinders invariant or permutes them.
By (ii), each of these maps can be identified with a self map of
$F_a$. This makes the monodromy action {\em self-similar}. 

In terms of the coding trees, this self-similarity  is described as follows:
Take a loop $\gamma$ representing an element of the fundamental group
$\pi_1(\C\sm \{0,-1\}$, lift it to a path $\gamma'$  in $\C\sm \{0,-1\}$
connecting each $\alpha_i$, $i=0,1,$ to some $\alpha_j$ with $j=j(i)$ ,
and consider the concatenation $l_i\circ \gamma' \circ l_j^{-1}$.
This path represents the element of the fundamental group
acting between the subtrees  $\TT^{n-1}_i$ and $\TT^{n-1}_j$   
(where the latter are identified with $\TT^{n-1}$ by means of the holonomy
along $l_i$ and $l_j$, respectively). 
This action is also denoted $\gamma|\, i$. 

Explicitly,  the corresponding recursive  wreath relations are
described as follows.

Let  $a$ be the positively oriented loop in $\C\sm \{0,-1\}$ going around the point $0$,
and $b$ be such a  loop around  $-1$. The lifts of $a$ are  two loops $a_1, a_2$
based at $-1$ and $1$ respectively. The first one represents the same
$\pi_1$-element as  $b$, while the second (after bringing it back to
$\alpha$)  is trivial. Fix $\alpha > 0$ small enough. 
The loop $b$ lifts to an  arc $b_1$ connecting
	$\alpha$ to $-\alpha$ inside the upper half plane $\{\Im z \geq 0\}$ and to an arc $b_2$ connecting $-\alpha$ to
$\alpha$ in the lower half plane $\{\Im z \geq 0\}$.
 

The concatenation $l_0\circ b_2 \circ l_1^{-1}$ represents the sam
$\pi_1$-element as $a$, while the concatenation $l_1\circ b_1\circ l_0^{-1}$
is trivial (see the picture below).

\begin{figure}[h!]
\def\svgwidth{10cm}
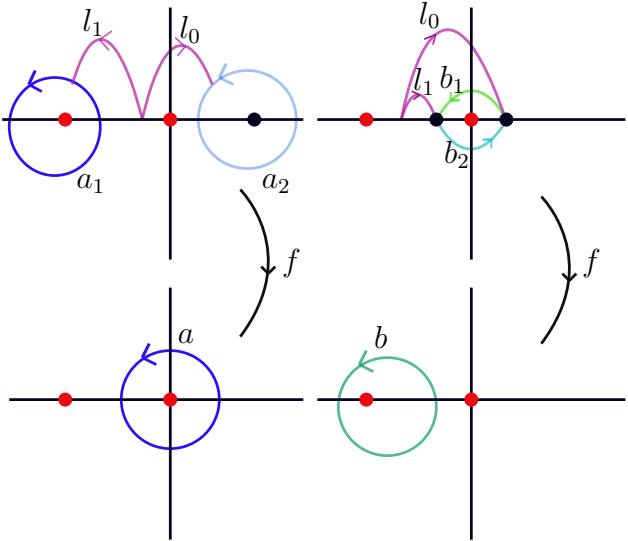
\caption{Figure of the pullbacks by $f$ of the loop $a,b$ around the $0, -1$ (in red)}
\end{figure}

 These observations lead to the following wreath relations
in the Basilica group:

\begin{equation}
a = \left ( \begin{array}{ll}
b & 0 \\
0 & 1 
\end{array} \right ), \quad b = \left ( \begin{array}{ll}
0 & a \\
1 & 0 
\end{array} \right ). 
\end{equation}

%

\subsection{Schreier graphs and some spectral invariants} \label{section_schreier} 

The {\em Schreier graph} on some $G$-orbit  can  identified with the  Julia
set on the corresponding leaf of the lamination (see \cite{cabrera}).
It is approximated by the finite  graphs $\Gamma_n$ with vertices $f^{-n} (a)$
corresponding to the monodromy action on the $n$th  level of the
Cantor hierarchy. The approximating graph of level $n$ is just the
lift of the initial graph $\Gamma_0$ by $f^n$.


\begin{center}
\begin{figure}[h!]
\begin{picture}(400,210)
\letvertex A=(0,110)\letvertex B=(40,110)\letvertex C=(80,110)\letvertex 
D=(100,130)
\letvertex E=(100,90)\letvertex F=(120,110)\letvertex 
G=(140,130)\letvertex H=(140,160)
\letvertex I=(160,110)\letvertex L=(140,90)\letvertex 
M=(140,60)\letvertex N=(170,140)
\letvertex O=(200,150)\letvertex R=(230,140)\letvertex 
S=(240,110)\letvertex T=(230,80)
\letvertex U=(200,70)\letvertex V=(170,80)\letvertex 
P=(200,180)\letvertex Q=(200,210)
\letvertex Z=(200,40)\letvertex J=(200,10)\letvertex 
K=(260,130)\letvertex X=(280,110)
\letvertex W=(260,90)\letvertex g=(260,160)\letvertex 
h=(260,60)\letvertex c=(300,130)
\letvertex Y=(300,90)\letvertex d=(320,110)\letvertex 
e=(360,110)\letvertex f=(400,110)

\tiny{ \put(-7,102){11110} \put(33,102){01110}\put(69,102){00110}
\put(80,128){10110}\put(80,87){10010}\put(101,108){00010}
\put(121,128){01010}\put(121,158){11010} \put(141,108){00000}
\put(121,88){01000}\put(121,58){11000}
\put(152,137){10000}\put(192,143){00100}\put(232,137){10100}
\put(243,108){00001}\put(232,78){10001}\put(192,74){00101}
\put(152,77){10101}\put(182,177){01100} \put(182,207){11100}
\put(182,37){01101}\put(182,7){11101}
\put(263,128){01001}\put(283,108){00011}\put(263,88){01011}
\put(263,157){11001}\put(263,58){11011}
\put(303,128){10011}\put(303,88){$10111=10011^\ast$}\put(316,102){00111}
\put(352,102){01111}\put(390,102){11111}}

\drawvertex(A){$\bullet$}\drawvertex(B){$\bullet$}
\drawvertex(C){$\bullet$}\drawvertex(D){$\bullet$}
\drawvertex(E){$\bullet$}\drawvertex(F){$\bullet$}
\drawvertex(G){$\bullet$}\drawvertex(H){$\bullet$}
\drawvertex(I){$\bullet$}\drawvertex(L){$\bullet$}
\drawvertex(M){$\bullet$}\drawvertex(N){$\bullet$}
\drawvertex(O){$\bullet$}\drawvertex(P){$\bullet$}
\drawvertex(J){$\bullet$}\drawvertex(K){$\bullet$}
\drawvertex(Q){$\bullet$}\drawvertex(R){$\bullet$}
\drawvertex(S){$\bullet$}\drawvertex(T){$\bullet$}
\drawvertex(U){$\bullet$}\drawvertex(V){$\bullet$}
\drawvertex(W){$\bullet$}\drawvertex(X){$\bullet$}
\drawvertex(Y){$\bullet$}\drawvertex(Z){$\bullet$}
\drawvertex(g){$\bullet$}\drawvertex(h){$\bullet$}
\drawvertex(c){$\bullet$}\drawvertex(f){$\bullet$}
\drawvertex(d){$\bullet$}\drawvertex(e){$\bullet$}

\drawundirectedcurvededge(A,B){$b$}\drawundirectedcurvededge(B,A){$b$}
\drawundirectedcurvededge(B,C){$a$}\drawundirectedcurvededge(C,B){$a$}
\drawundirectededge(C,D){$b$} \drawundirectededge(D,F){$b$}
\drawundirectededge(F,E){$b$} \drawundirectededge(E,C){$b$}

\drawundirectededge(F,G){$a$} \drawundirectededge(G,I){$a$}
\drawundirectededge(I,L){$a$} \drawundirectededge(L,F){$a$}

\drawundirectedcurvededge(G,H){$b$}\drawundirectedcurvededge(H,G){$b$}
\drawundirectedcurvededge(L,M){$b$}\drawundirectedcurvededge(M,L){$b$}

\drawundirectededge(I,N){$b$} \drawundirectededge(N,O){$b$}
\drawundirectededge(O,R){$b$} \drawundirectededge(R,S){$b$}
\drawundirectededge(S,T){$b$} \drawundirectededge(T,U){$b$}
\drawundirectededge(U,V){$b$} \drawundirectededge(V,I){$b$}

\drawundirectedcurvededge(O,P){$a$}\drawundirectedcurvededge(P,O){$a$}
\drawundirectedcurvededge(Q,P){$b$}\drawundirectedcurvededge(P,Q){$b$}

\drawundirectedcurvededge(U,Z){$a$}\drawundirectedcurvededge(Z,U){$a$}
\drawundirectedcurvededge(Z,J){$b$}\drawundirectedcurvededge(J,Z){$b$}

\drawundirectededge(S,K){$a$} \drawundirectededge(K,X){$a$}
\drawundirectededge(X,W){$a$} \drawundirectededge(W,S){$a$}
\drawundirectededge(X,c){$b$} \drawundirectededge(c,d){$b$}
\drawundirectededge(d,Y){$b$} \drawundirectededge(Y,X){$b$}

\drawundirectedcurvededge(d,e){$a$}\drawundirectedcurvededge(e,d){$a$}
\drawundirectedcurvededge(e,f){$b$}\drawundirectedcurvededge(f,e){$b$}
\drawundirectedcurvededge(K,g){$b$}\drawundirectedcurvededge(g,K){$b$}
\drawundirectedcurvededge(W,h){$b$}\drawundirectedcurvededge(h,W){$b$}

\drawundirectedloop[l](A){$a$}\drawundirectedloop(D){$a$}\drawundirectedloop[b](E){$a$}
\drawundirectedloop(H){$a$}\drawundirectedloop[b](M){$a$}\drawundirectedloop(Q){$a$}\drawundirectedloop[b](J){$a$}
\drawundirectedloop(g){$a$}\drawundirectedloop[b](h){$a$}\drawundirectedloop(c){$a$}\drawundirectedloop[b](Y){$a$}
\drawundirectedloop[r](f){$a$}\drawundirectedloop[b](V){$a$}\drawundirectedloop(N){$a$}\drawundirectedloop[b](T){$a$}
\drawundirectedloop(R){$a$}
\end{picture}
\caption{Figure of the graph $\Gamma_5$ due to Alfredo Donno}
\end{figure}
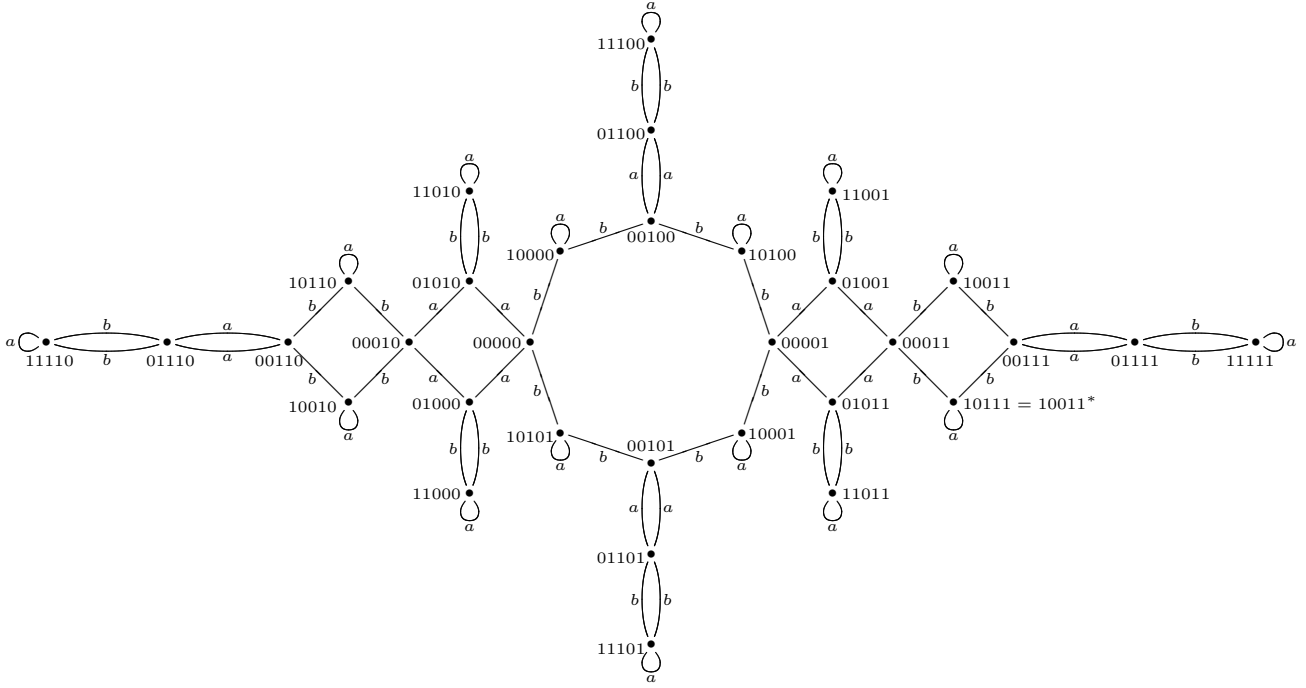
\end{center}

Write by $V_n$ the vertices of the graphs $\Gamma_n$ for each $n$. 
Since the edges of the Scheier graphs $\Gamma_n$ are labeled $a,b$, one can denote by $a_n, b_n$ the operator acting on $L^2(V_n)$ induced by the transition along each edge labeled $a,b$ respectively. Observe that each of the operators $a_n, b_n$ are square matrices of size $2^n$.  

\begin{defi} \label{def_density}The density of states measure $\omega$ of the Basilica group is the limit, when it exists of the counting measures $\omega_n$:
\begin{equation*}
\omega = \lim_{n \rightarrow +\infty }\omega_n = \lim_{n\rightarrow +\infty} \sum_{\lambda \in sp(a_n+ a_n^{-1} + b_n+ b_n^{-1})} \delta_{\lambda}.  
\end{equation*}
\end{defi}

  The existence of the limit is established in \cite{grigorchuk_zuk_99}. 

While the support of $\omega$ make sense of the accumulation points of the eigenvalues of the operators $a_n + a_n^{-1} + b_n + b_n^{-1}$, we define a natural higher dimensional counterpart, which is the spectral current, introduced for other groups in \cite{dang_grigorchuk_lyubich}.
For each $n \geqslant 1$, consider  the algebraic curve $$S_n = \{  [\mu_1:\mu_2:\lambda] \in \mathbb{P}^2(\C) | \det(\mu_1(a_n + a_n^{-1}) + \mu_2(b_n + b_n^{-1}) - \lambda) = 0\}. $$ 
One writes by $[S_n]$ the $(1,1)$-current of integration along $S_n$. 

\begin{defi} The spectral current $S$ is the limit in the sense of currents on $\mathbb{P}^2(\C)$, when it exists of the sequence of currents:
\begin{equation*}
S = \lim_{n\rightarrow +\infty} [S_n] 
\end{equation*} 
\end{defi}

\subsection{The Schur-Renormalization map $F$ and the density of states}

The renormalization map $F : \C^2 \dashrightarrow \C^2$ associated to the Basilica group is  given by:
\begin{equation*}
F \colon (\lambda,\mu) \mapsto \left (-2 - \dfrac{\lambda(2 - \lambda)}{\mu^2} , \dfrac{\lambda -2 }{\mu^2} \right ).
\end{equation*}
$F$ in homogeneized coordinates is of the form: 
\begin{equation}
\tilde{F} (\lambda, \mu , w) = \left ( - 2 \mu^2 - 2\lambda w + \lambda^2 , (\lambda - 2 w)w , \mu^2  \right ). 
\end{equation}

Let us denote by $a_n, b_n$ the action of the generators $a,b$ on $L^2(\Gamma_n)$. 
Consider for all $\mu \in \C$, the pencil $$\Delta_{\mu,n} = a_n + a_n^{-1} + \mu (b_n + b_n^{-1}).$$ 
We also define the characteristic polynomials: 
\begin{equation*}
P_{n}(\lambda, \mu) = \det(\Delta_{\mu,n} - \lambda Id). 
\end{equation*}
The polynomials $P_n$ are not homogeneous in $\mu$ and $\lambda$, we can also set $$\tilde{P}_n (\lambda, \mu, w) = w^{2^n} P_{n}(\lambda/w , \mu/w).$$
The recursion formula for the action of $\Delta_{\mu,n}$ between the spaces $L^2(\Gamma_n)$ was carried out in \cite[Lemma 4.2]{grigorchuk_zuk_spectral}, we reproduce the same argument here: 
\begin{lem} \label{lem_recursion} One has for all $n \geqslant 0$: 
\begin{equation} \label{eq_recursion}
P_{n+1}(\lambda, \mu) = \mu^{2^{n+1}} P_n (F(\lambda, \mu)),
\end{equation}
or equivalently: 
\begin{equation*}
\tilde{P}_{n+1}(\lambda, \mu, w)= \tilde{P}_n (\tilde{F}(\lambda, \mu , w)) 
\end{equation*}
where $P_0 = 2 + 2 \mu  - \lambda$ and its zero locus is the line $\Lf$. 
\end{lem}

\begin{proof}
The recursion is of the form: 
\begin{equation}
M_{n+1} = \left ( \begin{array}{ll}
 b_n + b_n^{-1} - \lambda  & \mu + \mu a_n^{-1} \\
 \mu a_n + \mu & 1- \lambda
\end{array} \right ). 
\end{equation}
We calculate the Schur complement with respect to the first corner, it is given by: 
\begin{align*}
S(M_{n+1}) &= b_n + b_n^{-1} - \lambda - \dfrac{\mu^2}{2 - \lambda} \left ( 2 + a_n + a_n^{-1} \right ) \\
  & = \dfrac{\mu^2}{\lambda-2} \left ( a_n+ a_n^{-1} + \dfrac{\lambda-2}{\mu^2}(b_n+ b_n^{-1})  - \left ( \dfrac{\lambda (\lambda -2)}{\mu^2} -2 \right ) \right )\\
   & = \dfrac{\mu^2}{\lambda -2} M_n(F(\lambda,\mu)). 
\end{align*}
Using the relation $\det M_{n+1} = \det S(M_{n+1}) \det (2  - \lambda)I_{2^n}$, we get 
\begin{align*}
\det M_{n+1}(\lambda,\mu) &= (2 - \lambda)^{2^n} \left (\dfrac{\mu^2}{\lambda-2} \right )^{2^n} \det M_n (F(\lambda,\mu)) \\
 & = \mu^{2^{n+1}} \det M_n \circ F, 
\end{align*}
as required. One obtains the second relation using the homogeneity of the determinant. 
\end{proof}

\section{Appendix B: Local expansion of algebraic curves}

In this appendix, we detail the power series expansion of certain explicit curves. 

\begin{lem} \label{lem_expansion_explicit_1} The curve of equation $ x^2 -2 x -4y^2 =0$ admits a power series expansion at zero of the form: 
\begin{equation*}
x = -2 y^2 + \sum_{n\geqslant 2} (-1)^n 2^{n-1} y^{2n}.
\end{equation*}
In particular, its radius of convergence is $1/2$. 
\end{lem}

\begin{proof}
Observe first that we must have $x = -2 y^2 + O(y^3)$. 
Setting $x = -2 y^2 + \sum_{k\geqslant 3} a_k y^k$, the equation $x (-2 + x) = 4 y^2$ yields the conditions: 
\begin{equation*}
\left \lbrace \begin{array}{ll}
a_3 = 0 \\ 
a_4 = 2 \\
a_k = -2 a_{k-2} \text{ if } k\geqslant 5.
\end{array} \right .
\end{equation*}
This yields $a_{2n+1}=0$ for all $n$ and $a_{2n} = 2 (-2)^{n-2}$ for all $n\geqslant 2$ and the lemma is proved. 
\end{proof}

\begin{lem} \label{lem_expansion_explicit_2} The pullback by $F^2$ of the vertical line through $(0,0)$ is an algebraic curve through the point $x_2= \epsilon \sqrt{2}$ with $\epsilon = \pm 1$ which is locally of the form: 
\begin{equation}
x- x_2 =  - \dfrac{3(x_2 - 1)}{3 x_2 -4} y^2 + \sum_{n\geqslant 2} \dfrac{P_n(x_2)}{Q_n(x_2)} y^{2n},
\end{equation}
where $P_n, Q_n$ are polynomials in one variables. 
\end{lem}

\begin{proof}
The explicit expression \eqref{eq_G} gives $x \circ G = \frac{x^4-4 x^3-6 x^2 y^2+2 x^2+12 x y^2+8 x+8 y^4-8}{(x-2)^2}$. 
In particular, near $x_2 = \epsilon \sqrt{2}$ with $\epsilon = \pm 1$, the local equation is determined by the zero locus of 
$x \circ G( x_2 + x, y) =0$ which gives:  
\begin{multline*}
 \dfrac{1}{(-2 + \sqrt{2} \epsilon + x)^2} \left ( 4 \sqrt{2} \epsilon x^3-12 \sqrt{2} \epsilon x^2-12 \sqrt{2} \epsilon x y^2+12 \sqrt{2} \epsilon x \right . \\ \left . +12 \sqrt{2} \epsilon y^2+x^4-4 x^3-6 x^2 y^2+14 x^2+12 x y^2-16 x+8 y^4-12 y^2 \right )  = 0
\end{multline*}
In particular, the numerator in the above expression admits a power expansion of the form $ (12 \sqrt{2} \epsilon - 16) x + 12(\epsilon \sqrt{2} -1) y^2 + h.o.t $ and this shows that the local equations are of the form $x = - \dfrac{12 (x_2 -1)}{12 x_2 - 16 } y^2 + O(y^3)$.
Moreover, observe that all the powers in $y$ are even in the implicit equation. This shows that $x$ is an implicit function of $y^2$ which is locally analytic. Moreover, as the implicit function has coefficient that are polynomial in $x_2$, we deduce that the power series in $y^2$ has coefficient that are rational function in $x_2$, as required.
\end{proof}

\begin{lem} Fix a neighborhood $U$ of $\bIz$, there exists a constant $C>0$ such that for any point $q=(x_0,0) \in \bIz$ and for any curve $\gamma$ obtained as the pullback by $F^2$ of a line of slope $\pm 1/2$ through $F^2(q)$, then for any point $(x,y)\in \gamma \cap U$, one has:
\begin{equation}
C^{-1} y^2 \leqslant |x-x_0| \leqslant Cy^2.
\end{equation}   
\end{lem}

\bibliographystyle{amsalpha}
\bibliography{references}

\end{document}